\def\version{revised May 23, 2008}

\documentclass[11pt]{article}
\usepackage{amsmath,enumerate, amsfonts, amssymb, amsthm, graphicx, eepic, bbm}

\def\UseSection{
        \numberwithin{equation}{section}
    \theoremstyle{plain}
        \newtheorem{theorem}    {Theorem}[section]
        \DefineTheorems 
}

\def\DefineTheorems{
    
    \newtheorem{lemma}      [theorem] {Lemma}
    
    \newtheorem{prop}       [theorem] {Proposition}
    
    \newtheorem{cor}        [theorem] {Corollary}
    \newtheorem{ass}        [theorem] {Assumption}
    \newtheorem{claim}      [theorem] {Claim}
    \theoremstyle{definition}
    \newtheorem{defn}       [theorem] {Definition}

    \theoremstyle{definition}

}

\newcommand{\bt}   {\begin{theorem}}
\newcommand{\et}   {\end  {theorem}}
\newcommand{\bl}   {\begin{lemma}}
\newcommand{\el}   {\end  {lemma}}
\newcommand{\bp}   {\begin{prop}}
\newcommand{\ep}   {\end  {prop}}
\newcommand{\bc}   {\begin{cor}}
\newcommand{\ec}   {\end  {cor}}
\newcommand{\bd}   {\begin{defn}}
\newcommand{\ed}   {\end  {defn}}

\newcommand{\ba}   {\begin{array}}
\newcommand{\ea}   {\end  {array}}
\newcommand{\be}   {\begin{enumerate}}
\newcommand{\ee}   {\end  {enumerate}}
\newcommand{\bi}   {\begin{itemize}}
\newcommand{\ei}   {\end  {itemize}}

\def\eq#1\en{\begin{equation}#1\end{equation}}
\def\eqsplit#1\ensplit{
    \begin{equation}\begin{split}#1\end{split}\end{equation}
    }
\def\eqalign#1\enalign{
    \begin{align}#1\end{align}
    }
\def\eqmul#1\enmul{
    \begin{multline}#1\end{multline}
    }
\newcommand{\eqarrstar} {\begin{eqnarray*}}
\newcommand{\enarrstar} {\end{eqnarray*}}
\newcommand{\eqarray}   {\begin{eqnarray}}
\newcommand{\enarray}   {\end{eqnarray}}
\newcommand{\nnb}   {\nonumber \\}

\newcommand{\lbeq}[1]  {\label{e:#1}}

\makeatletter
\newcommand{\labelcounter}[2]{{%
    \stepcounter{#1}
    \protected@write\@auxout{}%
    {\string\newlabel{#2}{{\csname the#1\endcsname}{\thepage}}}%
    {\ref{#2}}
    }}
\makeatother
\newcommand{\sss}   { \scriptscriptstyle }

\newcommand{\Ebold} {{\mathbb E}}

\newcommand{\Rbold} {{\mathbb R}}

\newcommand{\Zbold} {{\mathbb Z}}

\newcommand{\Ccal}   {\mathcal{C}}

\newcommand{\Hcal}   {\mathcal{H}}

\newcommand{\Wcal}   {\mathcal{W}}

\newcommand{\Fhat} {{\hat{F} }}

\newcommand{\Nhat} {{\hat{N} }}

\newcommand{\Rd}    {{ {\Rbold}^d}}
\newcommand{\Zd}    {{ {\Zbold}^d }}

\newcommand{\spose}[1] {{\hbox to 0pt{#1\hss}} }
\newcommand{\ltapprox} {\mathrel{\spose{\lower 3pt\hbox{$\mathchar"218$}}
 \raise 2.0pt\hbox{$\mathchar"13C$}}}
\newcommand{\gtapprox} {\mathrel{\spose{\lower 3pt\hbox{$\mathchar"218$}}
 \raise 2.0pt\hbox{$\mathchar"13E$}}}

\newcommand{\nin}  {{ \not\in }}

\UseSection  
\newcommand{\C}     {\mathbb{C}}
\newcommand{\R}     {\mathbb{R}}

\renewcommand{\P}   {\mathbb{P}}

\newcommand{\E}     {\mathbb{E}}

\newcommand\1{\mathbbm{1}}

\newcommand{\citeAkira}[1]{\cite[#1]{Sakai07}}  

\renewcommand{\d}{{\rm d}}
\newcommand{\ua}{\nearrow}
\newcommand{\da}{\searrow}

\newcommand{\eps}{\varepsilon}

\newcommand{\ssup}[1] {{\scriptscriptstyle{({#1}})}}

\newcommand{\sN}{\ssup{N}}
\newcommand{\sM}{{\scriptscriptstyle{M}}}

\newcommand{\Gk}{\hat{G}_z(k)}
\newcommand{\GN}{\hat{G}_z(0)}
\newcommand{\Dk}{\hat{D}(k)}
\newcommand{\dk}{\frac{\operatorname{d}\!k}{(2\pi)^d}}
\newcommand{\e}{\operatorname{e}}
\newcommand{\PsiK}{\hat{\Psi}_z(k)}
\newcommand{\PhiK}{\hat{\Phi}_z(k)}
\newcommand{\PsiN}{\hat{\Psi}_z(0)}
\newcommand{\PhiN}{\hat{\Phi}_z(0)}
\newcommand{\Td}{{\left[-\pi,\pi\right)^d}}

\newcommand{\ClK}{\hat{C}_{\lambda_z}(k)}
\newcommand{\Ci}{\hat{C}_{\lambda_z}(k)^{-1}} 
\newcommand{\ClN}{\hat{C}_{\lambda_z}(0)}
\newcommand{\Cl}{\hat{C}_{\lambda_z}}
\newcommand{\G}{\hat{G}}
\newcommand{\tG}{\tilde{G}}

\newcommand{\const}{{\operatorname {const}\,}}

\newcommand{\Ul}{U_{\lambda_z}(k,l)}

\renewcommand{\phi}{\varphi}
\newcommand{\fatphi}{\boldsymbol{\varphi}}

\newcommand{\gs}{\gamma_{\rm \sss S}}
\newcommand{\es}{\eta_{\rm \sss S}}
\newcommand{\gp}{\gamma_{\rm \sss P}}
\renewcommand{\bp}{\beta_{\rm \sss P}}
\newcommand{\dep}{\delta_{\rm \sss P}}
\newcommand{\hatdep}{\hat\delta_{\rm \sss P}}
\renewcommand{\ep}{\eta_{\rm \sss P}}
\newcommand{\gi}{\gamma_{\rm \sss I}}
\renewcommand{\bi}{\beta_{\rm \sss I}}
\newcommand{\di}{\delta_{\rm \sss I}}
\newcommand{\etai}{\eta_{\rm \sss I}}

\newcommand{\torus}{\mathbb{T}}
\newcommand{\GrR}{G_{z,{\torus_r}}^{\sss (R)}}
\newcommand{\chiRR}{\chi_{\torus_r}^{\sss (R)}}
\newcommand{\PrR}{\mathbb{P}_{z,{\torus_r}}^{\sss (R)}}
\newcommand{\ErR}{\mathbb{E}_{z,{\torus_r}}^{\sss (R)}}

\begin{document}

\thispagestyle{empty}
\vspace{1cm}
\centerline{\LARGE \bf Mean-field behavior for long- and finite range}
\vspace{0.2cm}
\centerline{\LARGE \bf Ising model, percolation and self-avoiding walk}
\vspace{1cm}

\centerline {{\sc Markus Heydenreich$^1$}, {\sc Remco van der Hofstad$^1$} and {\sc Akira Sakai$^2$}}
\vspace{.6cm}

\centerline{\em $^1$Eindhoven University of Technology,}
\centerline{\em Department of Mathematics and Computer Science,}
\centerline{\em P.O.~Box 513, 5600~MB Eindhoven, The~Netherlands}
\centerline{\tt m.o.heydenreich@tue.nl, r.w.v.d.hofstad@tue.nl}
\vspace{.3cm}

\centerline{\em $^2$Hokkaido University,}
\centerline{\em Creative Research Initiative ``Sousei'',}
\centerline{\em North 21, West 10, Kita-ku}
\centerline{\em Sapporo 001-0021, Japan}
\centerline{\tt sakai@cris.hokudai.ac.jp}
\vspace{1cm}

\centerline{\small(\version)}
\vspace{.6cm}

\begin{quote}
  {\small {\bf Abstract:}}
We consider self-avoiding walk, percolation and the Ising model  with long and finite range.
By means of the lace expansion we prove mean-field behavior for these models if $d>2(\alpha\wedge2)$ for self-avoiding walk and the Ising model, and $d>3(\alpha\wedge2)$ for percolation,
where $d$ denotes the dimension and $\alpha$ the power-law decay exponent of the coupling function.
We provide a simplified analysis of the lace expansion based on the trigonometric approach in Borgs et al.\ \cite{BorgsChayeHofstSladeSpenc05b}.
\end{quote}

\vspace{0.5cm}



\section{Introduction}

\subsection{Motivation and overview}
Since its invention in 1985 \cite{BrydgSpenc85}, the \emph{lace expansion} has become a powerful tool for proving mean-field behavior in various spatial stochastic systems, such as the self-avoiding walk, percolation, oriented percolation, the contact process, lattice trees and -animals, and the Ising model.
This paper provides a generalized lace expansion approach that holds for self-avoiding walk, percolation and the Ising model.
We consider the classical nearest-neighbor model as well as various spread-out cases.
Of particular interest are those spread-out models where the underlying step distribution has infinite variance, so-called \emph{long-range} models.
We show that a sufficiently long range can reduce the upper critical dimension,
above which the system shows mean-field behavior.

We shall not perform the complete lace expansion here,
but rather use bounds on the lace expansion coefficients proved elsewhere.
Nevertheless, we give an analysis of the lace expansion inspired by \cite{BorgsChayeHofstSladeSpenc05b},
which is simplified compared to previous work, and generalized so that it deals with long-range models.

Using this generalized framework, we do the analysis of the lace expansion in such a way that it holds for \emph{any} model provided that the expansion has a specific form and certain bounds on the lace expansion coefficients are satisfied (see Section \ref{sectFramework}).
These bounds are proved to follow from a related random walk condition, which is relatively simple to verify.

\subsection{The model}
We study self-avoiding walk, percolation and the Ising model on the hypercubic lattice $\Zd$.
We consider $\Zd$ as a complete graph, i.e., the graph with vertex set $\Zd$ and corresponding edge set $\Zd\times\Zd$. We will refer to the edges as \emph{bonds} and to the vertices as \emph{sites}.
We assign each (undirected) bond $\{x,y\}$ a weight $D(x-y)$, where $D$ is a probability distribution specified in Section \ref{sectPropD} below. If $D(x-y)=0$, then we can omit the bond $\{x,y\}$.

Our analysis is based on Fourier analysis.
Unless specified otherwise, $k$ will always denote an arbitrary element from the Fourier dual of the discrete lattice, which is the torus $\Td$.
The Fourier transform of a summable function $f\colon\Zd\to\C$ is defined by
$    \hat f(k)=\sum_{x\in\Zd}f(x)\,\e^{ik\cdot x}$.

\subsubsection{The step distribution $D$: 3 versions}\label{sectPropD}
Let $D$ denote a probability distribution on $\Zd$ that is symmetric under reflections in coordinate hyperplanes and rotations by $\pi/2$.
We refer to $D$ as a \emph{step} distribution, having in mind a random walker taking independent steps distributed according to $D$.
Without loss of generality we henceforth assume that there is no mass at the origin, i.e.\ $D(0)=0$.

In this paper, we consider three different versions of $D$.
While we explicitly state our main results for these versions,
they actually hold more generally under a random walk condition formulated in Assumption \ref{assumptionBeta} below.
The first version is the \emph{nearest-neighbor model}, where $D$ is the uniform distribution on the nearest neighbors, i.e.,
\begin{equation}\label{eqDefNNmodel}
D(x)=\frac{1}{2d}\1_{\{|x|=1\}},\qquad x\in\Zd.
\end{equation}
Here, and throughout the paper, we denote by $|\cdot|$ the Euclidian norm on $\Zd$
and $\1_E$ represents the indicator function of the event $E$.
This nearest-neighbor version of $D$ corresponds to the classical model for the study of self-avoiding walk, percolation, and the Ising model, see e.g.\ \cite{FernaFrohlSokal92,Grimm99,MadraSlade93}.

We further consider two versions of \emph{spread-out} models. They involve some spread-out parameter $L$, which is typically chosen large. In order to stress the $L$-dependence of $D$ we will write $D_{\sss L}$ in the definitions, but later omit the subscript. In the \emph{finite-variance spread-out} model we require $D_{\sss L}$ to satisfy the following conditions\footnote{These conditions coincide with Assumption D in \cite{HofstSlade02}.}:
\begin{enumerate}
    \item[(D1)] There is an $\eps>0$ such that $$\sum_{x\in\Zd}|x|^{2+\eps}D_{\sss L}(x)<\infty.$$
    \item[(D2)] There is a constant $C$ such that, for all $L\ge 1$, $$\|D_{\sss L}\|_\infty\le CL^{-d}.$$
    \item[(D3)] There exist constants $c_1$, $c_2>0$ such that
        \begin{eqnarray}
              \label{eqPropD1}
              1-\hat{D}_{\sss L}(k)&\ge\hspace{0.2cm} c_1L^2|k|^2 \qquad&\text{if $\|k\|_\infty\le L^{-1}$,}\\
              \label{eqPropD2}
              1-\hat{D}_{\sss L}(k)&>\hspace{0.2cm} c_2\qquad\qquad&\text{if $\|k\|_\infty\ge L^{-1}$,}\\
              \label{eqPropD3}
              1-\hat{D}_{\sss L}(k)&<\hspace{0.2cm} 2-c_2,\qquad &k\in\Td.
        \end{eqnarray}
\end{enumerate}

\paragraph{Example.} Let $h$ be a non-negative bounded function on $\Rd$ which is almost everywhere continuous, and symmetric under the lattice symmetries of reflection in coordinate hyperplanes and rotations by ninety degrees. Assume that there is an integrable function $H$ on $\Rd$ with $H(te)$ non-increasing in $t\ge0$ for every unit vector $e\in\Rd$, such that $h(x)\le H(x)$ for all $x\in\Rd$. Assume further that the $(2+\eps)$-th moment of $h$ exists for some $\eps>0$. The monotonicity and integrability hypotheses on $H$ imply that $\sum_xh(x/L)<\infty$ for all $L$, with $x/L=(x_1/L,\dots,x_d/L)$. Then
\begin{equation}\label{eqDefDwithH}
    D_{\sss L}(x)=\frac{h(x/L)}{\sum_{y\in\Zd}h(y/L)},\qquad x\in\Zd,
\end{equation}
obeys the conditions (D1)--(D3), whenever $L$ is large enough (cf.\ \cite[Appendix A]{HofstSlade02}).
For $h(x)=\1_{\{0<\|x\|_\infty\le 1\}}$ we obtain the \emph{uniform spread-out model} with
\begin{equation}\label{eqDefUSOmodel}
    D_{\sss L}(x)=\frac{1}{(2L+1)^d-1}\1_{\{0<\|x\|_\infty\le L\}},\qquad x\in\Zd.
\end{equation}

In the \emph{spread-out power-law model} we replace assumptions (D1) and (D3) by the condition that there exists an $\alpha>0$ such that
\begin{enumerate}
    \item[(D$1'$)] all $\eps>0$ satisfy $$\sum_{x\in\Zd}|x|^{\alpha-\eps}D_{\sss L}(x)<\infty;$$
    \item[(D$3'$)] there exist constants $c_1,c_2>0$ such that
        \begin{eqnarray}\label{eqPropDPowerLaw1}
              1-\hat{D}_{\sss L}(k)&\ge\hspace{0.2cm} c_1L^\alpha|k|^\alpha
              \qquad&\text{if $\|k\|_\infty\le L^{-1}$,}\\
              \label{eqPropDPowerLaw2}
              1-\hat{D}_{\sss L}(k)&>\hspace{0.2cm}c_2\qquad\qquad&\text{if $\|k\|_\infty\ge L^{-1}$,}\\
              \label{eqPropDPowerLaw3}
              1-\hat{D}_{\sss L}(k)&<\hspace{0.2cm}2-c_2,\qquad &k\in\Td.
        \end{eqnarray}
\end{enumerate}
The condition (D2)=(D$2'$) remains unchanged.

As an example, let $D_{\sss L}$ be of the form (\ref{eqDefDwithH}), but instead of the existence of the $(2+\eps)$-th moment of $h$, require $h$ to decay as $|x|^{-d-\alpha}$ as $|x|\to\infty$. In particular, there exist positive constants $c_h$ and $l_h$ such that
\begin{equation}\label{eqDefPowerLawDecay}
    h(x)\ge c_h|x|^{-d-\alpha},\qquad\text{whenever $|x|\ge l_h$.}
\end{equation}
In this setting,
the $\kappa^\text{th}$ moment $\sum_{x\in\Zd}|x|^\kappa D_{\sss L}(x)$
does not exist if $\kappa\ge\alpha$, but exists and equals $O(L^\alpha)$ if $\kappa<\alpha$.
Take e.g.\
\begin{equation}\label{eqDefHforPowerLaw}
    h(x)=(|x|\vee1)^{-d-\alpha},
\end{equation}
so that $D_{\sss L}$ has the form
\begin{equation}\label{eqDefDPowerLaw}
D_{\sss L}(x)=\frac{\left(|x/L|\vee1\right)^{-d-\alpha}}{\sum_{y\in\Zd}\left(|y/L|\vee1\right)^{-d-\alpha}},\qquad x\in\Zd.
\end{equation}
Chen and Sakai \cite[Prop.\ 1.1]{ChenSakai05} showed that, analogously to the finite-variance spread-out model, the spread-out power-law model (\ref{eqDefDPowerLaw}) satisfies conditions (D$1'$)--(D$3'$).

Note that the spread-out power-law model with parameter $\alpha>2$ satisfies the finite variance condition (D1), and hence is covered in the finite variance case.
For simplicity we further write $\alpha\wedge2$ indicating the minimum of $\alpha$ and $2$ in the spread-out power-law case, and 2 in the nearest-neighbor case or in the finite-variance spread-out case.

For the finite-variance spread-out model and the spread-out power-law model we require that the support of $D$ contains the nearest neighbors of $0$, see the discussion below (\ref{eqDefZc}).

We next introduce the models that we shall consider, i.e., self-avoiding walk, percolation and the Ising model.

\subsubsection{Self-avoiding walk}\label{sectSAW}
For every lattice site $x\in\Zd$, we denote by
\begin{equation}\label{eqDefWn}
    \Wcal_n(x)=\{(w_0,\dots,w_{n}) \mid w_0=0,\, w_n=x,\, w_i\in\Zd, 1\le i\le n-1\}
\end{equation}
the set of $n$-step walks from the origin $0$ to $x$. We call such a walk $w\in\Wcal_n(x)$ \emph{self-avoiding} if $w_i\neq w_j$ for $i\neq j$ with $i,j\in\{0,\dots,n\}$. We define $c_0(x)=\delta_{0,x}$ and, for $n\ge1$,
\begin{equation}\label{eqDefCn}
    c_n(x):=\sum_{w\in\Wcal_n(x)}\prod_{i=1}^nD(w_i-w_{i-1})\,\1_{\text{\{$w$ is self-avoiding\}}}.
\end{equation}
where $D$ is as in Section \ref{sectPropD}.

\subsubsection{Percolation}\label{sectPercolation}
In percolation we consider the set of \emph{bonds}, which are unordered pairs of lattice sites. We set each bond $\{x,y\}\in\Zd\times\Zd$ \emph{occupied}, independently of all other bonds, with probability $zD(y-x)$ and \emph{vacant} otherwise.
Thus for the nearest-neighbor model, each nearest-neighbor bond is occupied with probability $z/(2d)$.
The corresponding product measure is denoted by $\P_z$ with corresponding expectation $\E_z$.
We require $z\in[0,\|D\|_\infty^{-1}]$ to ensure that $zD(x-y)\le1$. We write $\{x\leftrightarrow y\}$ for the event that there exists a path of occupied bonds from $x$ to $y$.
When the event $\{x\leftrightarrow y\}$ occurs we call the vertices $x$ and $y$ \emph{connected}. For $x\in\Zd$, the set $\Ccal(x):=\{y\in\Zd\mid y\leftrightarrow x\}$ of connected vertices is called the \emph{cluster} of $x$. It is the size and geometry of these clusters that we are interested in. Due to the shift invariance of the model, we can restrict attention to the cluster at the origin $\Ccal:=\Ccal(0)$.

For $z$ small, $\Ccal$ is $\P_z$-a.s.\ finite, whereas for $d\ge2$ and large $z$, the probability that the size of the cluster $\Ccal$ is infinite,
\begin{equation}\label{eqDefTheta}
    \theta(z):=\P_z(|\Ccal|=\infty),
\end{equation}
is strictly greater than zero.
Since $z\mapsto\theta(z)$ is non-decreasing, there exists some critical value $z_c$ where this probability turns positive (see e.g.\ \cite{Grimm99}).

\subsubsection{Ising model}\label{sectIsing}
For the Ising model we consider the space $\{-1,1\}^{\Zd}$ of spin configurations on the hypercubic lattice,
with a probability distribution thereon.
For a formal definition, we consider a finite subset $\Lambda\subset\Zd$,
and for every spin configuration $\fatphi=\{\phi_x|x\in\Lambda\}\in\{-1,1\}^{\Lambda}$
the energy given by the Hamiltonian
\begin{equation}\label{eqDefHcalN}
    \Hcal_\Lambda(\fatphi)=-\!\!\!\!\sum_{\{x,y\}\in\Lambda\times\Lambda}J(y-x)\,\phi_x\,\phi_y,
\end{equation}
where $J$ and $D$ are related via the identity
\begin{equation}\label{eqDefJ}
    D(x)=\frac{\tanh(zJ(x))}{\sum_{y\in\Zd}\tanh(zJ(y))},
\end{equation}
and $z$ is the inverse temperature.
For example, in the nearest-neighbor case, $D=J$.
For the Ising model, $J$ is known as the \emph{spin-spin coupling}.
If $J\ge0$ (and hence $D\ge0$, as in the cases we consider) then the model is called \emph{ferromagnetic}.

\subsubsection{Two-point function and susceptibility}\label{sectTwoPointFunction}
We study self-avoiding walk, percolation and the Ising model in a unified way. For this, we need to introduce some notation.
We consider the function $G_z(x)$ for $x\in\Zd$ with
\begin{equation}\label{eqDefGsaw}
G_z(x)=\sum_{n=0}^\infty  c_n(x)\,z^n
\end{equation}
being the Green's function for self-avoiding walk, while for percolation
\begin{equation}\label{eqDefGperc}
G_z(x)=\P_z(0\leftrightarrow x)
\end{equation}
being the probability of the event that there is a path consisting of occupied edges from $0$ to $x$.
For the Ising model, we consider the spin correlation $G_z$ as the \emph{thermodynamic limit}
\begin{equation}\label{eqDefGising}
    G_z(x)=\lim_{\Lambda\ua\Zd}
    \frac{\sum_{\fatphi\in\{-1,1\}^{\Lambda}}\phi_0\,\phi_x\exp(-z\Hcal_\Lambda(\fatphi))}
    {\sum_{\fatphi\in\{-1,1\}^{\Lambda}}\exp(-z\Hcal_\Lambda(\fatphi))}.
\end{equation}
Here the limit is taken over any non-decreasing sequence of $\Lambda$'s converging to $\Zd$.
This limit exists and is independent from the chosen sequence of $\Lambda$'s due to Griffiths' second inequality \cite{Griff67a}.
We will refer to $G_z$ as the \emph{two-point function}.
This is inspired by the fact that $G_z(x)$ describes features of the models depending on the two points $0$ and $x$.

We further introduce the \emph{susceptibility} as
\begin{equation}\label{eqDefChi}
\chi(z):=\sum_{x\in\Zd}G_z(x).
\end{equation}
For percolation, the susceptibility is the expected cluster size $\chi(z)=\Ebold_z|\Ccal|$.

We define $z_c$, the critical value of $z$, as
\begin{equation}\label{eqDefZc}
    z_c:=\sup\,\left\{z\,|\,\chi(z)<\infty\right\}.
\end{equation}
For self-avoiding walk, $z_c$ is the convergence radius of the power series (\ref{eqDefGsaw}).
For percolation, $z_c$ is characterized by the explosion of the expected cluster size.
Menshikov \cite{Mensh86}, as well as Aizenman and Barsky \cite{AizenBarsk87}, showed that this characterization coincides with the critical value described in Section \ref{sectPercolation}.

For the spread-out models, we require that the support of $D$ contains the nearest neighbors of $0$.
In percolation and the Ising model, this enables a Peierls type argument showing that that a (finite) critical threshold $z_c\in(0,\infty)$ exists, where the susceptibility $\chi(z)$ diverges as $z\nearrow z_c$. This is exemplified in \cite[Sect.\ 2.1]{FernaFrohlSokal92} for the Ising model, and \cite[Sect.\ 1.4]{Grimm99} for percolation.

For the Ising model, we define the \emph{magnetization} $M$ to be
\begin{equation}\label{egDefMag}
    M(z,h)=\lim_{\Lambda\ua\Zd}
    \frac{\sum_{\fatphi\in\{-1,1\}^{\Lambda}}\,\phi_0\exp\{-z\Hcal_\Lambda(\fatphi)+h\sum_{y\in\Lambda}\phi_y\}}
    {\sum_{\fatphi\in\{-1,1\}^{\Lambda}}\exp\{-z\Hcal_\Lambda(\fatphi)+h\sum_{y\in\Lambda}\phi_y\}},
\end{equation}
and write $M(z,0^+)$ for the limit $\lim_{h\searrow0}M(z,h)$.
The magnetization gives rise to another characterization of $z_c$, namely
$z_c=\inf\{z\,|\,M(z,0^+)>0\}$.
As proved by Aizenman, Barsky and Fern{\'a}ndez \cite{AizenBarskFerna87}, this is equivalent to (\ref{eqDefZc}).

\subsubsection{Critical exponents and mean-field behavior}
All three models, self-avoiding walk, percolation and the Ising model,
exhibit a phase transition at some (model-dependent) critical value $z_c$.
One of the fundamental question in statistical mechanics is how models behave at and nearby this critical value.
We use the notion of \emph{critical exponents} to describe this behavior.
While the existence of these critical exponents is folklore,
there is no general argument proving this.

We write $f(z)\asymp g(z)$ if the ratio $f(z)/g(z)$ is bounded away from $0$ and infinity, for some appropriate limit.
For self-avoiding walk, we define the critical exponents $\gs$ and $\es$ by
\begin{eqnarray}
\label{eqDefGammaSAW}
    \chi(z)&\asymp&\left({z_c}-{z}\right)^{-\gs}\qquad\text{as $z\ua z_c$,}\\
\label{eqDefEtaSAW}
    \hat G_{z_c}(k)&\asymp&\frac{1}{|k|^{(\alpha\wedge2)-\es}}\qquad\text{as $k\to0$.}
\end{eqnarray}
For percolation we define the critical exponents $\gp$, $\bp$, $\dep$ and $\ep$ by
\begin{eqnarray}
\label{eqDefGammaPercolation}
    \chi(z)&\asymp&\left(z_c-z\right)^{-\gp}\qquad\text{as $z\ua z_c$,}\\
\label{eqDefBetaPercolation}
    \theta(z)&\asymp&\left(z-z_c\right)^{\bp}\qquad\text{as $z\da z_c$,}\\
\label{eqDefDeltaPercolation}
    \P_{z_c}(|\Ccal|\ge n)&\asymp&\frac{1}{n^{1/\dep}}\qquad\text{as $n\to\infty$,}\\
\label{eqDefEtaPercolation}
    \hat G_{z_c}(k)&\asymp&\frac{1}{|k|^{(\alpha\wedge2)-\ep}}\qquad\text{as $k\to0$.}
\end{eqnarray}
The exponent $\gp$ describes the asymptotic behavior in the subcritical regime $\{z<z_c\}$,
$\bp$ describes the behavior in the supercritical regime $\{z>z_c\}$,
and $\dep$ and $\ep$ describe the behavior at criticality.
For the Ising model, we consider the critical exponents $\gi$, $\bi$, $\di$, $\etai$ defined by
\begin{eqnarray}
\label{eqDefGammaIsing}
    \chi(z)&\asymp&\left(z_c-z\right)^{-\gi}\qquad\text{as $z\ua z_c$,}\\
\label{eqDefBetaIsing}
    M(z,0+)&\asymp&\left(z-z_c\right)^{\bi}\qquad\text{as $z\da z_c$,}\\
\label{eqDefDeltaIsing}
    M(z_c,h)&\asymp&{h^{1/\di}}\qquad\text{as $h\da 0$,}\\
\label{eqDefEtaIsing}
    \hat G_{z_c}(k)&\asymp&\frac{1}{|k|^{(\alpha\wedge2)-\etai}}\qquad\text{as $k\to0$.}
\end{eqnarray}
For a discussion on the construction of $\hat G_{z_c}(k)$ we refer to Section \ref{sectFourier} below.

It is believed that critical exponents are \emph{universal}, i.e., minor modifications of the model, like changes in the underlying graph, leave the general asymptotic behavior, as described by the critical exponents, unchanged.
Their values depend on the dimension $d$. However, it is predicted that there is an \emph{upper critical dimension} $d_c$, such that the critical exponents take the same value for all $d>d_c$.
These values are the \emph{mean-field} values of the critical exponents. For self-avoiding walk these are the values obtained for simple random walk, i.e., $\gs=1$ and $\es=0$, whereas for percolation the mean-field values are $\gp=1$, $\bp=1$, $\dep=2$ and $\ep=0$, which coincide with the corresponding critical exponents obtained for percolation on an infinite regular tree, see \cite[Section 10.1]{Grimm99}.
For the Ising model, these mean-field values are $\gi=1$, $\bi=1/2$, $\di=3$ and $\etai=0$, as obtained for the Curie-Weiss model.

The present paper uses the lace expansion to show that these critical exponents exist and take their mean-field values in sufficiently high dimensions for the nearest-neighbor version of $D$, or $d$ exceeding some critical dimension $d_c$ and $L$ sufficiently large for the spread-out models, respectively.

\subsection{Results}\label{sectResults}
We introduce the (small) quantity $\beta$ by
$\beta=K/d$ for the nearest-neighbor model ($K$ is a uniform constant),
or $\beta=K\,L^{-d}$ for the spread-out models ($K$ is a constant depending on $d$ and $\alpha$).
We make this relation more explicit in Proposition \ref{propBetaSufficient} below.
Be aware that the \emph{critical exponents} $\bp$ and $\bi$ have no relation with the $\beta$ introduced here.

We further introduce the function $\tau\colon z\mapsto \tau(z)$, where $\tau(z)=z$ for self-avoiding walk and percolation, and
    \begin{equation}
    \tau(z)=\sum_{y\in\Zd}\tanh(zJ(y))
    \end{equation}
for the Ising model, cf.\ (\ref{eqDefJ}).

Our main result is the following infrared behavior:
\begin{theorem}[Infrared bound]\label{theoremInfraredBound}
Fix $s=2$ for self-avoiding walk and the Ising model, and $s=3$ for percolation.
Let $d$ sufficiently large in the nearest-neighbor case (at least $d>4s$), or $d>2s$ and $L$ sufficiently large in the finite-variance spread-out case, or $d>(\alpha\wedge2)s$ and $L$ sufficiently large in the spread-out power-law case.
Then
\begin{equation}\label{eqInfraredBound}
    \Gk=\frac{1+O(\beta)}{\chi(z)^{-1}+\tau(z)[1-\Dk]}
\end{equation}
uniformly for $z\in[0,z_c)$ and $k\in\Td$.
\end{theorem}
The infrared bound is well-known in several cases.
Hara and Slade proved the infrared bound for the nearest-neighbor case and the finite-variance spread-out case,
for self-avoiding walk \cite{HaraSlade92b,HaraSlade92a} (see also \cite[Theorem 6.1.6]{MadraSlade93})
as well as for percolation \cite{HaraSlade90a}.
Fr\"ohlich, Simon and Spencer \cite{FrohlSimonSpenc76} proved the upper bound in (\ref{eqInfraredBound}) for the Ising model under the \emph{reflection positivity} assumption, which holds e.g.\ for the nearest-neighbor case.
We discuss reflection positivity in more detail in Section \ref{sectDiscussion}.

By discarding the term $\chi(z)^{-1}$ in (\ref{eqInfraredBound}),
we obtain from Theorem \ref{theoremInfraredBound} that
(under the assumptions formulated there)  
\begin{equation}\label{eqGammaCondition}
    \Gk\le\frac{1+O(\beta)}{\tau(z)[1-\Dk]}
\end{equation}
uniformly for $z<z_c$.

Note that the bound
\begin{equation}\label{eqSingleStepBound}
     G_z(x)-\delta_{0,x}\le\tau(z)\left(D\ast G_z\right)(x).
\end{equation}
holds in all our three models:
for self-avoiding walk this is obvious,
for percolation it follows from the BK-inequality \cite{BergKeste85},
and for the Ising model we use \citeAkira{(4.2)} in the infinite-volume limit.
Thus for $s=2$,
\begin{equation}\label{eqBubbleCondition}
    B(z):=\sum_x G_z(x)^2
    \le 1+\sum_x \tau(z)^2\left(D\ast G_z\right)(x)^2
    \le 1+\tau(z)^2\int_{\Td}\Dk^2\,\Gk^2\dk.
\end{equation}
A combination of (\ref{eqGammaCondition}) and (\ref{eqBubbleCondition}) gives rise to
\begin{equation}\label{eqBubbleCondition2}
    B(z)\le 1+O(1)\int_{\Td}\frac{\Dk^2}{[1-\Dk]^2}\dk
    \le 1+O(\beta),
\end{equation}
where we use that the integrated term is $O(\beta)$ by Assumption \ref{assumptionBeta} and Proposition \ref{propBetaSufficient} below.
A similar calculation gives the corresponding result for $s=3$. More specifically,
\begin{equation}\label{eqTriangleCondition}
    T(z):= \sum_{x,y} G_z(0,x)\,G_z(x,y)\,G_z(y,0)\le1+O(\beta)\qquad\text{when $s=3$.}
\end{equation}
The bounds (\ref{eqBubbleCondition})--(\ref{eqTriangleCondition}) hold uniformly for $z<z_c$ under the assumptions in Theorem \ref{theoremInfraredBound}. Note that in (\ref{eqTriangleCondition}) we write $G_z(x,y)=G_z(x-y)$.
We call $B(z)$ the \emph{bubble diagram} and $T(z)$ the \emph{triangle diagram}.

The two-point function $G_z(x)$ seen as a function of $z$ (for fixed $x$) is continuous.
For self-avoiding walk this fact follows from Abel's Theorem, and for percolation it is a consequence of Aizenman, Kesten and Newman \cite{AizenKesteNewma87}.
A general argument that holds for all our three models is the following:
the quantity $G_z(x)$ 
can be realized as an increasing limit (finite volume approximation) of a function which is continuous and non-decreasing in $z$, hence $G_z(x)$ is left-continuous (cf.\ \cite[Appendix A]{Hara08}).
It follows that (\ref{eqBubbleCondition})--(\ref{eqTriangleCondition}) even hold at criticality, i.e.\ when $z=z_c$.
In  particular, this implies the \emph{bubble condition} (i.e., $B(z_c)<\infty$) or the \emph{triangle condition} (i.e., $T(z_c)<\infty$) for $s=2$ or 3, respectively.
We formulate this fact as a corollary:
\begin{cor}[Bubble/Triangle condition]\label{corolGammaCondition}
Under the assumptions in Theorem \ref{theoremInfraredBound},
$B(z_c)\le1+O(\beta)$ for $s=2$ (self-avoiding walk and Ising model),
and $T(z_c)\le1+O(\beta)$ for $s=3$ (percolation).
\end{cor}
The bubble/triangle condition is important since it implies \emph{mean-field} behavior of the model, which is formulated in the next theorem.
In fact, (\ref{eqInfraredBound}) extends to the critical case $z=z_c$ as
\begin{equation}\label{criticalInfraredBound1}
    \hat G_{z_c}(k)=\frac{1+O(\beta)}{1-\Dk},
\end{equation}
and we refer to the discussion around (\ref{criticalInfraredBound}) below for a construction of $\hat G_{z_c}(k)$ and a derivation of (\ref{criticalInfraredBound1}).

We now use Theorem \ref{theoremInfraredBound} to establish the existence of the formerly introduced critical exponents.
\begin{theorem}[Critical exponents]\label{theoremCritExp}
~\\
\vspace{-.5cm}
\begin{enumerate}
\item\label{itemCritExpSaw}
\textbf{Self-avoiding walk.}
Consider the self-avoiding walk model ($s=2$).
Under the assumptions in Theorem \ref{theoremInfraredBound},
the critical exponent $\gs=1$ for the self-avoiding walk exists.
\item\label{itemCritExpPerc}
\textbf{Percolation.}
Consider the percolation model ($s=3$).
Under the assumptions in Theorem \ref{theoremInfraredBound},
the critical exponents $\gp=1$,
$\bp=1$ and
$\dep=2$
for percolation exist.
\item\label{itemCritExpIsing}
\textbf{Ising model.}
Consider the Ising model ($s=2$).
Under the assumptions in Theorem \ref{theoremInfraredBound},
the critical exponents
$\gi=1$, $\bi=1/2$ and $\di=3$
for the Ising model exist.
\item\label{itemCritExpEta}
For all three models, under the assumptions in Theorem \ref{theoremInfraredBound}
and if $1-\Dk\asymp |k|^{\alpha\wedge2}$, then
\begin{equation}\label{eqCritExpEta}
    \hat G_{z_c}(k)\asymp\frac{1}{|k|^{\alpha\wedge2}}
    \qquad\text{as $k\to0$,}
\end{equation}
i.e., the critical exponents $\es=\ep=\etai=0$ exist.
\end{enumerate}
\end{theorem}
The derivation of the critical exponents from the bubble-/triangle condition (Corollary \ref{corolGammaCondition}) is well-known in the literature.
However, the mode of convergence required for the existence of the critical exponents varies, and some derivations are stated only for finite range models.
We therefore add a more detailed discussion of the literature here.

For self-avoiding walk, the existence (and the value) of the critical exponent $\gs$ is based on the inequality
\begin{equation}\label{eqSawDiffInequality}
    \frac{z_c}{z_c-z}\le\chi(z)\le B(z_c)\left(\frac{z_c}{z_c-z}+1\right).
\end{equation}
Thus the bubble condition (\ref{eqBubbleCondition}) is sufficient to prove that $\gs$ exists and that $\gs=1$. The inequality (\ref{eqSawDiffInequality}) is derived from a differential inequality in \cite[Theorem 2.3]{Slade06},
which was proved there for uniform spread-out models.
The derivation still holds for infinite-range spread-out models due to the multiplicative structure of the \emph{weights} of the self-avoiding walks in (\ref{eqDefCn}).
A version of (\ref{eqSawDiffInequality}) appeared earlier in \cite[(5.30)--(5.33)]{BovieFeldeFroeh84}.

The derivation of the exponents $\gp=1$, $\bp=1$ and $\dep=2$ from the triangle condition is due to Aizenman--Newman \cite{AizenNewma84} and Barsky--Aizenman \cite{BarskAizen91}.
To apply these results in our settings, there are some subtle issues to be resolved, and we discuss these in more detail in Appendix \ref{appendixCritExp}.

For the Ising model, it has been proven by Aizenman \cite[Proposition 7.1]{Aizen82} that the bubble condition implies $\gi=1$ as long as $|J|=\sum_x J(x)<\infty$ (which is equivalent to $\sum_xD(x)<\infty$).
Under the same condition, Aizenman and Fern\'andez \cite{AizenFerna86} proved the existence and mean-field values of the critical exponents $\bi$ and $\di$.

The statement in \emph{(iv)} is an immediate consequence of (\ref{criticalInfraredBound1}).
The lower bound in $1-\Dk\asymp|k|^{\alpha\wedge2}$ follows from (D3)/(D3').
The upper bound indeed holds for a number of examples,
and in particular if $D$ is chosen as in the nearest-neighbor model (\ref{eqDefNNmodel}),
the finite-variance spread-out model (\ref{eqDefUSOmodel})
or the spread-out power-law model (\ref{eqDefDPowerLaw}) with $\alpha\neq2$, cf. \cite{ChenSakai05, HofstSlade02}.
However, if $D$ is chosen as in (\ref{eqDefDPowerLaw}) with $\alpha=2$, then
$1-\Dk\asymp (L|k|)^2\log(\pi/(L|k|)), $
cf.\ \cite[Prop.\ 1.1]{ChenSakai05}.

The proof of Theorem \ref{theoremInfraredBound}, as well as the proof of Corollary \ref{corolGammaCondition}, is given at the end of Section \ref{sectFramework}.

\subsection{Discussion and related work}
\label{sectDiscussion}
There is numerous work on the application of the lace expansion, see the lecture notes by Slade \cite{Slade06} and references therein.
We give more references below at places where we use lace expansion methodology and need particular results.
We now briefly summarize the results known for \emph{long-range} systems.

Long-range \textbf{self-avoiding walk} has rarely been studied.
Klein and Yang \cite{YangKlein88} showed that weakly self-avoiding walk in dimension $d\ge3$ jumping $m$ lattice sites \emph{along the coordinate axes} with probability proportional to $1/m^2$ converges to a Cauchy process
(as for ordinary random walk with such step distribution).
A similar result for strictly self-avoiding walk has been obtained by Cheng \cite{Cheng00}.

For \textbf{percolation}, Hara and Slade \cite{HaraSlade90a} proved the infrared bound for the finite-variance spread-out case when $D$ has exponential tails.
The study of long-range percolation with power law spread-out bonds started in the 1980's by considering the one-dimensional case
\cite{AizenNewma86, NewmaSchul86, Schul83}.
These authors study the case where occupation probabilities are given by (\ref{eqDefDPowerLaw}) with $\alpha\in(0,1]$
and prove criteria for the existence of an infinite cluster.
For example, Aizenman--Newman \cite{AizenNewma86} show that if $D(x)\,|x|^{2}\to 1$ as $|x|\to\infty$ in one dimension, and $D(1)$ is sufficiently large,
then there exists a critical infinite cluster and hence the percolation probability $z\mapsto\theta(z)$ is \emph{discontinuous} at $z_c$.
This is compatible with our results, which imply that there is no infinite cluster at criticality  for $d>3\alpha$ (and here $\alpha=1$).
Berger \cite{Berge02} uses a renormalization argument to show that in dimension $d=1,2$ the infinite cluster (if it exists) is transient if $0<\alpha<d$ and recurrent if $\alpha\ge d$.
He further concludes that in the $d$-dimensional case ($d\ge1$) there is no infinite cluster at criticality if $0<\alpha<d$.
The question whether there exists an infinite critical cluster for $d\ge2$ and $\alpha\ge d$ \cite[Question 6.4]{Berge02} is answered negatively by the present paper for $d>6$ and $L$ sufficiently large.

In a recent paper, Chen and Sakai \cite{ChenSakai05} study \textbf{oriented percolation} in the spread-out power-law case.
Using similar methods, they prove that the two-point function in oriented percolation obeys an infrared bound if $d>2(\alpha\wedge2)$, which implies mean-field behavior of the model.

A long-range \textbf{Ising model} in one dimension has been studied by Aizenman, Chayes, Chayes, and Newman \cite{AizenChayeChayeNewma88}.
Similar to the percolation result in \cite{AizenNewma86}, they prove that in the one-dimensional case where $D(x)\,|x|^2\to1$ as $|x|\to\infty$, the spontaneous magnetization $M(z,0+)$ has a discontinuity at the critical point $z_c$.

The infrared bound for the Ising model was proved in \cite{FrohlSimonSpenc76} for $d>\alpha\wedge2$
for a class of models obeying the \emph{reflection positivity} (RP) property.
The class of models satisfying (RP) includes the nearest-neighbor model (where $D(x)=(2d)^{-1}\1_{\{|x|=1\}}$),
exponential decaying potentials (where $D(x)\propto\exp\{-\mu\|x\|_1\}$ for $\mu>0$), power-law decaying interactions (where $D(x)\propto|x|^{-s}$ for $s>0$), and combinations thereof.
For a definition of (RP) and a discussion of the above mentioned models, we refer to \cite{BiskuChayeCrawf06}.
Nevertheless, (RP) fails in most cases for small perturbations of these models, although it is believed that the asymptotics still hold.
Moreover, (RP) only implies the upper bound in (\ref{eqInfraredBound}), in that implying that the critical exponent $\eta$ (when it exists) is nonnegative.
Our approach using the lace expansion does \emph{not} require reflection positivity,
it is much more universal in the choice of $D$ (cf. Section \ref{sectPropD}),
and also gives a matching lower bound in (\ref{eqInfraredBound}), yielding $\eta=0$.
On the other hand, our approach requires that the dimension $d$ or the spread-out parameter $L$ are sufficiently large, a limitation that one may not expect to reflect the physics.
The literature for the long-range Ising model in higher dimensions based on (RP) arguments is summarized by Aizenman and Fern\'andez \cite{AizenFerna88}, who also identify $2(\alpha\wedge2)$ as upper critical dimension.\footnote{There is a typo in \cite{AizenFerna88}, the value of $\delta$ in \cite[(1.2)]{AizenFerna88} should be 3.}

Given (\ref{eqCritExpEta}) it is folklore that
\begin{equation}\label{eqXspaceAsymptotics}
    G_{z_c}(x)\asymp|x|^{-d+(\alpha\wedge2)}
\end{equation}
holds in the general setting considered here.
Partial results towards (\ref{eqXspaceAsymptotics}) have been obtained.
Indeed, Hara, van der Hofstad and Slade \cite{HaraHofstSlade03} proved (\ref{eqXspaceAsymptotics}) in the
finite-range spread-out setting for self-avoiding walk and percolation, Hara \cite{Hara08} proved it in the nearest-neighbor setting, and Sakai \cite{Sakai07} proved it for the Ising model in finite-range spread-out and nearest-neighbor settings.
We discuss the critical two-point function $G_{z_c}(x)$ at the end of Sect.\ \ref{sectFourier}.

\section{A general framework}\label{sectFramework}
In order to study the various models in a unified way, we use this section to set up a generalized framework. We make two assumptions in terms of the general framework, and use the subsequent two sections to show that our models actually satisfy these assumptions.
We then prove the results within the abstract setting, based on the two assumptions made.
\subsection{An expansion of the two-point function}\label{sectFourier}
Given a step distribution $D$, we consider the \emph{random walk two-point function} or \emph{Green's function} of the random walk defined by
\begin{equation}\label{eqDefCz}
C_z(x)=\sum_{n=0}^\infty D^{\ast n}(x)\,z^n,
\end{equation}
where ${D}^{\ast n}$ is the $n$-fold convolution of $D$ and
$D^{\ast 0}(x)\,z^0=\delta_{x,0}$.
We write $\delta$ for the Kronecker delta function.
By conditioning on the first step we obtain
\begin{equation}\label{eqCz2}
    C_z(x)=\delta_{0,x}+z\left(D\ast C_z\right)(x).
\end{equation}
Taking the Fourier transform and solving for $\hat C_z(k)$ yields
\begin{equation}\label{eqCz3}
    \hat C_z(k)=\frac{1}{1-z\Dk}, \qquad z<1.
\end{equation}

Next we consider $G_z(x)$ defined in (\ref{eqDefGsaw})--(\ref{eqDefGising}).
For each of the three models, i.e., for self-avoiding walk, percolation and the Ising model, we use the \emph{lace expansion} to obtain an expansion formula of the form
\begin{equation}\label{eqDefGx}
    G_z(x)=\delta_{0,x}+\tau(z)\left(D\ast G_z\right)(x)+\left(G_z\ast\Phi_z\right)(x)+\Psi_z(x).
\end{equation}
The coefficients $\Phi_z(x)$ and $\Psi_z(x)$ depend on the model, but above their respective upper critical dimension they obey similar bounds.
Assuming the existence of $\PhiK$ and $\PsiK$, Fourier transformation yields
\begin{equation}\label{eqDefG}
    \Gk=\frac{1+\PsiK}{1-\tau(z)\Dk-\PhiK}, \qquad z<z_c.
\end{equation}
The full derivation of the lace expansion will not be carried out in this paper.
We discuss the lace expansion briefly in Section \ref{sectLaceExpansion}, where we also define the lace expansion coefficients  $\Phi_z$ and $\Psi_z$, and cite bounds on them from \cite{BorgsChayeHofstSladeSpenc05b,Sakai07,Slade06}.
We will see that, for $z=0$, $\hat\Psi_0(k)\equiv0$ and $\hat\Phi_0(k)\equiv0$ for all models.
We recall that $\tau(z)=z$ for self-avoiding walk and percolation, and $\tau(z)=\sum_{y\in\Zd}\tanh(zJ(y))$ for the Ising model, see Sect.\ \ref{sectResults}.

For the critical case (i.e., $z=z_c$) we have
\begin{equation}\label{eqTauCrit}
    1\le\tau(z_c)\le1+O(\beta),
\end{equation}
where the lower bound is a consequence of (\ref{eqSingleStepBound}), and the upper bound emerges from (\ref{eqDefF1}) and (\ref{eqBoundOnF}) below.
The function $G_{z_c}(x)=\lim_{z\nearrow z_c}G_z(x)$ is not in $\ell^1(\Zd)$, hence the Fourier transform does not exist.
However, diagrammatic bounds of the lace expansion coefficients (Prop.\ \ref{propBoundPhiPsi}) and the dominated convergence theorem guarantee the absolute convergence of the various sums involved defining $\PsiK$ and $\PhiK$, which shows that the critical quantities
$\hat\Psi_{z_c}(k)$ and $\hat\Phi_{z_c}(k)$
are well-defined.
This justifies the introduction of $\hat G_{z_c}(k)$ as a solution to (\ref{eqDefG}) with $z=z_c$.
Note that we do not assume any continuity of $z\mapsto\PsiK$ and $z\mapsto\PhiK$ to do this.
Nevertheless, we can extend (\ref{eqInfraredBound}) to the critical case $z=z_c$, and further use (\ref{eqTauCrit}) to obtain
\begin{equation}\label{criticalInfraredBound}
    \hat G_{z_c}(k)=\frac{1+O(\beta)}{1-\Dk}.
\end{equation}
An issue of interest is the (left-) continuity of $\Gk$ at $z=z_c$.
In particular, the identity
\begin{equation}\label{eqGzc2}
    G_{z_c}(x)=\int_{\Td}\e^{-ik\cdot x}\hat G_{z_c}(k)\dk,\qquad x\in\Zd,
\end{equation}
would follow from the the fact that $\PsiK$ and $\PhiK$ are left-continuous at $z=z_c$,
as explained by Hara \cite[Appendix A]{Hara08}.
The left-continuity of $\PsiK$ and $\PhiK$ at $z=z_c$ indeed holds for self-avoiding walk (by Abel's Theorem) and for percolation (by \cite[Lemma A.1]{Hara08}), but a proof for the Ising model is not known.

\subsection{The random walk condition}
Recall that the \emph{model parameter} $s$ is $2$ for self-avoiding walk or Ising model, and $3$ for percolation.
We now make an assumption on the step distribution $D$.

\begin{ass}[Random walk $s$-condition]\label{assumptionBeta}
There exists $\beta>0$ sufficiently small such that
\begin{equation}\label{eqBeta0}
    \sup_{x\in\Zd} D(x)\le\beta
\end{equation}
and
\begin{equation}\label{eqBeta}
    \int_{\Td}\frac{\hat{D}(k)^2}{[1-\hat{D}(k)]^s}\,\dk\le\beta.
\end{equation}
\end{ass}
\noindent
\textit{Remark:} The specific amount of smallness required in (\ref{eqBeta0})--(\ref{eqBeta}) will be specified in the proofs in Section \ref{sectAnalysisLE}.

For $s=2$ we call (\ref{eqBeta}) the random walk bubble condition.
This is inspired by the fact that its $x$-space analogue reads
\begin{equation}\label{eqRWbubble}
(D\ast C_1\ast D\ast C_1)(0)\le\beta.
\end{equation}
In other words, we have an (ordinary) random walk from $0$ to $x$ of at least one step, and a second walk from $x$ to $0$ and subsequently sum over all $x$.
Correspondingly, for $s=3$, we obtain the $x$-space representation
\begin{equation}\label{eqRWtriangle}
    (C_1\ast D\ast C_1\ast D\ast C_1)(0)\le\beta,
\end{equation}
and refer to (\ref{eqBeta}) as the random walk triangle condition.
See the graphical representation in Figure \ref{figureBubble}.

\begin{figure}[ht]
\centering
  \includegraphics[width=8cm]{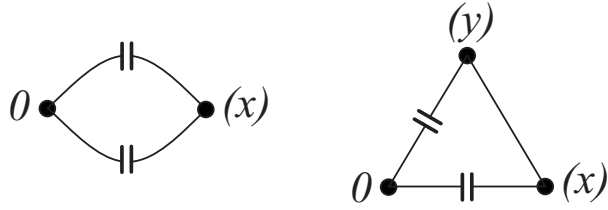}\\
  \caption{Graphical representation of the random walk bubble diagram in (\ref{eqRWbubble}) and the random walk triangle diagram in (\ref{eqRWtriangle}).
  A line between two points, say $x$ and $y$, represents the two-point function $C_1(y-x)$,
  a line with a double dash in the middle requires at least one step, e.g.\ a line between $0$ and $x$ represents $(D\ast C_1)(x)$.
  Vertices labeled in brackets are summed over $\Zd$.}
  \label{figureBubble}
\end{figure}

\begin{prop}\label{propBetaSufficient}
Assumption \ref{assumptionBeta} is satisfied for arbitrarily small $\beta$ if $d$ is chosen sufficiently large in the nearest-neighbor model (at least $d>4s$) or $d>d_c=s(\alpha\wedge2)$ and $L$ is sufficiently large in the spread-out models.
More specifically, the assumption holds with $\beta=O(d^{-1})$ in the nearest-neighbor case, and $\beta=O(L^{-d})$ in the spread-out cases.
\end{prop}
We prove Proposition \ref{propBetaSufficient} in Section \ref{sectBetaCondition}.
We shall prove the following generalized version of Theorem \ref{theoremInfraredBound}.
By Proposition \ref{propBetaSufficient}, Theorem \ref{theoremInfraredBound-GeneralForm} below immediately implies Theorem \ref{theoremInfraredBound}.
\begin{theorem}\label{theoremInfraredBound-GeneralForm}
Fix $s=2$ for self-avoiding walk and the Ising model, and $s=3$ for percolation.
If Assumption \ref{assumptionBeta} is satisfied for $\beta$ sufficiently small, then (\ref{eqInfraredBound})
holds uniformly for $z\in[0,z_c)$ and $k\in\Td$.
\end{theorem}
We remark that Theorem \ref{theoremCritExp} generalizes in the same way.

\subsection{Diagrammatic bounds}
\label{sectAssumptionBounds}
We introduce the quantity
\begin{equation}\label{eqDefLambdaZ}
\lambda_z:=1-\frac{1}{\GN}=1-\frac{1}{\chi(z)}\in[0,1].
\end{equation}
Then $\lambda_z$ satisfies the equality
\begin{equation}
    \GN=\ClN.
\end{equation}
The idea of the proof of Theorem \ref{theoremInfraredBound-GeneralForm} is motivated by the intuition that $\Gk$ and $\ClK$ are comparable in size and, moreover, the discretized second derivative
\begin{equation}
    \Delta_k\hat G_z(l):=\hat{G}_z(l-k)+\hat{G}_z(l+k)-2\hat G(l)
\end{equation}
is bounded by
\begin{equation}\label{eqDefU}
    \Ul:=200\,\ClK^{-1}\left\{ \hat{C}_{\lambda_z}(l-k)\hat{C}_{\lambda_z}(l)+\hat{C}_{\lambda_z}(l)\hat{C}_{\lambda_z}(l+k)+\hat{C}_{\lambda_z}(l-k)\hat{C}_{\lambda_z}(l+k)\right\}.
\end{equation}
More precisely, we will show that the function $f\colon[0,z_c)\to\R$, defined by
\begin{equation}\label{eqDefF}
f:=f_1\vee f_2\vee f_3
\end{equation}
with
\begin{equation}\label{eqDefF1}
f_1(z):=\tau(z),
\qquad f_2(z):=\sup_{k\in\Td}\frac{\Gk}{\ClK},
\end{equation}
and
\begin{equation}\label{eqDefF3}
f_3(z):=\sup_{k,l\in\Td}\frac{|\Delta_k\hat G_z(l)|}{\Ul},
\end{equation}
is small, given that $\beta$ in Assumption \ref{assumptionBeta} is sufficiently small.
To make this rigorous, we need the following assumption:

\begin{ass}[Bounds on the lace expansion coefficients]\label{assumptionDiagrammaticBounds}
    If, for some $K>0$, the inequality $f(z)\le K$ holds uniformly for $z\in(0,z_c)$,
    then there exists a constant $c_K>0$ such that, for all $k\in\Td$,
    \begin{equation}\label{eqBoundPhiK}
    \left|\PsiK\right|\le c_K\beta,\qquad\left|\PhiK\right|\le c_K\beta
    \end{equation}
    and
    \begin{equation}\label{eqBoundPhi}
    \sum_x[1-\cos(k\cdot x)]\left|\Psi_z(x)\right|\le c_K\beta\,\ClK^{-1}, \qquad
    \sum_x[1-\cos(k\cdot x)]\left|\Phi_z(x)\right|\le \tau(z)\,c_K\beta\,\ClK^{-1}
    \end{equation}
    where $\Phi_z$ and $\Psi_z$ refer to the model-dependent coefficients in the expansion formula (\ref{eqDefGx}).
\end{ass}

The key to our results is that the bounds (\ref{eqBoundPhiK})--(\ref{eqBoundPhi}) imply Theorem \ref{theoremInfraredBound-GeneralForm} (and hence Theorem \ref{theoremInfraredBound}):

\proof[Proof of Theorem \ref{theoremInfraredBound-GeneralForm} subject to (\ref{eqBoundPhiK})--(\ref{eqBoundPhi})]
Let
\begin{equation}\label{eqLowerCritBound}
    m_z=1-\tau(z)-\PhiN.
\end{equation}
Then,
\begin{equation}\label{eqMainProof5}
    \Gk=\frac{1+\PsiK}{1-\tau(z)\Dk-\PhiK}
    =\frac{1+\PsiK}{m_z+\tau(z)[1-\Dk]+[\PhiN-\PhiK]}.
\end{equation}
By the first inequality in (\ref{eqBoundPhiK}) and the second in (\ref{eqBoundPhi}) in Assumption \ref{assumptionDiagrammaticBounds},
\begin{equation}\label{eqMainProof8}
    \Gk=\frac{1+O(\beta)}{m_z+\tau(z)\,[1-\Dk]+\tau(z)\,O(\beta)\,\Ci}.
\end{equation}
Evaluating (\ref{eqMainProof5}) for  $k=0$ yields
\begin{equation}\label{eqMainProof6}
    \chi(z)=\G_z(0)=\frac{1+\PsiN}{m_z},
\end{equation}
and the first inequality in (\ref{eqBoundPhiK}) implies
\begin{equation}\label{eqMainProof7}
    m_z=(1+O(\beta))\,\chi(z)^{-1}.
\end{equation}
Furthermore, by (\ref{eqCz3}) and (\ref{eqDefLambdaZ}),
\begin{equation}\label{eqMainProof9}
    \Ci=1-\lambda_z\Dk=1-\Dk+\chi(z)^{-1}\Dk.
\end{equation}
A combination of (\ref{eqMainProof8}), (\ref{eqMainProof7}), (\ref{eqMainProof9}) and the bounds $|\Dk|\le1$, $\tau(z)\le O(1)$ leads to
\begin{equation}\label{eqMainProof10}
    \Gk=\frac{1+O(\beta)}{(1+O(\beta))\,\chi(z)^{-1}+\tau(z)\,(1+O(\beta))\,[1-\Dk]},
\end{equation}
which implies (\ref{eqInfraredBound}).
\qed

We proceed by validating  (\ref{eqBoundPhiK})--(\ref{eqBoundPhi}).
First we realize that Assumption \ref{assumptionDiagrammaticBounds} indeed holds for the models under consideration:
\begin{prop}\label{propBoundPhiPsi}
Under the assumptions in Theorem \ref{theoremInfraredBound}, Assumption \ref{assumptionDiagrammaticBounds} holds for self-avoiding walk, percolation and the Ising model.
\end{prop}
The relevant bounds have been proven by Slade \cite{Slade06} for self-avoiding walk, by Borgs et al.\ \cite{BorgsChayeHofstSladeSpenc05b} for percolation (on finite graphs), and by Sakai \cite{Sakai07} for the Ising model.
In Section \ref{sectLaceExpansion} we state the diagrammatic bounds proved in these papers, and relate them to our version of $\Phi_z$ and $\Psi_z$, thus proving Proposition \ref{propBoundPhiPsi} using \cite{BorgsChayeHofstSladeSpenc05b,Sakai07,Slade06}.

\subsection{Completion of the argument and organization of proofs}
\label{sectInfraredBound}
The proof of Theorem \ref{theoremInfraredBound-GeneralForm} will follow from the following proposition:
\begin{prop}\label{theoremLaceExpansion}
Suppose we are given a model with some model-dependent constant $s\in\{2,3,\dots\}$, and a two-point function $G_z$ of the form
(\ref{eqDefGx}), where the step distribution $D$ satisfies Assumption \ref{assumptionBeta},
and $\Phi_z$ and $\Psi_z$ satisfy Assumption \ref{assumptionDiagrammaticBounds},
both for the same sufficiently small $\beta>0$.
Assume further that $\chi'(z)\le O(\chi(z)^2)$, $z\in[0,z_c)$.
Then
\begin{equation}\label{eqBoundOnF}
    f(z)\le1+O(\beta)
\end{equation}
uniformly for $z< z_c$.
\end{prop}

The assumption $\chi'(z)\le \const\chi(z)^2$ in Proposition \ref{theoremLaceExpansion} can be replaced by assuming that $f$ is continuous on $[0,z_c)$, cf.\ Lemma \ref{lemmaContinuity} below.
It is known as a \emph{mean-field bound}, and a proof of it can be found in \cite[Theorem 2.3]{Slade06} for self-avoiding walk, and in \cite[Prop.\ 9.2]{Slade06} for percolation. For the Ising model, this mean-field bound is a consequence of the Lebowitz inequality \cite{Lebow74}.

In order for Theorem \ref{theoremInfraredBound-GeneralForm} (and hence Theorem \ref{theoremInfraredBound} and Corollary \ref{corolGammaCondition}) to hold,
we need to show (\ref{eqBoundPhiK})--(\ref{eqBoundPhi}).
Indeed, (\ref{eqBoundPhiK})--(\ref{eqBoundPhi}) follow from the statements above, as we explain now.
Propositions \ref{propBetaSufficient} and \ref{propBoundPhiPsi}
validate Assumptions \ref{assumptionBeta} and \ref{assumptionDiagrammaticBounds}.
With these assumptions, the prerequisites of Proposition \ref{theoremLaceExpansion} are satisfied and (\ref{eqBoundOnF}) holds for $\beta$ sufficiently small by Proposition \ref{propBetaSufficient}.
The latter can be achieved by taking $d$ or $L$ large enough.
Then we again use Assumption \ref{assumptionDiagrammaticBounds} to obtain (\ref{eqBoundPhiK})--(\ref{eqBoundPhi}),
thus proving (\ref{eqInfraredBound}).

The remainder of the paper is organized as follows.
In Section \ref{sectBetaCondition} we prove Proposition \ref{propBetaSufficient} by showing that Assumption \ref{assumptionBeta} is satisfied for our versions of $D$.
For the proof of Proposition \ref{propBoundPhiPsi} we need the \emph{lace expansion}.
The diagrammatic bounds are not derived in the present paper;
instead we explain in Section \ref{sectLaceExpansion} how to obtain the statement of Proposition \ref{propBoundPhiPsi} from the diagrammatic bounds in \cite{Slade06} for self-avoiding walk, \cite{BorgsChayeHofstSladeSpenc05b} for percolation, and \cite{Sakai07} for the Ising model.
Finally, the proof of Proposition \ref{theoremLaceExpansion} is contained in the last Section \ref{sectAnalysisLE},
and this completes the proof of Theorem \ref{theoremInfraredBound-GeneralForm} (and hence of Theorem \ref{theoremInfraredBound} and Corollary \ref{corolGammaCondition}).
Appendix \ref{appendixCritExp} contains a derivation of the existence and the mean-field values of the critical exponents $\gp$ and $\dep$ for percolation.
In Appendix \ref{appendixBoundsIsing} we show how the bounds on the lace expansion in Assumption \ref{assumptionDiagrammaticBounds} for the Ising model can be obtained from the diagrammatic bounds in \cite{Sakai07}. Our account in Appendix \ref{appendixBoundsIsing} follows the proof of \cite[Prop.\ 3.2]{Sakai07}, but with a modified bootstrap hypothesis.

\section{The random walk two-point function}\label{sectBetaCondition}
In this section we prove Proposition \ref{propBetaSufficient}.
The estimates below are contained in \cite[Sect.\ 2.2.2]{BorgsChayeHofstSladeSpenc05b}, where finite tori are considered. Restriction to the infinite lattice gives rise to a noteworthy simplification, which we shall present in the following.
\begin{proof}[Proof of Proposition \ref{propBetaSufficient} for the nearest-neighbor model.]
We follow \cite[Sect.\ 2.2.2]{BorgsChayeHofstSladeSpenc05b}.
Since $\|D\|_\infty=(2d)^{-1}$, the bound (\ref{eqBeta0}) is satisfied for $d$ sufficiently large, and it remains to prove (\ref{eqBeta}).

By the symmetry of $D$ we have
\begin{equation}\label{eqBeta1}
    \Dk=\sum_{x\in\Zd}D(x)\,\cos(k\cdot x)=\frac{1}{d}\sum_{j=1}^d\cos(k_j),\qquad k=(k_1,\dots,k_d)\in\Td.
\end{equation}
Since $1-\cos t\ge2\pi^{-2}t^2$ for $|t|\le\pi$, this implies the infrared bound
\begin{equation}\label{eqBeta5}
    1-\Dk\ge\frac{2}{\pi^2}\frac{|k|^2}{d}.
\end{equation}
The Cauchy-Schwarz inequality\footnote{The H\"older inequality gives better bounds here. In particular, it requires $d>2s$ only, cf.\ (2.19) in \cite{BorgsChayeHofstSladeSpenc05b}.}
yields
\begin{equation}\label{eqBeta2}
    \int_{\Td}\frac{\Dk^2}{[1-\Dk]^s}\,\dk\le\left(\int_{\Td}\Dk^{4}\,\dk\right)^{1/2}\left(\int_{\Td}\frac{1}{[1-\Dk]^{2s}}\,\dk\right)^{1/2}
\end{equation}
First we show that the first term on the right hand side of (\ref{eqBeta2}) is small if $d$ is large.
Note that $\int_{\Td} \Dk^4$ $(2\pi)^{-d} \,\d k$ $=$ $D^{\ast4} (0)$ is the probability that a nearest-neighbor random walk returns to its starting point after the fourth step. This is bounded from above by $c(2d)^{-2}$ with $c$ being a well-chosen constant,
because the first two steps must be compensated by the last two.
Finally, the square root yields the upper bound $O(d^{-1})$.

It remains to show that the second term on the right of (\ref{eqBeta2}) is bounded uniformly in $d$.
The infrared bound (\ref{eqBeta5}) gives
\begin{equation}\label{eqBeta6}
    \int_{\Td}\frac{1}{[1-\Dk]^{2s}}\,\dk\le\frac{\pi^{4s}}{2^{2s}}\int_{\Td}\frac{d^{2s}}{|k|^{4s}}\,\dk.
\end{equation}
The right hand side of (\ref{eqBeta6}) is finite if $d>4s$.
For $A>0$ and $m>0$,
\begin{equation}\label{eqBeta3}
    \frac{1}{A^m}=\frac{1}{\Gamma(m)}\int_0^\infty t^{m-1}\e^{-tA}\d t.
\end{equation}
Applying this with $A=|k|^2/d$ and $m=2s$ yields
\begin{equation}\label{eqBeta4}
    \frac{1}{\Gamma(2s)}\;\frac{\pi^{4s}}{2^{2s}} \int_0^\infty t^{2s-1}\Bigg(\int_{-\pi}^\pi\big(\e^{-t\theta^2}\big)^{1/d}\,\frac{\d\theta}{2\pi}\Bigg)^d\d t
\end{equation}
as an upper bound for (\ref{eqBeta6}).
This is non-increasing in $d$, because $\|f\|_p\le\|f\|_q$ for $0<p\le q\le\infty$ on a probability space by Lyapunov's inequality.
\end{proof}

\begin{proof}[Proof of Proposition \ref{propBetaSufficient} for the spread-out models.]
We again follow \cite[Sect.\ 2.2.2]{BorgsChayeHofstSladeSpenc05b}.
Obviously (\ref{eqBeta0}) is implied by condition (D2)/(D$2'$) for sufficiently large $L$, hence it remains to prove (\ref{eqBeta}).

The power-law spread-out model with $\alpha>2$ satisfies the finite variance condition (D1) with $\eps<\alpha-2$.
Note further that (D3) and (D$3'$) agree when the exponent in the first inequality is taken $\alpha\wedge2$.

We separately consider the regions $\|k\|_\infty\le L^{-1}$ and $\|k\|_\infty> L^{-1}$.
By (\ref{eqPropD1}), (\ref{eqPropDPowerLaw1}) and the bound $\Dk^2\le1$, the corresponding contributions to the integral are
\begin{equation}\label{eqBeta7}
    \int_{k:\|k\|_\infty\le L^{-1}}\frac{\Dk^2}{[1-\Dk]^s}\dk
    \le\frac{1}{c_1^sL^{(\alpha\wedge2)s}}\int_{k:\|k\|_\infty\le L^{-1}}\frac{1}{|k|^{(\alpha\wedge2)s}}\dk
    \le C_{d,c_1}L^{-d}
\end{equation}
if $d>(\alpha\wedge2)s$, where $C_{d,c_1}$ is a constant depending (only) on $d$ and $c_1$,
and by (\ref{eqPropD2}), (\ref{eqPropDPowerLaw2}),
\begin{equation}\label{eqBeta8}
    \int_{k:\|k\|_\infty> L^{-1}}\frac{\Dk^2}{[1-\Dk]^s}\dk
    \le {c_2}^{-s}\int_{k:\|k\|_\infty> L^{-1}}\Dk^2\dk
    \le \const L^{-d},
\end{equation}
for some positive constant. In the last step we used assumption (D2) / (D$2'$) to see that
\begin{equation}
    \int_{k\in\Td}\Dk^2\dk=(D\ast D)(0)=\sum_{y\in\Zd}D(y)^2\le\sum_{y\in\Zd}D(y)\,\|D\|_\infty=\|D\|_\infty\le\const L^{-d}.
\end{equation}
\end{proof}

\section{The lace expansion}\label{sectLaceExpansion}

In this section, we discuss the \emph{lace expansion} which obtains an expansion of the two-point function of the form
$$G_z(x)=\delta_{0,x}+\tau(z)\left(D\ast G_z\right)(x)+\left(G_z\ast\Phi_z\right)(x)+\Psi_z(x), $$
cf.\ (\ref{eqDefGx}).
The key point is to identify the \emph{lace-expansion coefficients} $\Phi_z$ and $\Psi_z$ in a way that allows for sufficient bounds, known as \emph{diagrammatic bounds}.
The derivation is not carried out in this paper; full expansions and detailed derivations of the diagrammatic bounds are performed in \cite{Hofst05,Slade06} for self-avoiding walk, in \cite{BorgsChayeHofstSladeSpenc05b} for percolation and in \cite{Sakai07} for the Ising model.

\subsection{The lace expansion for the self-avoiding walk}\label{sectLaceExplansionSaw}
The lace expansion for self-avoiding walks was first presented by Brydges and Spencer \cite{BrydgSpenc85}. They provide an algebraic expansion using graphs. A special class of graphs that play an important role here, the laces, gave the lace expansion its name.
An alternative approach is based on an inclusion-exclusion argument, and was first presented by Slade \cite{Slade87}.

We refer the reader to \cite[Sect.\ 2.2.1]{Hofst05} or \cite[Sect.\ 3]{Slade06} for a full derivation of the expansion.
For example, in \cite[Sect.\ 2.2.1]{Hofst05} it is shown that
\begin{equation}\label{e:cnk}
    {c}_{n+1}(x) = ({D}\ast{c}_n)(x)
    + \sum_{m=2}^{n+1} \left({\pi}_m \ast {c}_{n+1-m}\right)(x)
\end{equation}
for suitable $\pi_m(x)$.
We multiply (\ref{e:cnk}) by $z^{n+1}$ and sum over $n\ge0$.
By letting $\Pi_z(x)=\sum_{m=2}^\infty\pi_m(x)z^m$ and recalling
$G_z(x)=\sum_{n=0}^\infty c_n(x)z^n$ this yields
\begin{equation}\label{eqSawExpansion}
    G_z(x)=\delta_{0,x}+z(D\ast G_z)(x)+(G_z\ast\Pi_z)(x),
\end{equation}
see also \cite[(3.27)]{Slade06}.
For the lace expansion coefficient $\Pi_z$ the following diagrammatic bound is proven:
\begin{prop}[Diagrammatic estimates for self-avoiding walk from {\cite{Slade06}}]\label{propBoundSaw}
Fix $z\in(0,z_c)$.
If $f(z)$ of (\ref{eqDefF}) obeys $f(z)\le K$, then there are positive constants $c_K$ and $\beta_0=\beta_0(K)$, such that the following holds:
If Assumption \ref{assumptionBeta} holds for some $\beta\le\beta_0$, then
\begin{equation}\label{eqSawBound1}
    \sum_{x\in\Zd}\left|\Pi_z(x)\right|\le c_K\beta,
\end{equation}
\begin{equation}\label{eqSawBound2}
    \sum_{x\in\Zd}[1-\cos(k\cdot x)]\left|\Pi_z(x)\right|\le c_K\beta\, \Ci, \qquad k\in\Td.
\end{equation}
\end{prop}
The term \emph{diagrammatic estimate} originates from the fact that $\Pi_z$ is expressed in terms of diagrams.
The underlying structure expressed in terms of these diagrams is heavily used to obtain the bounds in (\ref{eqSawBound1}) and (\ref{eqSawBound2}).

A proof of Prop.\ \ref{propBoundSaw} can be found in \cite[Lemma 5.11]{Slade06}, and we do not repeat it here.
In fact, the proof in \cite{Slade06} can be modified to obtain
\begin{equation}\label{eqSawBound2a}
    \sum_{x\in\Zd}[1-\cos(k\cdot x)]\left|\Pi_z(x)\right|\le z\,c_K\beta\, \Ci, \qquad k\in\Td.
\end{equation}
instead of (\ref{eqSawBound2}). This is achieved by leaving the factor $z$ in \cite[(5.42) and (5.43)]{Slade06} explicit (rather then bounding above by $K$).

We choose $\tau(z)=z$, $\Phi_z(x)=\Pi_z(x)$ and $\Psi_z(x)=0$, which makes (\ref{eqSawExpansion}) equivalent to (\ref{eqDefGx}).
Hence Prop.\ \ref{propBoundSaw} along with (\ref{eqSawBound2a}) is sufficient to prove Proposition \ref{propBoundPhiPsi} for self-avoiding walk.

\subsection{The lace expansion for percolation}\label{sectLaceExpansionPercolation}
The lace expansion for percolation was first derived in \cite{HaraSlade90a}.
It is based on an inclusion-exclusion argument, and holds quite generally for any connected graph, finite or infinite.
The graph does not even need to be transitive or regular.

In \cite[Sect.\ 3.2]{BorgsChayeHofstSladeSpenc05b}, the identity
\begin{equation}\label{eqPercExpansion1}
G_z(x)=\delta_{0,x}+z\left(D\ast G_z\right)(x)+z\left(\Pi_\sM\ast D\ast G_z\right)(x)+\Pi_\sM(x)+R_\sM(x)
\end{equation}
is derived for $M=0,1,2,\dots$.
The $z$-dependence of $\Pi_\sM$ and $R_\sM$ is left implicit.
The function $\Pi_\sM\colon\Zd\to\R$ is the central quantity in the expansion, and $R_\sM(x)$ is a remainder term.
When the expansion converges, one has
\begin{equation}\label{eqPercExpansion3}
    \lim_{M\to\infty}\sum_x|R_\sM(x)|=0.
\end{equation}
The subscript $M$ denotes the level to which the (inclusion-exclusion) expansion is carried out, and we shall later fix $M$ so large that (\ref{eqDiaBoundPerco3}) and (\ref{eqDiaBoundPerco4}) below are satisfied for $K=4$.
The equality (\ref{eqPercExpansion1}) is equivalent to (\ref{eqDefGx}) if we let $\tau(z)=z$, and
\begin{equation}\label{eqDefPhiPerc}
    \Phi_z(x)=z(D\ast\Pi_\sM)(x),\qquad x\in\Zd,
\end{equation}
\begin{equation}\label{eqDefPsiPerc}
    \Psi_z(x)=\Pi_\sM(x)+R_\sM(x),\qquad x\in\Zd.
\end{equation}

The key point is that $\Pi_\sM$ and $R_\sM$ satisfy useful \emph{diagrammatic bounds}:
\begin{prop}[Diagrammatic estimates for percolation from {\cite{BorgsChayeHofstSladeSpenc05b}}]\label{propBoundPerc}
Fix $z\in(0,z_c)$.
If $f(z)$ of (\ref{eqDefF}) obeys $f(z)\le K$, then there are positive constants $c_K$ and $\beta_0=\beta_0(K)$, such that the following holds:
If Assumption \ref{assumptionBeta} holds for some $\beta\le\beta_0$, then for all $M=0,1,2,\dots$,
\begin{equation}\label{eqDiaBoundPerco1}
    \sum_x|\Pi_\sM(x)|\le c_K\beta,
\end{equation}
\begin{equation}\label{eqDiaBoundPerco2}
    \sum_x[1-\cos(k\cdot x)]\,|\Pi_\sM(x)|\le c_K\beta\Ci,
\end{equation}
and for $M$ sufficiently large (depending on $K$ and $z$),
\begin{equation}\label{eqDiaBoundPerco3}
    \sum_x|R_\sM(x)|\le\beta,
\end{equation}
\begin{equation}\label{eqDiaBoundPerco4}
    \sum_x[1-\cos(k\cdot x)]\,|R_\sM(x)|\le\beta\Ci.
\end{equation}
\end{prop}
In fact, the bounds in Proposition \ref{propBoundPerc} are not exactly as phrased in \cite{BorgsChayeHofstSladeSpenc05b}.
In the following we explain how the proof of \cite[Prop.\ 5.2]{BorgsChayeHofstSladeSpenc05b} can be modified to obtain Prop.\ \ref{propBoundPerc}.
There are two differences to consider.
First, in the definition of $f_3$ there is a factor 16 in the denominator, whereas we have a factor 200, cf.\ (\ref{eqDefF3}).
This can be controlled easily by changing the factor appropriately throughout the proof of \cite[Prop.\ 5.2]{BorgsChayeHofstSladeSpenc05b}.
This changes the specific value of $c_K$, but the statement of \cite[Prop.\ 5.2]{BorgsChayeHofstSladeSpenc05b} remains unchanged.
The second (and more important) issue is the replacement of $1-\Dk=\hat C_1(k)^{-1}$ in \cite[Prop.\ 5.2]{BorgsChayeHofstSladeSpenc05b} by $1-\lambda_z\Dk=\Ci$ in Prop.\ \ref{propBoundPerc}.
We need to do this replacement to achieve continuity of the function $f_3$.
Wherever the bound on $f_3$ is used in the proof of \cite[Prop.\ 5.2]{BorgsChayeHofstSladeSpenc05b}, which is in \cite[(5.63)]{BorgsChayeHofstSladeSpenc05b}, \cite[(5.77)]{BorgsChayeHofstSladeSpenc05b}, below \cite[(5.93)]{BorgsChayeHofstSladeSpenc05b} and in \cite[(5.97)]{BorgsChayeHofstSladeSpenc05b}, we replace the factor $[1-\Dk]$ by $\Ci$.
Other occurrences of $[1-\Dk]$, as in \cite[(5.75)]{BorgsChayeHofstSladeSpenc05b} and \cite[(5.91)]{BorgsChayeHofstSladeSpenc05b},
can be treated with the bound
\begin{equation}\label{eqDCbound}
    0\le1-\Dk\le2\Ci,\qquad k\in[-\pi,\pi)^d,
\end{equation}
which itself is a consequence of
\begin{equation}\label{eqBootstrap5}
    0\le\ClK\,[1-\Dk]=1+\frac{\lambda_z-1}{1-\lambda_z\Dk}\Dk\le2.
\end{equation}
Again, this increases the value of the constant $c_K$, but leaves the statement of \cite[Prop.\ 5.2]{BorgsChayeHofstSladeSpenc05b} otherwise unchanged.

For a sketch of the argument of how $f(z)\le K$ actually implies (\ref{eqDiaBoundPerco1})--(\ref{eqDiaBoundPerco4}) we refer to \cite[Sect.\ 3.2]{Sakai08}.
In the following we show how Proposition \ref{propBoundPerc} implies Proposition \ref{propBoundPhiPsi} in the percolation case.
\begin{proof}[Proof of Proposition \ref{propBoundPhiPsi} for percolation.]
Recall (\ref{eqDefPhiPerc})--(\ref{eqDefPsiPerc}).
The bounds on $\Psi_z(x)$ in (\ref{eqBoundPhiK})--(\ref{eqBoundPhi}) follow directly from Proposition \ref{propBoundPerc} if $M$ is chosen so large that (\ref{eqDiaBoundPerco3})--(\ref{eqDiaBoundPerco4}) is satisfied.

For the bounds on $\Phi_z(x)=z(D\ast \Pi_\sM)(x)$ we use the estimate
\begin{equation}
    [1-\cos(t_1+t_2)]\le5\left([1-\cos t_1]+[1-\cos t_2]\right), \qquad t_1,t_2\in\R,
\end{equation}
(see \cite[(4.51)]{BorgsChayeHofstSladeSpenc05b}) to obtain
\begin{eqnarray}
  \sum_x[1-\cos(k\cdot x)]\left|\Phi_z(x)\right|
  &\le& 5\sum_xz\sum_y\big([1-\cos(k\cdot y)]+[1-\cos(k\cdot (x-y))]\big)D(y)\left|\Pi_\sM(x-y)\right|\nnb
  &\le& 5z\sum_x[1-\Dk]\left|\Pi_\sM(x-y)\right|\nnb
  &&    +5z\sum_x[1-\cos(k\cdot (x-y))]\,\left|\Pi_\sM(x-y)\right|\nnb
  &\le& 5z\left(2c_K\beta\Ci+c_K\beta\Ci\right).
\end{eqnarray}
by (\ref{eqDiaBoundPerco1})--(\ref{eqDiaBoundPerco2}) and (\ref{eqDCbound}).
\end{proof}

\subsection{The lace expansion for the Ising model}\label{sectLaceExpansionIsing}
The lace expansion for the Ising model has been established recently by Sakai \cite{Sakai07}.
It is similar in spirit to a high-temperature expansion.
A key point is to rewrite the two-point function (spin-spin correlation) using the random-current representation.
This gives rise to a representation involving bonds, in that showing some similarities to a percolation configuration.
The lace expansion is then performed using ideas from the lace expansion for percolation, however, it is considerably more involved.

For the Ising model on a finite graph $\Lambda$, Sakai in \cite[Prop.\ 1.1]{Sakai07} proved the expansion formula
\begin{equation}\label{eqIsingExpansion}
    G_z^{\Lambda}(x)
    =\delta_{0,x}+\tau(z)\left(D\ast G^{\Lambda}_z\right)(x)
    +\tau(z)\left(D\ast\Pi^{\Lambda}_\sM\ast G^{\Lambda}_z\right)(x)
    +\Pi^{\Lambda}_\sM(x)
    +R^{\Lambda}_{\sss M}(x),
\end{equation}
where the $z$-dependence of $\Pi^{\Lambda}_\sM$ and $R^{\Lambda}_\sM$ is omitted from the notation.
Note that $R^{\Lambda}_\sM$ in this paper is $(-1)^{M+1}R^{\sss (M+1)}_{p;\Lambda}$ in \cite{Sakai07}.
Here $M$ refers to the level of the expansion, and $G_z^{\Lambda}$ denotes the finite-volume two-point function.
This is equivalent to (\ref{eqDefGx}) if we let
\begin{equation}\label{eqDefPhiIsing}
    \Phi_z^{\Lambda}(x)=\tau(z)(D\ast\Pi^{\Lambda}_\sM)(x), \qquad x\in\Zd,
\end{equation}
\begin{equation}\label{eqDefPsiIsing}
    \Psi_z^{\Lambda}(x)=\Pi^{\Lambda}_\sM(x)+R^{\Lambda}_{\sss M}(x), \qquad x\in\Zd,
\end{equation}
then choose $M$ so large that (\ref{eqDiaBoundIsing3}) and (\ref{eqDiaBoundIsing4}) below are satisfied for a certain $K$, say $K=4$, and subsequently taking the thermodynamic limit $\Lambda\nearrow\Zd$.
Note that, if comparing (\ref{eqIsingExpansion}) to \cite[(1.11)]{Sakai07}, we explicitly extract the $\delta_{0,x}$-term from the $\Pi$-term in \cite{Sakai07}, i.e., $\Pi^{\sss(M)}_{p;{\Lambda}}(x)$ in \cite{Sakai07} corresponds to $\Pi^\Lambda_\sM(x)+\delta_{0,x}$ in this paper.
For $\Pi^{\Lambda}_\sM$ and $R^{\Lambda}_\sM$ we have the following bounds:

\begin{prop}[Diagrammatic estimates for the Ising model from \cite{Sakai07}]\label{propBoundIsing}
Fix $z\in(0,z_c)$.
If $f(z)$ of (\ref{eqDefF}) obeys $f(z)\le K$, then there are positive constants $c_K$ and $\beta_0=\beta_0(K)$, such that the following holds:
If Assumption \ref{assumptionBeta} holds for some $\beta\le\beta_0$, then for all $M=0,1,2,\dots$,
\begin{equation}\label{eqDiaBoundIsing1}
    \sum_x|\Pi^{\Lambda}_\sM(x)|\le c_K\beta,
\end{equation}
\begin{equation}\label{eqDiaBoundIsing2}
    \sum_x[1-\cos(k\cdot x)]\,|\Pi^{\Lambda}_\sM(x)|\le c_K\beta\Ci,
\end{equation}
and for $M$ sufficiently large (depending on $K$ and $z$),
\begin{equation}\label{eqDiaBoundIsing3}
    \sum_x|R^{\Lambda}_\sM(x)|\le\beta,
\end{equation}
\begin{equation}\label{eqDiaBoundIsing4}
    \sum_x[1-\cos(k\cdot x)]\,|R^{\Lambda}_\sM(x)|\le\beta\Ci.
\end{equation}
These bounds hold uniformly in $\Lambda$.
\end{prop}
Since the bootstrapping hypothesis used in Section \ref{sectAnalysisLE} in this paper is different from that in \cite{Sakai07}, it is not so obvious how Prop.\ \ref{propBoundIsing} follows from the results in \cite{Sakai07}.
In Appendix \ref{appendixBoundsIsing} we explain how the statement in \cite[Prop.\ 3.2]{Sakai07} can be modified to obtain the desired bounds (\ref{eqDiaBoundIsing1})--(\ref{eqDiaBoundIsing4}).

We prove Prop.\ \ref{propBoundPhiPsi} for the Ising model as in the percolation case, now using Prop.\ \ref{propBoundIsing} instead of Prop.\ \ref{propBoundPerc}.
We refrain from repeating the argument.

\section{Analysis of the lace expansion}\label{sectAnalysisLE}

\subsection{The bootstrap argument}
In this section we prove Proposition \ref{theoremLaceExpansion} and, by doing so, complete the proof of Theorem \ref{theoremInfraredBound}.
The proof is based on the following lemma:
\begin{lemma}[The bootstrap / forbidden region argument]\label{lemmaBootstrap}
Let $f$ be a continuous function on the interval $\left[0,z_c\right)$, and assume that $f(0)\le3$. Suppose for each $z\in(0,z_c)$ that if $f(z)\le4$, then in fact $f(z)\le3$. Then $f(z)\le3$ for all $z\in\left[0,z_c\right)$.
\end{lemma}
\begin{proof}
This is a straightforward application of the intermediate value theorem for continuous functions, see also \cite[Lemma 5.9]{Slade06}.
\end{proof}
The bootstrap argument in Lemma \ref{lemmaBootstrap} is often used in lace expansion, see e.g.\ \cite[Section 6.1]{MadraSlade93}.
An alternative approach that involves an induction argument has been applied in \cite{HofstSlade02}, see also the lecture notes by van der Hofstad \cite{Hofst05}.

In the remainder of the section, we prove that the function $f$ defined in (\ref{eqDefF}) obeys the prerequisites of Lemma \ref{lemmaBootstrap}.
We therefore have to show that $f(0)\le3$, that $f$ is continuous on $\left[0,z_c\right)$, and that $f(z)\le4$ implies $f(z)\le3$ for $z\in(0,z_c)$. The latter is referred to as the \emph{improvement of the bounds}.

Let us first check that $f(0)\le3$. Clearly, $f_1(0)=0$. Note that $\hat\Psi_0(k)\equiv0$ and $\hat\Phi_0(k)\equiv0$. This leads to  $\hat G_0(k)\equiv1$ and $\lambda_0=0$, hence $f_2(0)=1$ and $f_3(0)=0$.

Next we want to prove continuity of $f$.
To this end, we need the following lemma:
\begin{lemma}[Continuity of equicontinuous functions]\label{lemmaEquicontinuity}
Let $\left(f_\alpha\right)_{\alpha\in A}$ be an equicontinuous family of functions on an interval $[t_1,t_2]$,
i.e., for every given $\eps>0$, there is a $\delta>0$ such that $|f_\alpha(s)-f_\alpha(t)|<\eps$ whenever $|s-t|<\delta$, uniformly in $\alpha\in A$.
Furthermore, suppose that $\sup_{\alpha\in A}f_\alpha(t)<\infty$ for each $t\in[t_1,t_2]$.
Then $t\mapsto\sup_{\alpha\in A}f_\alpha(t)$ is continuous on $[t_1,t_2]$.
\end{lemma}
A proof of this standard result can be found e.g.\ in \cite[Lemma 5.12]{Slade06}.
\begin{lemma}[Continuity]\label{lemmaContinuity}
Assume that, for $z\in(0,z_c)$, $\chi'(z)\le c\chi(z)^2$ for some constant $c$.
Then, the function $f$ defined in (\ref{eqDefF}) is continuous on $(0,z_c)$.
\end{lemma}
\begin{proof} It is sufficient to show that $f_1$, $f_2$ and $f_3$ are continuous.
The continuity of $f_1$ is obvious.
We show that $f_2$ and $f_3$ are continuous on the closed interval $[0,z_c-\eps]$ for any $\eps>0$ by taking derivatives with respect to $z$ and bound it uniformly in $k$ on $[0,z_c-\eps]$.

We do $f_2$ first. To this end, we consider the derivative
\begin{equation}\label{eqContinuity1}
    \frac{\d}{\d z}\frac{\Gk}{\ClK}=\frac{1}{\ClK^2}\left[\ClK\,\frac{\d\Gk}{\d z}-\Gk\,\frac{\d \hat C_{\lambda}(k)}{\d\lambda}\bigg|_{\lambda=\lambda_z}\frac{\d\lambda_z}{\d z}\right].
\end{equation}
We proceed by showing that each of the terms on the right hand side is uniformly bounded in $k$ and $z\in[0,z_c-\eps]$, and hence the derivative is bounded.
First we recall the definition of $\lambda_z$ in (\ref{eqDefLambdaZ}) to see that
\begin{equation}\label{eqContinuity2}
    \frac{1}{2}\le\frac{1}{1-\lambda_z\Dk}=\ClK\le\ClN=\chi(z).
\end{equation}
Furthermore, $\chi(z)\le\chi(z_c-\eps)$, and the latter is finite by the definition of $z_c$ in (\ref{eqDefZc}).
For every $k\in[-\pi,\pi)^d$, the two-point function is bounded from above by
\begin{equation}
|\Gk|\le|\GN|=\chi(z)\le\chi(z_c-\eps),
\end{equation}
For the derivative of the two-point function, we bound
\begin{equation}\label{eqContinuity3}
    \left|\frac\d{\d z}\Gk\right|
    =\left|\sum_x \e^{ik\cdot x} \frac\d{\d z}G_z(x)\right|
    \le \sum_x \frac\d{\d z}G_z(x)
    =\frac\d{\d z}\sum_x G_z(x)
    =\chi'(z),
\end{equation}
where the exchange in the order of sum and derivative is validated by the fact that both $\sum_x \e^{ik\cdot x} G_z(x)$ and $\sum_x G_z(x)$ are uniformly convergent series of functions.
By the assumed mean-field bound $\chi'(z)\le c\chi(z)^2$, (\ref{eqContinuity3}) is bounded above by $c\chi(z_c-\eps)^2$.

Moreover, we obtain from (\ref{eqCz3}) that
$|\d\hat C_{\lambda}(k)/\d\lambda|\le \hat C_{\lambda}(k)^2$, and, for $\lambda=\lambda_z$, this is in turn bounded by $\chi(z_c-\eps)^2$, cf.\ (\ref{eqContinuity2}).
Finally,
$|\d\lambda_z/\d z|=\chi'(z)/\chi(z)^2\le c$
by (\ref{eqDefLambdaZ}) and our assumption.

We treat $f_3$ in exactly the same way as $f_2$, and omit the details here.
\end{proof}

\subsection{Improvement of the bounds}

The following lemma covers the remaining prerequisite of Lemma \ref{lemmaBootstrap} and thus proves the final ingredient needed for the proof of Proposition \ref{theoremLaceExpansion}.

\begin{lemma}[Improvement of the bounds]\label{lemmaImprovedBounds}
If the assumptions of Proposition \ref{theoremLaceExpansion} are satisfied for some sufficiently small $\beta$,
and if $f(z)\le4$, then there exists a constant $c>0$ such that $f(z)\le1+c\beta$ for all $z\in(0,z_c)$.
In particular, if $\beta$ is small enough, then $f(z)\le3$.
\end{lemma}

The following lemma will help us for the improvement of the bound on $f_3$.
\begin{lemma}[Slade \cite{Slade06}]\label{lemmaTrigo}
Suppose that $a(x)=a(-x)$ for all $x\in\Zd$, and let
\begin{equation}\label{eqTrigo0}
    \hat A(k)=\frac{1}{1-\hat a(k)}.
\end{equation}
Then, for all $k,l\in[-\pi,\pi)^d$,
\begin{eqnarray}
  \label{eqTrigo1}
  \left|\Delta_k \hat A(l)\right|
  &\le& \left(\hat A(l-k)+\hat A(l+k)\right)\hat A(l)\,\left(\widehat{|a|}(0)-\widehat{|a|}(k)\right) \\
  \nonumber
  && {}+8\hat A(l-k)\,\hat A(l)\,\hat A(l+k) \,\left(\widehat{|a|}(0)-\widehat{|a|}(l)\right) \,\left(\widehat{|a|}(0)-\widehat{|a|}(k)\right).
\end{eqnarray}
\end{lemma}
By $\widehat{|a|}$ we denote the Fourier transform of the absolute value of $a$.
The proof of Lemma \ref{lemmaTrigo} uses several bounds on trigonometric quantities, and can be found in \cite[Lemma 5.7]{Slade06}.

\begin{proof}[Proof of Lemma \ref{lemmaImprovedBounds}]
Fix $z\in(0,z_c)$ arbitrarily and assume $f(z)\le4$. Our general strategy will be to show that $f_i$ for $i=1,2,3$ is smaller then $(1+\const\beta)$ and thus, by taking $\beta$ small, $f(z)\le3$.

The bound on $f_1$ is easy.
First note that $\lambda_z=1-\chi(z)^{-1}\le1$.
Using (\ref{eqDefLambdaZ}) along with (\ref{eqLowerCritBound})--(\ref{eqMainProof6}) and Proposition \ref{propBoundPhiPsi} (with $K=4$) we obtain
\begin{eqnarray}
    \nonumber
    f_1(z)&=&\lambda_z\left(1+\PsiN\right)-\PhiN-\PsiN\\
    \nonumber
    &\le&\lambda_z\left(1+|\PsiN|\right)+|\PhiN|+|\PsiN|\\
    \label{eqImprovedF1}
    &\le&1+3\,c_4\beta.
\end{eqnarray}

The bound on $f_2$ is slightly more involved. We write $\G_z=\Nhat/\Fhat$, with
\begin{equation}\label{eqDefNF}
    \Nhat(k)=\frac{1+\PsiK}{1+\PsiN},\qquad\Fhat(k)=\frac{1-\tau(z)\Dk-\PhiK}{1+\PsiN}.
\end{equation}
Recall from (\ref{eqCz3}) that $\ClK=[1-\lambda_z\Dk]^{-1}$ and, by (\ref{eqDefG}) and (\ref{eqDefLambdaZ}),
\begin{equation}\label{eqDefNF2}
    \lambda_z=1-\frac{1-\tau(z)-\PhiN}{1+\PsiN}.
\end{equation}
This yields
\begin{eqnarray}\label{eqBootstrap9}
  \frac{\Gk}{\ClK}
  &=& \Nhat(k)+\Gk\left[1-\lambda_z\Dk-\Fhat(k)\right],
\end{eqnarray}
where
\begin{equation*}
    1-\lambda_z\Dk-\Fhat(k)=\frac{[1-\Dk]\PsiN+[\PhiK-\PhiN]\Dk+[1-\Dk]\PhiK}{1+\PsiN}.
\end{equation*}
By taking $c_4\beta\le1/2$, we obtain the bound
\begin{equation}
\frac{1+\ell c_4\beta}{1-c_4\beta}\le1+(2\ell+2)\,c_4\beta,
\qquad\ell=0,1,2,\dots,
\end{equation}
which we use frequently below.
For example, together with Assumption \ref{assumptionDiagrammaticBounds}, it enables us to bound
\begin{equation*}
    \left|\hat N(k)\right|
    =\left|\frac{1+\PsiK}{1+\PsiN}\right|
    \le\frac{1+|\PsiK|}{1-|\PsiN|}
    \le 1+4\,c_4\beta.
\end{equation*}
Together with (\ref{eqDCbound}) we obtain in the same fashion that
\begin{eqnarray*}
    \left|1-\lambda_z\Dk-\Fhat(k)\right|
    &\le&\frac{[1-\Dk]\,|\PsiN|+|\PhiK-\PhiN|+[1-\Dk]\,|\PhiK|}{1-|\PsiN|} \\
    &\le&\frac{2c_4\beta[1-\Dk]+c_4\beta\Ci}{1-c_4\beta}
    \le12\,c_4\beta\,\Ci
\end{eqnarray*}
By our assumption that $\Gk\le4\ClK$ (which follows from $f(z)\le4$) and the above inequalities,
we can bound  (\ref{eqBootstrap9}) from above by
\begin{eqnarray}\label{eqImprovedF2}
  \left|\frac{\Gk}{\ClK}\right|
  &\le& 1+4\,c_4\beta+4\cdot12\,c_4\beta\left|\ClK\,\Ci\right|
  = 1+52\,c_4\beta.
\end{eqnarray}
for every $k\in\Td$.
This proves the bound on $f_2$.

%
%
\vspace{.5em}
It remains to show the bound on $f_3$.
In the following, we write $K$ for a positive constant, whose value may change from line to line.
Furthermore, we write
\begin{equation}\label{eqNewTrigo2}
    \Gk=\frac{\hat b(k)}{1-\hat a(k)},
    \qquad\text{where \hspace{.5em}$\hat b(k)=1+\PsiK$, \hspace{.5em}$\hat a(k)=\tau(z)\Dk+\PhiK$.}
\end{equation}
A straightforward calculation (see also \cite[(4.18)]{ChenSakai05}) shows that
\begin{equation}\label{eqNewTrigo1}
    \Delta_k\hat G_z(l)
    =\frac{\Delta_k\,\hat b(l)}{1-\hat a(l)}
     +\sum_{\sigma\in\{1,-1\}}\frac{\big(\hat a(l+\sigma k)-\hat a(l)\big)\,\big(\hat b(l+\sigma k)-\hat b(l)\big)}{\left(1-\hat a(l)\right)\left(1-\hat a(l+\sigma k)\right)}
     +\hat b(l)\,\Delta_k\!\left[\frac{1}{1-\hat a(l)}\right].
\end{equation}
We now bound all three summands in (\ref{eqNewTrigo1}), and start with the first one:
\begin{equation}\label{eqNewTrigo3}
    \left|\frac{\Delta_k\,\hat b(l)}{1-\hat a(l)}\right|
    =\left|\frac{\Delta_k\,\hat b(l)}{\hat b(l)}\right|\left|\hat G_z(l)\right|
    =\left|\frac{\Delta_k\,\hat \Psi_z(l)}{1+\hat \Psi_z(l)}\right|\left|\hat G_z(l)\right|
    \le\left|\Delta_k\,\hat \Psi_z(l)\right| 2(1+K\beta)\,\Cl(l),
\end{equation}
where the last bound uses (\ref{eqBoundPhi}) to bound the denominator, and (\ref{eqImprovedF2}).
A basic calculation shows that any function $g\colon\Zd\to\R$ with $g(x)=g(-x)$ satisfies
\begin{equation}\label{eqNewTrigo4}
    \left|\Delta_k\hat g(l)\right|\le\sum_x[1-\cos(k\cdot x)]\left|g(x)\right|,
\end{equation}
cf.\ \cite[(5.32)]{BorgsChayeHofstSladeSpenc05b}.
We apply this bound with $g(x)=\Psi_z(x)$, combine it with (\ref{eqNewTrigo3}) and (\ref{eqBoundPhi}),
and use $\hat C_{\lambda_z}(l\pm k)\ge 1/2$ and the definition of $U_{\lambda_z}(l,k)$ in (\ref{eqDefU}) to obtain
\begin{equation}\label{eqNewTrigo5}
    \left|\frac{\Delta_k\,\hat b(l)}{1-\hat a(l)}\right|
    \le K\beta\,\Ci\hat C_{\lambda_z}(l)
    \le O(\beta) \,U_{\lambda_z}(l,k).
\end{equation}

The second term in (\ref{eqNewTrigo1}) is bounded as follows.
First, since
\begin{equation}
|\e^{il\cdot x}(\e^{i(\pm k\cdot x)}-1)|\le|\sin(k\cdot x)|+1-\cos(k\cdot x),
\end{equation}
we obtain
\begin{equation}\label{eqNewTrigo6}
    \big|\hat b(l\pm k)-\hat b(l)\big|
    =\big|\hat\Psi_z(l\pm k)-\hat\Psi_z(l)\big|
    \le \sum_x|\sin(k\cdot x)|\,\big|\Psi_z(x)\big|
       +\sum_x[1-\cos(k\cdot x)]\,\big|\Psi_z(x)\big|.
\end{equation}
The second term on the right hand side of (\ref{eqNewTrigo6}) is bounded by $O(\beta)\,\Ci$;
on the first term we apply the Cauchy-Schwarz inequality and (\ref{eqBoundPhiK})--(\ref{eqBoundPhi}):
\begin{eqnarray}
    \nonumber
    \sum_x|\sin(k\cdot x)|\,\big|\Psi_z(x)\big|
    &\le& \Big(\sum_{x\neq0}|\Psi_z(x)|\Big)^{1/2}
          \Big(\sum_{x\neq0}\sin(k\cdot x)^2\,|\Psi_z(x)|\Big)\!^{1/2} \\
    \nonumber
    &\le& O(\beta)^{1/2}\,\Big(\sum_{x\neq0}[1-\cos(k\cdot x)]\,|\Psi_z(x)|\Big)^{1/2} \\
    \label{eqNewTrigo7}
    &\le& O(\beta)\,\Cl(k)^{-1/2}.
\end{eqnarray}
Furthermore,
\begin{equation}
    \hat a(l\pm k)-\hat a(l)
    =\tau(z)\left(\hat D(l\pm k)-\hat D(l)\right)+\left(\hat\Phi_z(l\pm k)-\hat\Phi_z(l)\right).
\end{equation}
In a similar fashion as (\ref{eqNewTrigo6})--(\ref{eqNewTrigo7}),
we bound $\left|\hat\Phi_z(l\pm k)-\hat\Phi_z(l)\right|\le O(\beta)\,\ClK^{-1/2}$ and
\begin{eqnarray}
  \nonumber
  \left|\hat D(l\pm k)-\hat D(l)\right|
  &\le& \Big(\sum_xD(x)\Big)^{1/2}\Big(\sum_x[1-\cos(k\cdot x)]\,D(x)\Big)^{1/2}+\sum_x[1-\cos(k\cdot x)]\,D(x) \\
  \nonumber
  &=& 1\cdot [1-\Dk]^{1/2} +[1-\Dk]\\
  \label{eqNewTrigo8}
  &\le&  2\ClK^{-1/2} +2\Ci
  \le O(1)\,\ClK^{-1/2},
\end{eqnarray}
where the last line uses (\ref{eqDCbound}).
The combination of (\ref{eqNewTrigo6})--(\ref{eqNewTrigo8}) and (\ref{eqImprovedF1}) yields
\begin{equation}\label{eqNewTrigo9}
    \big(\hat a(l\pm k)-\hat a(l)\big)\,\big(\hat b(l\pm k)-\hat b(l)\big)
    \le O(\beta)\,\Ci.
\end{equation}
On the other hand, by (\ref{eqImprovedF2})--(\ref{eqNewTrigo2}),
\begin{equation}\label{eqNewTrigo10}
    \frac{1}{1-\hat a(l+\sigma k)}
    =\frac{1}{\hat b(l+\sigma k)}\hat G(l+\sigma k)
    \le (1+O(\beta))\,\Cl(l+\sigma k),
    \quad \sigma\in\{-1,0,1\}.
\end{equation}
Combining (\ref{eqNewTrigo9}) and (\ref{eqNewTrigo10}) yields
\begin{equation}\label{eqNewTrigo11}
    \frac{\big(\hat a(l\pm k)-\hat a(l)\big)\,\big(\hat b(l\pm k)-\hat b(l)\big)}{\left(1-\hat a(l)\right)\left(1-\hat a(l\pm k)\right)}
    \le O(\beta)\,\Ci\;\Cl(l)\,\Cl(l\pm k)
    \le O(\beta)\, U_{\lambda_z}(l,k).
\end{equation}

For the third term in (\ref{eqNewTrigo1}) we argue that
$|\hat b(l)|=1+|\hat\Psi_z(l)|\le1+c_4\beta$ by our assumption on $\hat\Psi_z$.
In order to apply Lemma \ref{lemmaTrigo} to bound $\Delta_k(1-\hat a(l))^{-1}$, we estimate
\begin{equation}
    \hat A(l):=\frac{1}{1-\hat a(l)}=\frac{1}{\hat b(l)}\,\hat G_z(l)
    \le(1+2c_4\beta)\,(1+51c_4\beta)\,\Cl(l)
    \le(1+K\beta)\Cl(l)
\end{equation}
by Assumption \ref{assumptionDiagrammaticBounds} and (\ref{eqImprovedF2}),
and
\begin{eqnarray*}
  \widehat{|a|}(0)-\widehat{|a|}(k)
  &=& \sum_x\,[1-\cos(k\cdot x)] \,\big|\tau(z)D(x)+\Phi_z(x)\big|\\
  &\le& \tau(z)[1-\Dk]+\sum_x\,[1-\cos(k\cdot x)]\,\big|\Phi_z(x)\big|\\
  &\le& \left(2(1+c_4\beta)+c_4\beta\right)\Ci
  \le \, 5\,\Ci,
\end{eqnarray*}
where the last line uses again (\ref{eqDCbound}) and, as usual, requires a certain smallness of $\beta$ (here we need $c_4\beta\le1$).
Plugging these estimates into (\ref{eqTrigo1}) yields
\begin{equation}
    \Delta_k\frac{1}{1-\hat a(l)}
    \le (1+K\beta)^3\cdot 8\cdot 5^2\cdot
    \Ci\left\{\hat{C}_{\lambda_z}(l-k)\hat{C}_{\lambda_z}(l)+\hat{C}_{\lambda_z}(l)\hat{C}_{\lambda_z}(l+k)+\hat{C}_{\lambda_z}(l-k)\hat{C}_{\lambda_z}(l+k)\right\},
\end{equation}
so that finally
\begin{equation}
    \frac{|\Delta_k\hat G_z(l)|}{\Ul}
    \le (1+K\beta),
\end{equation}
as required.
In conclusion $f_3(z)\le 1+K\beta$, and thus we obtain the improved bound $f(z)\le 1+O(\beta)$.
\end{proof}

\begin{proof}[Proof of Proposition \ref{theoremLaceExpansion}]
Note first that $f$ is continuous on $(0,z_c)$ by Lemma \ref{lemmaContinuity} and the assumed mean-field bound $\chi(z)'\le\const\chi(z)^2$.
Whence the prerequisites of Lemma \ref{lemmaBootstrap} are satisfied by Lemma \ref{lemmaImprovedBounds} and the fact that $f(0)=1$.
Therefore, $f(z)\le3$ for all $z<z_c$.
Moreover, Lemma \ref{lemmaImprovedBounds} shows that, if $f\le4$, then in fact $f\le1+O(\beta)$.
Hence $f(z)\le1+O(\beta)$, uniformly for $z<z_c$.
\end{proof}

\appendix

\section{Derivation of critical exponents for percolation}\label{appendixCritExp}
\subsection{Derivation of $\gp=1$}
Aizenman and Newman \cite{AizenNewma84} prove that the triangle condition $T(z_c)<
\infty$ implies that the critical exponent $\gp$ for percolation exists, and satisfies $\gp=1$.
That is to say, they show $\chi(z)\asymp (z_c-z)^{-1}$ as $z\nearrow z_c$.
The lower bound $\gp\ge1$ in \cite[Prop.\ 3.1]{AizenNewma84} holds for any homogeneous bond percolation model.
On the other hand, the upper bound $\gp\le1$ is stated in \cite[Prop.\ 3.1]{AizenNewma84} for the nearest neighbor model only.
The aim of this section is to show how the derivation in \cite{AizenNewma84} can be extended to long range systems.

The argument requires a finite volume and range approximation in order to apply Russo's formula.
We denote by
$$\torus_r:=[-r,r]^d\cap\Zd$$
a cube of sidelength $2r+1$.
In order to achieve translation invariance, we equip the cube with periodic boundary conditions, that is, $\torus_r$ is a torus.
In \cite{AizenNewma84} free boundary conditions were used.
We write $\GrR(x,y)$ 
for the probability that the points $x$ and $y$ are connected on the torus using only bonds $\{u,v\}$ of length $|u-v|\le R$.
For $r>R$ (which we always assume), this is equivalent to removing all bonds from $\torus_r$ with length larger than $R$.
Define accordingly the \emph{restricted expected cluster size} by
\begin{equation}    \label{eqPercGamma2}
    \chiRR(z):=\sum_{x\in \torus_r}\GrR(0,x),
\end{equation}
and the \emph{restricted triangle diagram} by
\begin{equation}    \label{eqPercGamma2a}
    \nabla_{\torus_r}^{\sss (R)}(z)
    :=\sum_{\substack{v,s,t\in{\torus_r}\\|v|\le R}}D(v)\,\GrR(v,s)\,\GrR(s,t)\,\GrR(t,0).
\end{equation}

We proceed as follows.
We fix $\eps>0$ small, and first show that for $z<z_c-\eps$,
\begin{equation}\label{eqPercGamma3}
    {\left(1-\nabla_{\torus_r}^{\sss(R)}(z_c-\eps)-e_{\sss R}\right)}\;{(z_c-z-\eps)}\le\frac{1}{\chiRR(z)}-\frac{1}{\chiRR(z_c-\eps)}\le{(z_c-z-\eps)}
\end{equation}
holds uniformly in $r$ and $R$, where $e_{\sss R}=o(1)$ as $R\to\infty$.
We argue that indeed, for $z< z_c-\eps$,
\begin{equation}\label{eqHelp1}
    \lim_{R\to\infty}\lim_{r\to\infty}\chiRR(z)=\chi(z),
\end{equation}
and, for every $R>0$,
\begin{equation}\label{eqHelp1a}
    \nabla_{\torus_r}^{\sss (R)}(z_c-\eps)
    \le
    \nabla(z_c-\eps)+o(1)
    \qquad\text{as $r\ua\infty$,}
\end{equation}
where $\nabla(z)=(D\ast G_z\ast G_z\ast G_z)(0)$. Note that $\nabla(z)$ differs from $T(z)$ by the extra displacement $D$.
Then, taking $r\to\infty$ followed by $R\to\infty$, we obtain for every $\eps>0$,
\begin{equation}\label{eqPercGamma3a}
    {(1-\nabla(z_c-\eps))}\;{(z_c-z-\eps)}\le\frac{1}{\chi(z)}-\frac{1}{\chi(z_c-\eps)}\le{z_c-z-\eps}.
\end{equation}
The limit $\eps\da0$ then yields
\begin{equation}\label{eqPercGamma3b}
    {(1-\nabla(z_c))}\;{(z_c-z)}\le\frac{1}{\chi(z)}\le{z_c-z}.
\end{equation}
since $\chi(z_c-\eps)^{-1}\da0$ as $\eps\da0$.

It follows from the infrared bound (\ref{eqInfraredBound}) and (\ref{eqBeta}), together with the Cauchy-Schwarz inequality, that $\nabla(z_c)\le O(\beta^{1/2})$.
Thus (\ref{eqPercGamma3b}) implies $\gp=1$ if $\beta$ in Theorem \ref{theoremInfraredBound} is sufficiently small,
which suffices for our needs.
It is possible to extend the argument to any finite triangle diagram (rather than small triangle diagrams only) by using ultraviolet regularization, as done in \cite[Lemma 6.3]{AizenNewma84}.

We start by proving (\ref{eqPercGamma3}).
We call an (occupied or vacant) bond $(u,v)$ \emph{pivotal }for an increasing event $E$, if $E$ occurs if and only if $(u,v)$ is occupied.
A crucial tool in the proof is Russo's formula \cite[Theorem 2.25]{Grimm99}, stating that
\begin{equation}\label{eqPercGamma4}
    \frac \d{\d z} \GrR(x,y)=\sum_{\substack{(u,v)\in {\torus_r}\times {\torus_r}\\|u-v|\le R}}
    D(v-u)\;
    \PrR((u,v) \text{ is pivotal for }x\leftrightarrow y),
    \qquad x,y\in {\torus_r}.
\end{equation}
The factor $D(v-u)$ arises from the chain rule and the fact that the bond $(u,v)$ is occupied with probability $zD(v-u)$.
Since
\begin{equation}
    \{(u,v) \text{ is pivotal for }x\leftrightarrow y\}
    \subset\big\{\{x\leftrightarrow u\}\circ\{v\leftrightarrow y\}\big\}
    \cup\big\{\{x\leftrightarrow v\}\circ\{u\leftrightarrow y\}\big\},
\end{equation}
(\ref{eqPercGamma4}) and the BK-inequality \cite{BergKeste85} imply
\begin{equation}
    \frac \d{\d z} \GrR(x,y)\le\sum_{u,v\in {\torus_r}}D(v-u)\;\PrR(x\leftrightarrow u)\,\PrR(v\leftrightarrow y).
\end{equation}
Summing over $y$ yields the upper bound
\begin{eqnarray}
    \frac \d{\d z}\chiRR(z)
    \;\le\; \Big(\sum_{u\in{\torus_r}}\GrR(x,u)\Big)\Big(\sum_{v\in{\torus_r}}D(v-u)\Big)\Big(\sum_{y\in{\torus_r}}\GrR(v,y)\Big) 
    \;\le\;
    \chiRR(z)^2.\label{eqPercGamma5}
\end{eqnarray}
Therefore,
\begin{equation}\label{eqPercGamma7}
    \frac{\d}{\d z}\left[-\frac1{\chiRR(z)}\right]\le1.
\end{equation}
Integration over the interval $(z,z_c-\eps)$ yields
\begin{equation}
    \frac1{\chiRR(z)}-\frac1{\chiRR(z_c-\eps)}\le z_c-z-\eps.
\end{equation}

For the lower bound in (\ref{eqPercGamma3})
we use arguments as in \cite[Section 9.4]{Slade06} to obtain
    \setlength{\unitlength}{0.00023333in}
    \begingroup\makeatletter\ifx\SetFigFont\undefined%
    \gdef\SetFigFont#1#2#3#4#5{%
      \reset@font\fontsize{#1}{#2pt}%
      \fontfamily{#3}\fontseries{#4}\fontshape{#5}%
      \selectfont}%
    \fi\endgroup%
\begin{align}\label{eqPercGamma6}
    &\PrR((u,v) \text{ is pivotal for }x\leftrightarrow y )\nnb
    &\ge \GrR(x,u)\,\GrR(v,y)-\sum_{s,t\in{\torus_r}}\GrR(x,t)\,\GrR(t,s)\,\GrR(t,u)\,\GrR(s,v)\,\GrR(s,y).\\
    \nonumber
    &= \raisebox{-36pt}
    {\newcommand{\dashlinestretch}{30}
    \begin{picture}(6414,3272)(0,-10)
    \thicklines
    \path(6025,1651)(7375,1651)
    \path(7375,3001)(7375,301)
    \path(6025,3001)(6025,301)
    \path(3375,3001)(3375,301)
    \path(2025,3001)(2025,301)
    \put(4400,1426){\makebox(0,0)[lb]{{\SetFigFont{12}{14.4}{\rmdefault}{\mddefault}{\updefault}$-$}}}
    \put(5430,76){\makebox(0,0)[lb]{{\SetFigFont{12}{14.4}{\rmdefault}{\mddefault}{\updefault}$x$}}}
    \put(5430,1426){\makebox(0,0)[lb]{{\SetFigFont{12}{14.4}{\rmdefault}{\mddefault}{\updefault}$t$}}}
    \put(7590,1426){\makebox(0,0)[lb]{{\SetFigFont{12}{14.4}{\rmdefault}{\mddefault}{\updefault}$s$}}}
    \put(7590,76){\makebox(0,0)[lb]{{\SetFigFont{12}{14.4}{\rmdefault}{\mddefault}{\updefault}$y$}}}
    \put(7590,3001){\makebox(0,0)[lb]{{\SetFigFont{12}{14.4}{\rmdefault}{\mddefault}{\updefault}$v$}}}
    \put(5430,3001){\makebox(0,0)[lb]{{\SetFigFont{12}{14.4}{\rmdefault}{\mddefault}{\updefault}$u$}}}
    \put(3520,76){\makebox(0,0)[lb]{{\SetFigFont{12}{14.4}{\rmdefault}{\mddefault}{\updefault}$y$}}}
    \put(1430,76){\makebox(0,0)[lb]{{\SetFigFont{12}{14.4}{\rmdefault}{\mddefault}{\updefault}$x$}}}
    \put(1430,3001){\makebox(0,0)[lb]{{\SetFigFont{12}{14.4}{\rmdefault}{\mddefault}{\updefault}$u$}}}
    \put(3520,3001){\makebox(0,0)[lb]{{\SetFigFont{12}{14.4}{\rmdefault}{\mddefault}{\updefault}$v$}}}
    \end{picture}
    }
\end{align}
(The contribution to the second line in (\ref{eqPercGamma6}) with $u$ and $v$ interchanged is hidden there, but is incorporated in the next line when we sum over both, $u$ and $v$.)
With Russo's formula (\ref{eqPercGamma4}),
\begin{equation}\label{eqPercGamma9}
    \frac \d{\d z}\chiRR(z)
    \ge \chiRR(z)^2\sum_{|v|\le R}D(v)
    -\chiRR(z)^2\sum_{\substack{v,s,t\in{\torus_r}\\|v|\le R}}D(v)\,\GrR(v,s)\,\GrR(s,t)\,\GrR(t,0).
\end{equation}
Since $\sum_{v\in\Zd}D(v)=1$, the quantity $e_{\sss R}:=\sum_{|v|> R}D(v)$ is $o(1)$ as $R\to\infty$.
Recalling the definition of $\nabla_{\torus_r}^{\sss (R)}(z)$ in (\ref{eqPercGamma2a}) we arrive at
\begin{equation}\label{eqPercGamma10}
    \frac{\d}{\d z}\left[-\frac1{\chiRR(z)}\right]
    \ge (1-e_{\sss R})-\nabla_{\torus_r}^{\sss (R)}(z)
    \ge 1-\nabla_{\torus_r}^{\sss (R)}(z_c-\eps)-e_{\sss R}
\end{equation}
for $z<z_c-\eps$, and an integrated version of this proves (\ref{eqPercGamma3}).

We now consider (\ref{eqHelp1}) and fix $z<z_c-\eps$.
We write $\ErR|\Ccal|$ for the expected cluster size under the measure $\PrR$, i.e., $\ErR|\Ccal|=\chiRR(z)$.
We further denote by $\partial_R {\torus_r}:=\torus_{r+R}\setminus {\torus_r}$ the boundary of $\torus_r$ of thickness $R$. Hence,
\begin{equation}\label{eqHelp2}
    \E_{z,{\torus_{r+R}}}^{\sss(R)}|\Ccal|=\E_{z,{\torus_{r+R}}}^{\sss(R)}|\Ccal|\1_{\{0\nleftrightarrow\partial_R {\torus_r}\}}+\E_{z,{\torus_{r+R}}}^{\sss(R)}|\Ccal|\1_{\{0\leftrightarrow\partial_R {\torus_r}\}}.
\end{equation}
In the first summand, $\E_{z,{\torus_{r+R}}}^{\sss(R)}$ can be replaced by $\E_z^{\sss(R)}$ (the expected cluster size on the infinite lattice, where bonds are restricted to have length $\le R$), because the indicator guarantees $\Ccal\subset {\torus_r}$. This leads to
\begin{equation}\label{eqHelp3}
    \E_{z,{\torus_{r+R}}}^{\sss(R)}|\Ccal|=\E_z^{\sss(R)}|\Ccal|-\E_z^{\sss(R)}|\Ccal|\1_{\{0\leftrightarrow\partial_R {\torus_r}\}}+\E_{z,{\torus_{r+R}}}^{\sss(R)}|\Ccal|\1_{\{0\leftrightarrow\partial_R {\torus_r}\}}.
\end{equation}
By the tree graph bound \cite{AizenNewma84} and the monotonicity of $\E_z^{\sss(R)}|\Ccal|$ in $R$,
\begin{equation}
    \E_z^{\sss(R)}|\Ccal|^2\le\left(\E_z^{\sss(R)}|\Ccal|\right)^3
    \le\chi(z)^3,
\end{equation}
and hence the Cauchy-Schwarz inequality yields
\begin{equation}\label{eqHelp9}
    \E_z^{\sss(R)}|\Ccal|\1_{\{0\leftrightarrow\partial_R {\torus_r}\}}\le\chi(z)^{3/2}\,\P_z(0\leftrightarrow\partial_R {\torus_r})^{1/2}.
\end{equation}
For $z<z_c-\eps$, the first factor on the right is finite, and the latter vanishes as $r\to\infty$.
For the last summand in (\ref{eqHelp3}), we bound as follows:
\begin{equation}\label{eqHelp5}
    \E_{z,{\torus_{r+R}}}^{\sss(R)}|\Ccal|\1_{\{0\leftrightarrow\partial_R {\torus_r}\}}\le (2(r+R)+1)^d\;\P_{z,{\torus_{r+R}}}^{\sss(R)}(0\leftrightarrow\partial_R {\torus_r}),
\end{equation}
but, for $r>R$,
\begin{equation}\label{eqHelp6}
    \P_{z,{\torus_{r+R}}}^{\sss(R)}(0\leftrightarrow\partial_R {\torus_r})
    \le\P_{z,{\torus_{r+R}}}(|\Ccal|\ge r/R)
    \le\P_z(|\Ccal|\ge r/R)
    \le\exp\left\{-\frac{r}{2R\chi(z)^2}\right\},
\end{equation}
where in the first bound we use the fact that occupied bonds have length $\le R$ in the restricted model,
the second bound utilizes the fact that clusters on the torus are a.s.\ smaller than clusters in the infinite lattice \cite[Prop.\ 2.1]{HeydeHofst07},
and in the third bound uses \cite[Prop.\ 5.1]{AizenNewma84}.
The expression on the right hand side of (\ref{eqHelp6}) decays exponentially as $r$ increases, hence the right hand side of (\ref{eqHelp5}) vanishes and (\ref{eqHelp1}) is established
once we have shown that $\E_z^{\sss(R)}|\Ccal|\to\E_z|\Ccal|$ as $R\to\infty$.

This is done as follows.
We write $G^{\sss (R)}_z$ and $\chi^{\sss(R)}$ for the model on the infinite lattice where bonds are restricted to have length $\le R$.
Then obviously $\chi(z)\ge\chi^{\sss(R)}(z)$.
Furthermore, $$G_z(x)-G_z^{\sss(R)}(x)=\P_z\left(0\leftrightarrow x,\text{$\exists$ pivotal bond $(u,v)$ for $\{0\leftrightarrow x\}$ with $|u-v|>R$}\right),$$
hence, using the BK-inequality,
$$\chi(z)-\chi^{\sss(R)}(z)\le\chi(z)^2\left(z\sum_{v:|v|>R}D(v)\right).$$
Again, this vanishes as $R\to\infty$, because $z< z_c-\eps$ and $\sum_vD(v)=1$.

It remains to prove (\ref{eqHelp1a}).
We use again the coupling of \cite[Prop.\ 2.1]{HeydeHofst07} to write
\begin{equation}
    \P_{z_c-\eps,{\torus_{r+R}}}^{\sss(R)}(0\leftrightarrow x)
    \le \P_{z_c-\eps}(0\leftrightarrow x)
    +\P_{z_c-\eps}(0\leftrightarrow\partial_R {\torus_{r}}).
\end{equation}
Since the contribution from terms involving $\P_{z_c-\eps}(0\leftrightarrow\partial_R {\torus_{r}})$ is again exponentially small in $r$ (cf.\ (\ref{eqHelp6})), we readily obtain (\ref{eqHelp1a}).

\subsection{Derivation of $\dep=2$}
Barsky and Aizenman \cite{BarskAizen91} showed that the triangle condition implies also $\bp=1$ and $\dep=2$, where they used the general bounds $\bp\le1$ and $\dep\ge2$ due to \cite{ChayeChaye86} and \cite{AizenBarsk87}, respectively.
It should be noted, that in these references a different version of $\dep$ is considered, namely $\hatdep$ given by
\begin{equation}\label{eqDefM}
    M(z_c,h):=\sum_{k=1}^\infty[1-\e^{-kh}]\,\P_{z_c}(|\Ccal|=k)\asymp{h^{1/\hatdep}}\qquad\text{as $h\to\infty$.}
\end{equation}
The quantity $M$ is known as \emph{magnetization}.
If we consider the critical exponents in terms of slowly varying functions only (and not our stronger version $\asymp$), then the equivalence of $\dep$ and $\hatdep$ can be seen directly via a Tauberian Theorem (e.g.\ \cite[Theorem XIII.5.2]{Felle66}).

Our version of $\dep$ can be derived from (\ref{eqDefM}), as we show now for the mean-field value $\dep=2$.
In particular, we show that
\begin{equation}\label{eqMag3}
    c/\sqrt{n}\le M(z_c,1/n)\le C/\sqrt{n}, \qquad 0<c\le C<\infty,
\end{equation}
implies $\tilde c/\sqrt{n}\le \P_{z_c}(|\Ccal|\ge n) \le \tilde C/\sqrt{n}$ for certain constants $\tilde c, \tilde C\in(0,\infty)$.

For an upper bound on $\P_{z_c}(|\Ccal|\ge n)$ we bound
\begin{eqnarray}
  \P_{z_c}(|\Ccal|\ge n)
  &=& \sum_{k=n}^\infty \P_{z_c}(|\Ccal|=k)
    \le \sum_{k=n}^\infty \frac{1-\e^{-k/n}}{1-\e^{-1}}\; \P_{z_c}(|\Ccal|=k)     \nnb
  &\le& \left[1-\e^{-1}\right]^{-1}\sum_{k=1}^\infty \left[1-\e^{-k/n}\right]\P_{z_c}(|\Ccal|=k)     \nnb
  \label{eqMag4}
  &=& \left[1-\e^{-1}\right]^{-1}\,M(p_c,1/n),
\end{eqnarray}
and hence $\P_{z_c}(|\Ccal|\ge n)\le \tilde C/\sqrt{n}$
for $\tilde C=[1-\e^{-1}]^{-1}C$.

The lower bound is more involved.
For every $\eps>0$ we obtain
\begin{eqnarray*}
    \P_{z_c}(|\Ccal|\ge n)
    &\ge& \sum_{k=n}^\infty \left[1-\e^{-\eps k/n}\right]\; \P_{z_c}(|\Ccal|=k)\\
    &=& M(p_c,\eps/n)-\sum_{k=1}^{n-1} \left[1-\e^{-\eps k/n}\right]\P_{z_c}(|\Ccal|=k).
\end{eqnarray*}
We exploit $1-\e^{-x}\le x$ to bound further
\begin{equation*}
    \sum_{k=1}^{n-1} \left[1-\e^{-\eps k/n}\right]\P_{z_c}(|\Ccal|=k)
    \le \frac\eps n\sum_{k=1}^{n-1} k\,\P_{z_c}(|\Ccal|=k).
\end{equation*}
Note
\begin{equation*}
    \sum_{k=1}^{n-1} k\,\P_{z_c}(|\Ccal|=k)
    =   \sum_{k=1}^{n-1} \sum_{l=1}^k\,\P_{z_c}(|\Ccal|=k)
    =   \sum_{l=1}^{n-1} \sum_{k=l}^{n-1} \,\P_{z_c}(|\Ccal|=k)
    \le \sum_{l=1}^{n-1} \,\P_{z_c}(|\Ccal|\ge l),
\end{equation*}
whence
\begin{equation*}
    \P_{z_c}(|\Ccal|\ge n)
    \ge M(p_c,\eps/n) - \frac\eps n\sum_{k=1}^{n-1} \P_{z_c}(|\Ccal|\ge k).
\end{equation*}
We apply (\ref{eqMag4}) and compare with (\ref{eqMag3}) to obtain
\begin{equation}\label{eqMag5}
    \P_{z_c}(|\Ccal|\ge n)
    \ge\frac{c\sqrt{\eps}}{\sqrt{n}}-\frac{\eps}{n}
    \underbrace{\sum_{k=1}^{n-1} \frac{C}{[1-\e^{-1}]\,\sqrt{k}}}_{\le2C[1-\e^{-1}]^{-1}\sqrt{n}}.
\end{equation}
This proves that $\P_{z_c}(|\Ccal|\ge n)\ge \tilde c/\sqrt{n}$
with $\tilde c=c\sqrt{\eps}-2\eps C[1-\e^{-1}]^{-1}$,
and $\tilde c>0$ as long as $\eps$ is small enough.
With a modification in (\ref{eqMag5}), the argument can be extended to the case $\dep\neq2$,
but we refrain from giving this argument.

\section{Diagrammatic bounds for the Ising model}
\label{appendixBoundsIsing}
This appendix is devoted to the proof of Proposition \ref{propBoundIsing} for the Ising model.
We proceed by considering the quantities $\pi^{\sss (M)}_{\Lambda}$ $(M=0,1,2,\dots)$ defined in \cite{Sakai07},
which give rise to $\Pi_\sM^{\Lambda}$ and $R_{\sss M+1}^{\Lambda}$ by \cite[(1.12) and (1.13)]{Sakai07}:
\begin{equation}\label{eqPipi}
    \delta_{0,x}+\Pi_\sM^{\Lambda}(x)=\sum_{N=0}^{M}(-1)^N\pi^\sN_{\Lambda}(x), \qquad
    0\le \left|R_{\sss M}^{\Lambda}(x)\right|\le \tau(z)\,\sum_{u,v}\pi^{\sss {(M)}}_{\Lambda}(u)\,D(v-u)\,G(v,x).
\end{equation}
We first discuss a bound on $\pi^\sN_{\Lambda}$, and use this to prove Proposition \ref{propBoundIsing}.

\begin{prop}[Diagrammatic bounds for the Ising model]\label{propIsingDiagram}
    Suppose that, for the Ising model, $f(z)\le K$ for some $z\in(0,z_c)$, $K>1$.
    Then there exists a constant $\bar c_K>0$, such that
    \begin{equation}\label{eqIsingDiagramResult1}
        \delta_{0,N}\le\sum_x\pi^\sN_{\Lambda}(x)\le
        \begin{cases}
            1+\bar c_K\beta^2\quad &(N=0),\\
            (\bar c_K\beta)^N\quad &(N\ge1),
        \end{cases}
    \end{equation}
    and
    \begin{equation}\label{eqIsingDiagramResult2}
        \sum_x[1-\cos(k\cdot x)]\pi^\sN_{\Lambda}(x)\le
        \Ci(\bar c_K\beta)^{N\vee1},
    \end{equation}
    uniformly in $\Lambda$.
\end{prop}

This proposition is a variation of \cite[Proposition 3.2]{Sakai07}.
However, it is important that the bounds of the type $\sum_x|x|^2\pi_{\Lambda}^{\sN}(x)$ in \cite{Sakai07} have been replaced by bounds involving the factor $1-\cos(k\cdot x)$, as in (\ref{eqIsingDiagramResult2}).
This replacement is a basic philosophy for this paper.
The following heuristic reasoning explains why the factor $|x|^2$ is not sufficient in the case of infinite variance spread-out models.

By (\ref{eqAkira0}) below, $\pi^{\sss (0)}_z(x)\le G_z(x)^3$.
Let us assume that $G_z(x)\approx C_{\lambda_z}(x)$, as suggested by Theorem \ref{theoremInfraredBound}.
For $z=z_c$,
and using that $C_1(x)\approx \const/|x|^{d-(\alpha\wedge2)}$,
that would lead to
$$\sum_x|x|^2\pi_{z_c}^{\sss (0)}(x)\approx\sum_x|x|^2\frac{1}{|x|^{3(d-(\alpha\wedge2))}},$$
and this is finite if and only if $d<3(d-(\alpha\wedge2))-2$.
In particular, this suggests that for $\alpha<2$ and $2(\alpha\wedge2)<d<1+3/2 (\alpha\wedge2)$, $\sum_x|x|^2\pi_{z_c}^{\sss (0)}(x)=\infty$ but $\sum_x[1-\cos(k\cdot x)]\pi_{z_c}^{\sss (0)}(x)<\infty$. Thus, using $\sum_x|x|^2\pi_{z_c}^{\sss (0)}(x)<\infty$ as a criterion for $d>d_c$ suggests a wrong value for the critical dimension.
Rather, it appears that we must assume $\sum_x|x|^{\alpha\wedge2}\pi_{z_c}^{\sss (0)}(x)<\infty$ instead.

We first show how Proposition \ref{propIsingDiagram} implies Proposition \ref{propBoundIsing}, and afterwards discuss its proof.

\proof[Proof of Proposition \ref{propBoundIsing} subject to Proposition \ref{propIsingDiagram}]
We proceed as in the proof of \cite[Prop.\ 5.2]{BorgsChayeHofstSladeSpenc05b}.
The bounds (\ref{eqDiaBoundIsing1})--(\ref{eqDiaBoundIsing2}) follow immediately with $c_K=2\bar c_K$, where the extra $1$ in the $(N\!=\!0)$-case is compensated by the substraction of $\delta_{0,x}$, and the factor 2 comes from summing the geometric series (where we required $\beta$ small enough to ensure $\bar c_K\beta\le1/2$).
For the bounds on the remainder term $R_\sM$, we see by (\ref{eqPipi}) that
\begin{equation}
    \sum_x |R_{\sss M}^{\Lambda}(x)|
    \le K\hat\pi^{\sss {(M)}}_{\Lambda}(0)\,\chi(z).
\end{equation}
However, by (\ref{eqIsingDiagramResult1}), (\ref{eqDiaBoundIsing3}) follows if $z<z_c$ and $M=M(z)$ is so large that $(c_K\beta)^M\chi(z)\le c_K\beta$.
Finally, for (\ref{eqDiaBoundIsing4}), we use (\ref{eqSumCosBound}) below with $j=3$ to see that
\begin{equation}\label{eqBoundIsing1}
    \begin{split}
    \sum_{x\in\Zd} [1-\cos(k\cdot x)]\, |R_{\sss M}^{\Lambda}(x)|
    \le
    & {}7K[1-\Dk]\hat\pi^{\sss {(M)}}_{\Lambda}(0)\chi(z)
    +7K\left(\hat\pi^{\sss {(M)}}_{\Lambda}(0)-\hat\pi^{\sss {(M)}}_{\Lambda}(k)\right)\chi(z)\\
    &{}+7K\hat\pi^{\sss {(M)}}_{\Lambda}(0)\left(\GN-\Gk\right).
    \end{split}
\end{equation}
For the first term, we use (\ref{eqDCbound}) and (\ref{eqIsingDiagramResult1}) to bound
$$7K[1-\Dk]\hat\pi^{\sss {(M)}}_{\Lambda}(0)\chi(z)\le14K(\bar c_K\beta)^M\chi(z)\Ci.$$
For the second term, we use (\ref{eqIsingDiagramResult2}) to see that
$\hat\pi^{\sss {(M)}}_{\Lambda}(0)-\hat\pi^{\sss {(M)}}_{\Lambda}(k)\le\Ci(\bar c_K\beta)^{M\vee2}$.
Finally, for the third term in (\ref{eqBoundIsing1}), we use the upper bound on $f_3$ and the uniform bound $\ClK\le(1-\lambda_z)^{-1}=\chi(z)$ to obtain
\begin{equation}
    |\GN-\Gk|
    =\frac12|\Delta_k\GN|\le16K\Ci\left(3\left(1-\lambda_z\right)^{-2}\right)
    =48K\Ci\chi(z)^2.
\end{equation}
Together with (\ref{eqIsingDiagramResult1}), this yields the desired bound.
\qed

\vspace{.5em}
We now prove Proposition \ref{propIsingDiagram} subject to the diagrammatic bounds in \cite{Sakai07},
which will occupy the remainder of the paper.
Our proof is an adaptation of the proof of \cite[Prop.\ 3.2]{Sakai07}, with a modified bootstrap hypothesis.
In particular, the factor $|x|^2$ at various places in that proof is replaced by the factor $1-\cos(k\cdot x)$ here.
We fix $z\in(0,z_c)$ and throughout the remainder of the section omit it from the notation (e.g., we write $\tau$ for $\tau(z)$).
Also we fix some subset $\Lambda$ containing the origin.
We keep in mind that we are interested in the thermodynamic limit $\Lambda\nearrow\Zd$,
and in fact our bounds hold uniformly in $\Lambda$.
We elaborate on this after Prop.\ \ref{propDiagrammaticBoundAkira} below.
All sums below are taken over $\Zd$, unless stated otherwise.

We define the quantity
\begin{equation}\label{eqDefGtilde}
    \tilde G(x):=\tau(D\ast G)(x),
\end{equation}
and note the basic estimate
\begin{equation}\label{eqBoundGtilde}
    G(x)\le \delta_{0,x}+\tilde G(x)
\end{equation}
resulting from the random-current representation and the source switching lemma (cf.\ \citeAkira{(4.2)}).

In line with (\ref{eqBubbleCondition}), we write $B=(G\ast G)(0)=\sum_x G(x)^2$ for the \emph{bubble diagram}, and similarly $\tilde{B}=(\tilde{G}\ast \tilde{G})(0)$ for the ``non-vanishing bubble diagram''.
For the latter we bound
\begin{eqnarray}\label{eqAkira9}
\nonumber
    \tilde{B}
    &=&\tau^2 \int_\Td \left(\Dk\hat G(k)\right)^2\dk
    \le K^4 \int_\Td \left(\Dk\ClK\right)^2\dk\\
\nonumber
    &\le& 4K^4 \int_\Td \frac{\Dk^2}{[1-\Dk]^2}\dk
    \le 4K^4 \beta
\end{eqnarray}
using that $\tau=f_1(z)\le K$ and $f_2(z)\le K$ in the first line,
and (\ref{eqDCbound}) and Assumption \ref{assumptionBeta} in the second line.
On the other hand, by (\ref{eqBoundGtilde}),
\begin{equation}\label{eqAkira10}
    B=\sum_x G(x)^2=1+\sum_{x\neq0} G(x)^2\le1+\sum_x\tilde{G}(x)^2=1+\tilde{B}\le1+4K^4\beta.
\end{equation}
Furthermore, it is easy to see that, by the Cauchy-Schwarz inequality, ``open bubbles'' are bounded by a ``closed bubble'', i.e., for all $x\in\Zd$,
\begin{equation}\label{eqAkira19}
    (G\ast G)(x)=\sum_vG(v)\,G(x-v) \le B,
    \qquad (\tG\ast \tG)(x) \le \tilde B.
\end{equation}

Here is an outline of the proof.
We bound certain diagrams to be defined below in terms of $B$ and $\tilde B$.
In turn, these diagrams bound the lace expansion coefficients $\pi^{(j)}$, \cite{Sakai07}.
Hence, by exploiting (\ref{eqAkira9}) and  (\ref{eqAkira10}), we prove a sufficient decay of the lace expansion coefficients subject to $\beta$ being sufficiently small.

We now define various quantities needed to describe the bounding diagrams.
All notation is chosen consistently with \cite{Sakai07}, which provides our basic estimates.
In order to emphasize the diagrammatic structure, we write $G$ and $\tG$ with two arguments, with the understanding that $G(y,x)=G(x-y)$, and for $\tG$ appropriately.

Let
\begin{equation}\label{eqAkira13}
    \psi(y,x):=\sum_{j=0}^\infty(\tG^2)^{\ast j}(y,x)
    =\delta_{y,x}+\sum_{j=1}^{\infty} \sum_{\substack{u_0,u_1,\dots,u_{j}\in\\ \{x\}\times (\Zd)^{j-1}\times\{y\}}}
        \prod_{l=1}^{j}\tG(u_{l-1},u_l)^2
\end{equation}
denote a ``chain of bubbles'',
and
\begin{equation}
\tilde\psi(y,x)=\psi(y,x)-\delta_{y,x}.
\end{equation}
If $\beta$ is so small that $\tilde B<1/2$ (which we shall assume from now on), then a basic calculation shows that
\begin{equation}\label{eqAkira38}
    \tilde\psi:=\sup_y\sum_x \tilde\psi(y,x)\le 2\tilde B=O(\beta).
\end{equation}
Let
\begin{equation}\label{eqAkira11}
    P'^{(0)}_u(y,x):=G(y,x)^2G(y,u)\,G(u,x)=\raisebox{-12pt}{\includegraphics[scale=.3]{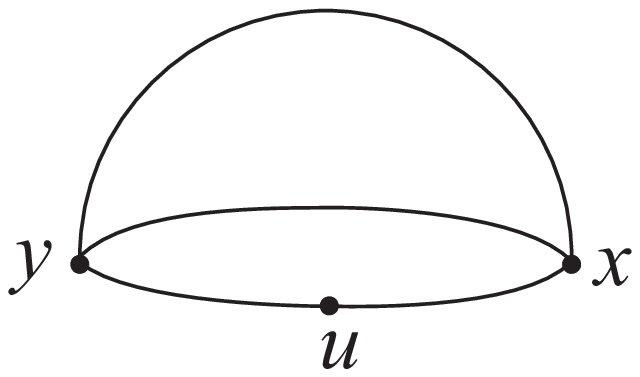}\;\;,}
\end{equation}
\begin{equation}\label{eqAkira12}
    P''^{(0)}_{u,v}(y,x):=G(y,x)\,G(y,u)\,G(u,x)\sum_{v'}G(y,v')\,G(v',x)\,\psi(v',v)=
    \raisebox{-14pt}{\includegraphics[scale=.3]{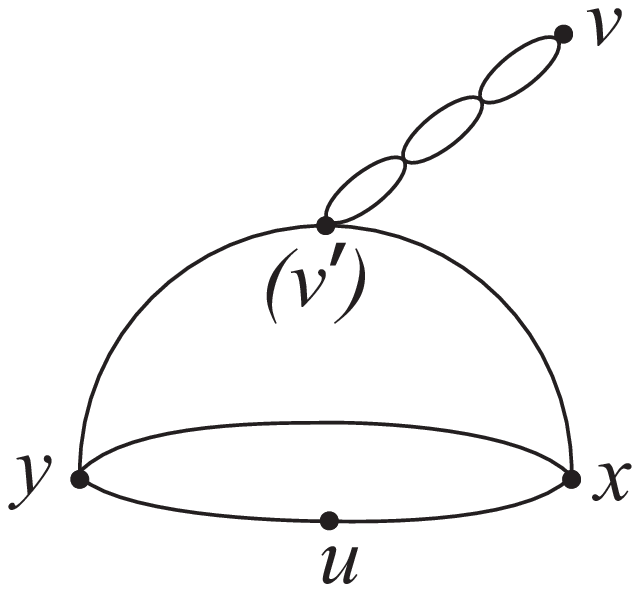}\;\;.}
\end{equation}
In the last equalities of (\ref{eqAkira11})--(\ref{eqAkira12}) we used the pictorial representation introduced in Figure \ref{figureBubble}.
Recall that a line between two points, say $y$ and $x$, represents the two-point function $G(y,x)$, and vertices in brackets are summed over.
The quantities $P'^{(0)}$ and $P''^{(0)}$ are the leading terms in the quantities $P'$ and $P''$, defined in (\ref{e:P'P''-def}) below.

We further define
\begin{equation}\label{eqAkira14}
    P^{(1)}(v_1,v'_1):=2\tilde\psi(v_1,v'_1)\,G(v_1,v'_1),
\end{equation}
and, for $j=2,3,\dots$,
\begin{equation}\label{eqAkira15}
\begin{split}
  P^{(j)}(v_1,v'_j) := \sum_{\substack{v_2,\dots,v_j\\v'_1,\dots,v'_{j-1}}}
    & G(v_1,v_2)\,G(v_2,v'_1)\left(\prod_{i=1}^{j} \tilde\psi(v_1,v'_1)\right) \\
    & \times\left(\prod_{i=2}^{j-1} G(v'_{i-1},v_{i+1})\,G(v_{i+1},v'_{i})\right)G(v_j,v'_{j-1}).
\end{split}
\end{equation}
The first three elements of the sequence look diagrammatically like
\begin{gather*}
P^{\sss(1)}(v_1,v'_1)=\raisebox{-7pt}{\includegraphics[scale=.3]{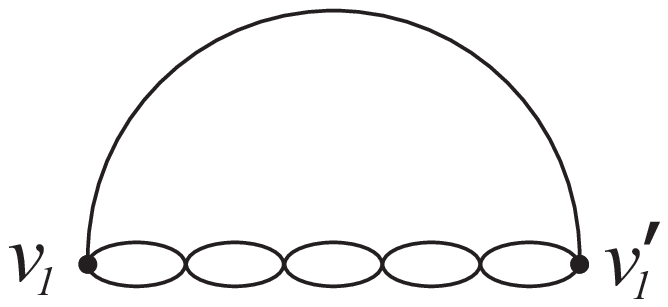}\;\;,}\qquad
P^{\sss(2)}(v_1,v'_2)=\raisebox{-15pt}{\includegraphics[scale=.3]{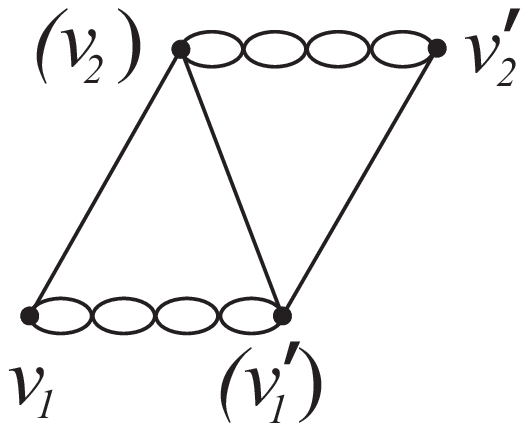}\;\;,}\qquad
P^{\sss(3)}(v_1,v'_3)=\raisebox{-15pt}{\includegraphics[scale=.3]{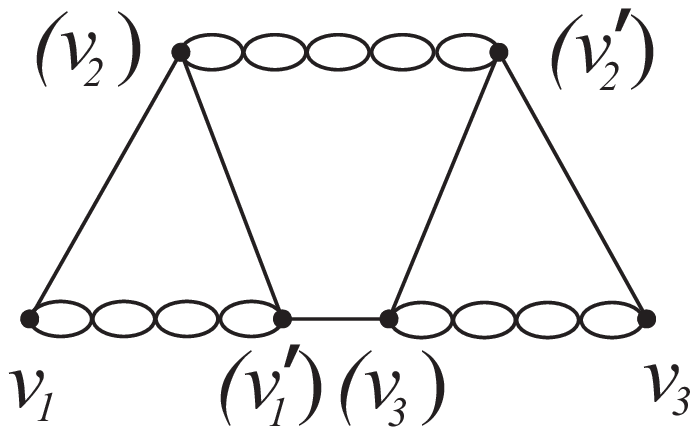}\;\;.}
\end{gather*}
Recall that vertices in brackets are summed over.

We now obtain quantities $P'$ and $P''$ as variations on $P$.
To this end, we define $P_{u}^{\prime{\sss(j)}}(v_1,v'_j)$ by
replacing one of the $2j-1$ two-point functions, say $G(z,z')$, on the right-hand
side of (\ref{eqAkira14})--(\ref{eqAkira15}) by the product of \emph{two}
two-point functions,
$G(z,u)\,G(u,z')$,
and then summing over all $2j-1$ choices of this replacement. For
example, we define
\begin{align}\lbeq{P'1-def}
    P_{u}^{\prime{\sss(1)}}(v_1,v'_1)=2\tilde\psi(v_1,v'_1)\,G(v_1,u)\,G(u,v'_1)
    =\raisebox{-14pt}{\includegraphics[scale=.3]{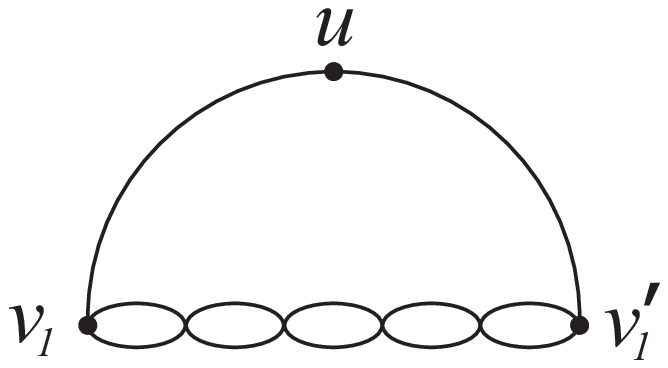}\;\;,}
\end{align}
and
\begin{align}
    P_{u}^{\prime{\sss(2)}}(v_1,v'_2)=\sum_{v_2,v'_1}
    \bigg(\prod_{i=1}^2\tilde\psi(v_i,v'_i)\bigg)
    \Big(
       G(v_1,u)\,G(u,v_2)\,G(v_2,v'_1)\,G(v'_1,v'_2)&\nonumber\\
       {}+G(v_1,v_2)\,G(v_2,u)\,G(u,v'_1)\,G(v'_1,v'_2)&\nonumber\\[7pt]
       {}+G(v_1,v_2)\,G(v_2,v'_1)\,G(v'_1,u)\,G(u,v'_2)&
    \Big).
\end{align}

We define $P_{u,v}^{\prime\prime{\sss(j)}}(v_1,v'_j)$
similarly as follows. First we take \emph{two} two-point functions
in $P^{\sss(j)}(v_1,v'_j)$, one of which (say,
$G(y_1,y'_1)$ for some $y_1,y'_1$) is
among the aforementioned $2j-1$ two-point functions, and the other
(say, $\tilde G(y_2,y'_2)$ for some $y_2,y'_2$) is among
those of which $\psi(v_i,v'_i)-\delta_{v_i,v'_i}$ for
$i=1,\dots,j$ are composed.  The product
$G(y_1,y'_1)\tilde G(y_2,y'_2)$
is then replaced by
\begin{align}
&\bigg(\sum_{v'}G(y_1,v')\,G(v',y'_1)\,\psi(v',v)\bigg)
\Big(G(y_2,u) \tilde G(u,y'_2)+\tilde G(y_2,y'_2)\,\delta_{u,y'_2}\Big)\nonumber\\
&+G(y_1,u)\,G(u,y'_1)\sum_{v'}\Big(G(y_2,v') \tilde G(v',y'_2)+
\tilde G(y_2,y'_2)\,\delta_{v',y'_2}\Big) \,\psi(v',v).
\end{align}
In our pictorial representation,
\begin{equation*}
    \raisebox{-18pt}{\includegraphics[scale=.3]{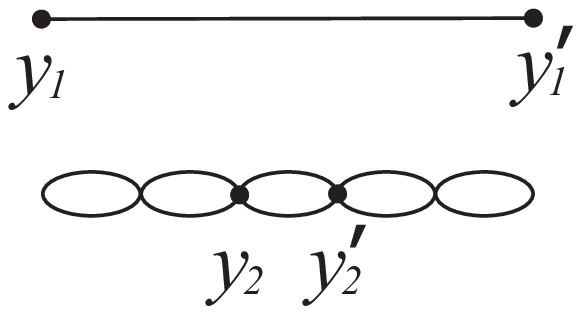}}
    \qquad\text{is replaced by}\qquad
    \raisebox{-18pt}{\includegraphics[scale=.3]{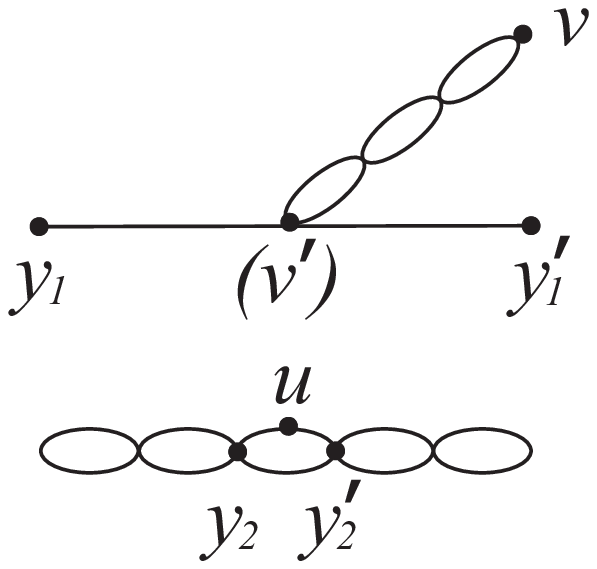}}
    \quad+\quad\raisebox{-18pt}{\includegraphics[scale=.3]{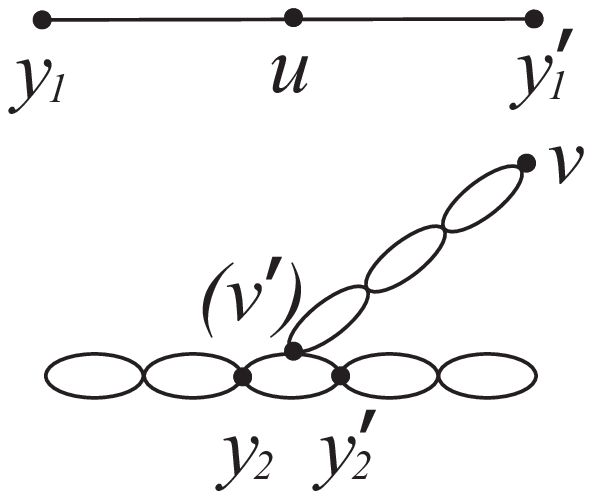}\;\;.}
\end{equation*}

\noindent
Finally, we define $P_{u,v}^{\prime\prime{\sss(j)}}(v_1,v'_j)$ by taking account
of all possible combinations of $G(y_1,y'_1)$ and $\tilde G(y_2,y'_2)$.
For example, we define $P_{u,v}^{\prime\prime{\sss(1)}}(v_1,v'_1)$
as
\begin{align}\lbeq{P''1-def}
    P_{u,v}^{\prime\prime{\sss(1)}}(v_1,v'_1)
    &=\sum_{u',u'',v'}\bigg(2\psi(v_1,u')\,\tilde G(u',u'')
     \Big(G(u',u)\,\tilde G(u,u'')+\tilde G(u',u'')\,\delta_{u,u''}\Big)\,
     \psi(u'',v'_1)\nonumber\\
    &\hspace{5em}\times G(v_1,v')\,G(v',v'_1)\psi(v',v)+(\text{permutation of $u$ and }v')
     \bigg)\\
    &=\raisebox{-14pt}{\includegraphics[scale=.3]{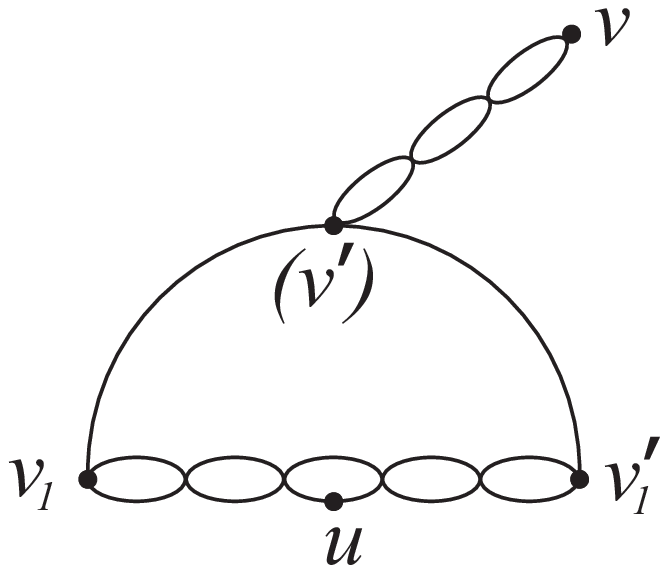}}+~
     \raisebox{-14pt}{\includegraphics[scale=.3]{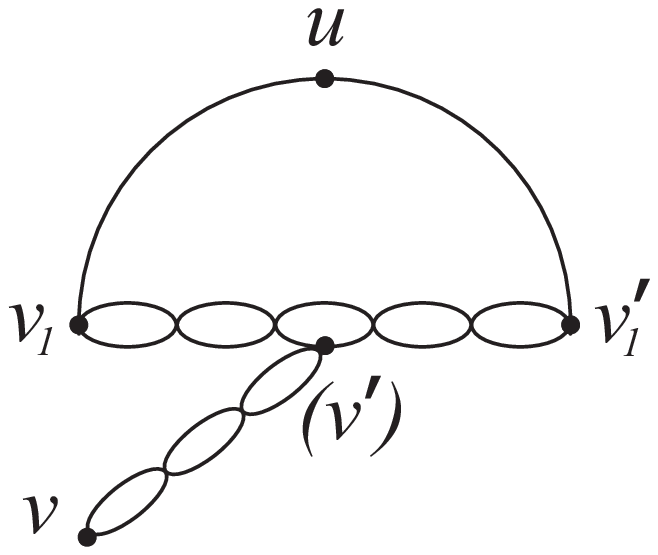}\;\;,}\nonumber
\end{align}
where the permutation term corresponds to the second diagram.

We let
\begin{align}\lbeq{P'P''-def}
P'_{u}(y,x)=\sum_{j\ge0}P_{u}^{\prime{\sss(j)}}(y,x)
    =\raisebox{-12pt}{\includegraphics[scale=.3]{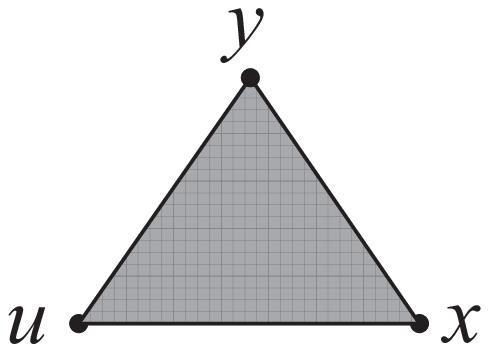}\;\;,}&&
P''_{u,v}(y,x)&=\sum_{j\ge0}P_{u,v}^{\prime\prime{\sss
 (j)}}(y,x)
    =\raisebox{-12pt}{\includegraphics[scale=.3]{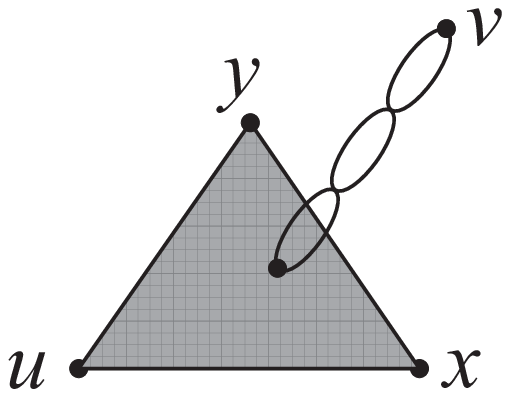}\;\;,}
\end{align}
where $P_{u}^{\prime{\sss(0)}}(y,x)$ and
$P_{u,v}^{\prime\prime{\sss(0)}}(y,x)$ are the leading
contributions to $P'_{u}(y,x)$ and $P''_{u,v}(y,x)$,
respectively.

Finally, we define
\begin{align}
Q'_{u}(y,x)&=\sum_z\big(\delta_{y,z}+\tilde G(y,z)\big)
 P'_{u}(z,x)=
 \raisebox{-12pt}{\includegraphics[scale=.3]{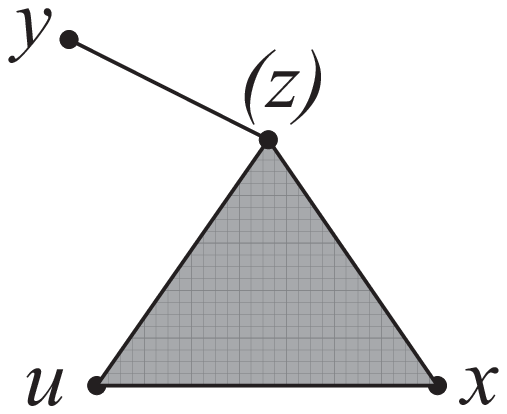}\;\;,}\lbeq{Q'-def}\\
Q''_{u,v}(y,x)&=\sum_z\big(\delta_{y,z}+\tilde G(y,z)
 \big)P''_{u,v}(z,x)\nonumber\\
&\quad+\sum_{v',z}\big(\delta_{y,v'}+\tilde G(y,v')\big)\,\tilde
 G(v',z)\,P'_{u}(z,x)\,\psi(v',v)\lbeq{Q''-def},
\end{align}
that is, pictorially,
\begin{equation}\label{eqQppPict}
    Q''_{u,v}(y,x)
    =\raisebox{-12pt}{\includegraphics[scale=.3]{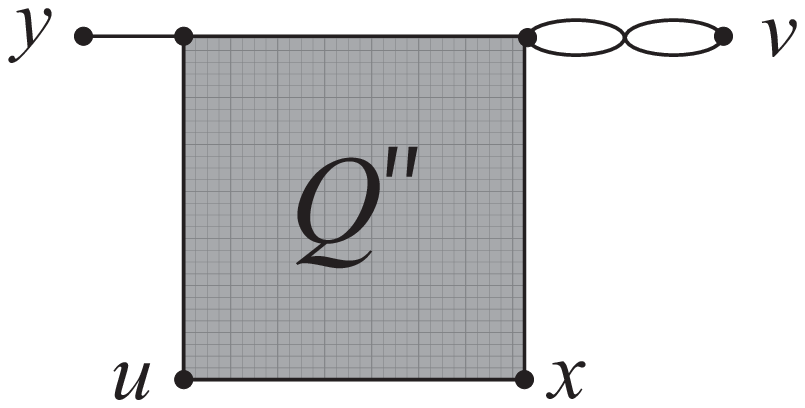}}\quad
    =\quad\raisebox{-12pt}{\includegraphics[scale=.3]{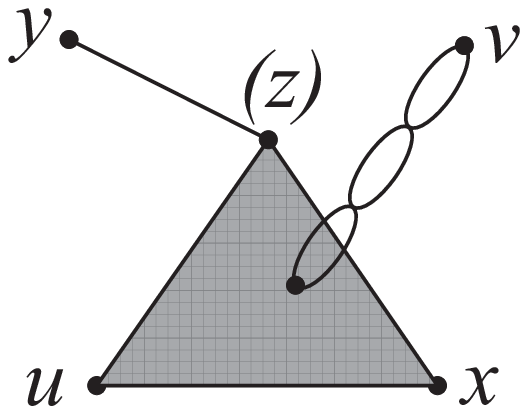}}\quad
    +\quad\raisebox{-12pt}{\includegraphics[scale=.3]{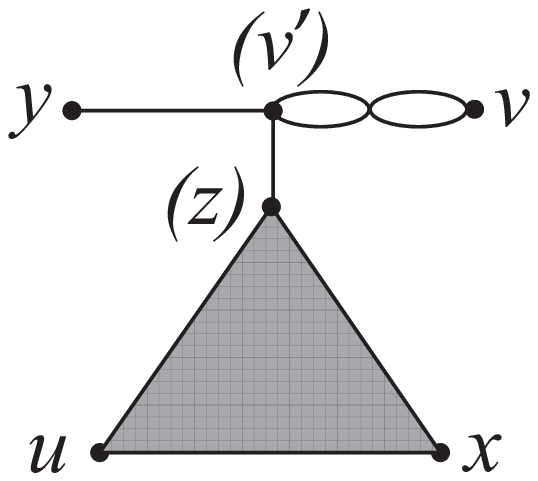}\;\;.}
\end{equation}

Based on the lace expansion, Sakai proved the following diagrammatic bound:
\begin{prop}[Diagrammatic bounds \citeAkira{Prop.\  4.1}]\label{propDiagrammaticBoundAkira}
For the ferromagnetic Ising model,
\begin{align}\label{eqDiagrammaticBound}
     \pi_\Lambda^\sN(x)\le
     \begin{cases}
        P'^{\sss(0)}_0(0,x)
          &(N=0),\\[5pt]
        \displaystyle\sum\limits_{\substack{b_1,\dots,b_j\\v_1,\dots,v_j}}
        P'^{\sss (0)}_{v_1}(0,\underline{b}_1)
        \left(
        \displaystyle\prod\limits_{i=1}^{N-1}
        \tau D(b_i)\, Q''_{v_i,v_{i+1}}(\overline{b}_{i},\underline{b}_{i+1})\right)
        \tau D(b_j)\, Q'_{v_i,v_{i+1}}(\overline{b}_{i},x)
        &(N\ge1),
     \end{cases}
\end{align}
where the sum is taken over vertices $v_i$ and (directed) bonds $b_i=(\underline{b}_{i},\overline{b}_{i})$, $i=1,\dots,j$.
We denote $D(b_i)=D(\overline{b}_{i}-\underline{b}_{i})$ and
regard the empty product as 1 by convention.
The bound (\ref{eqDiagrammaticBound}) holds uniformly in $\Lambda$.
\end{prop}

It should be noted that Sakai \cite{Sakai07} proved the bound (\ref{eqDiagrammaticBound}) on a finite graph $\Lambda$, where in particular all quantities on the right hand side are defined on $\Lambda$.
By Griffith's second inequality \cite{Griff67a}, the two-point correlation function $G_z$ is monotonically increasing in $\Lambda$, and thus so are $P'$, $Q'$ and $Q''$. Hence, the right hand side in (\ref{eqDiagrammaticBound}) is monotonically increasing in $\Lambda$, and we consider the thermodynamic limit $\Lambda\nearrow\Zd$ as a uniform upper bound on $\pi^\sN_\Lambda(x)$.
However, it is not obvious how to obtain the thermodynamic limit on the left hand side directly, since the quantities $\pi_\Lambda^\sN(x)$ are \emph{not} monotone in $\Lambda$.

\vspace{.5em}
\noindent
\emph{Proof of (\ref{eqIsingDiagramResult1}).}
We first show that $1\le \sum_x\pi^{\sss (0)}_\Lambda(x)\le1+O(\beta^2)$.
By the definition of $\pi^{\sss (0)}_\Lambda(x)$ and (\ref{eqAkira11}), $\delta_{0,x}\le\pi^{\sss (0)}_\Lambda(x)\le G(x)^3$.
Whence
\begin{equation}\label{eqAkira0}
    1\le\sum_x\pi^{\sss (0)}_\Lambda(x)\le1+\sum_{x\neq0}G(x)^3
    \le 1+\left(\sup_{x\neq0}G(x)\right)\sum_{x\neq0}\tilde G^2(x).
\end{equation}
The term $\sum_{x\neq0}\tilde G^2(x)$ is bounded above by a non-vanishing bubble $\tilde B$, yielding a factor $O(\beta)$ by (\ref{eqAkira9}).
The term $\sup_{x\neq0}G(x)$ can be bounded as follows.
We first apply (\ref{eqSingleStepBound}), to obtain
\begin{equation}\label{eqAkira5}
    \sup_{x\neq0}G(x)\le \tau\|D\|_\infty+\|\tau D\ast \tilde G\|_\infty.
\end{equation}
The first summand is bounded by $K\beta$, by our bound on $f_1$ and (\ref{eqBeta0}).
Furthermore,
$\|\tau D\ast \tilde G\|_\infty\le 4K^3\beta$
by a calculation similar to (\ref{eqAkira9}) and using $1\le2[1-\Dk]^{-1}$.
We thus obtain the bound on $\sum_x\pi^{\sss (0)}_\Lambda(x)$.

We next consider the bound on $\sum_x\pi^\sN_\Lambda(x)$ for $N\ge1$.
Here is a diagrammatic representation of the bounds on $\sum_x\pi^{\sN}_\Lambda(x)$ for $N=3$:
\begin{center}
{\includegraphics[scale=.42]{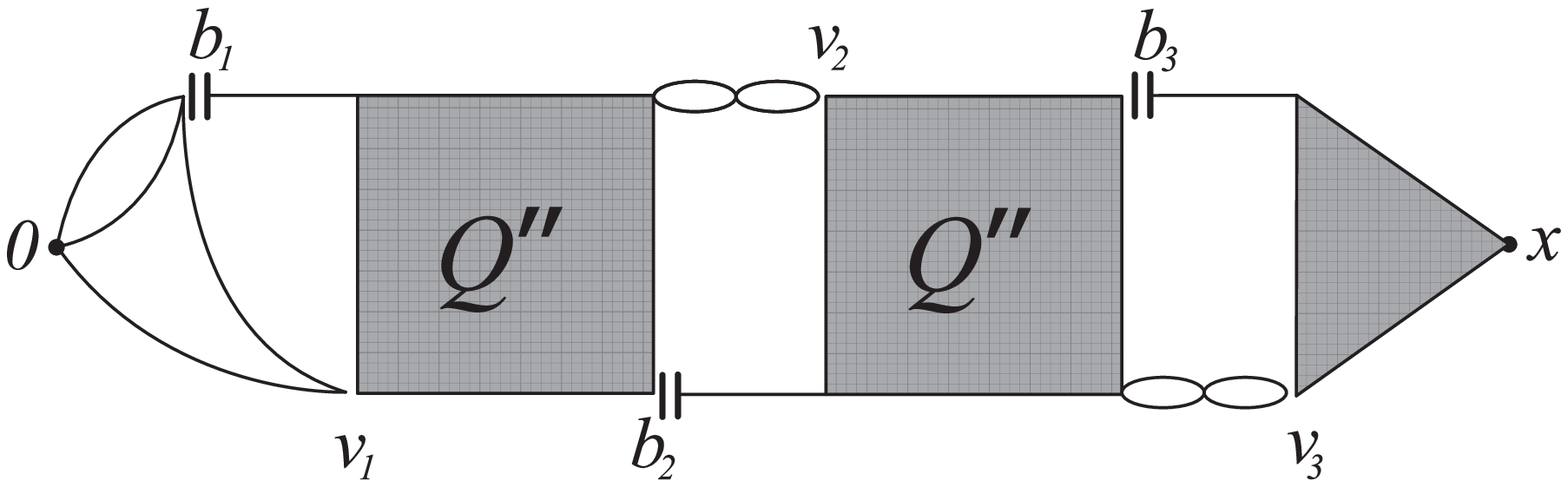}}
\end{center}
where all vertices $v_1$, $v_2$, $v_3$ and bonds $b_1$, $b_2$, $b_3$ are summed over.
Since the diagrammatic bound (\ref{eqDiagrammaticBound}) implies
\begin{equation}\label{eqAkira1}
    \sum_x\pi^\sN_\Lambda(x)
    \le
    \left(\sum_{v,x}P'^{\sss (0)}_v(0,x)\right)\!
    \Bigg(\sup_y\sum_{w,v,x}\tau D(w-y)Q''_{0,v}(w,x)\Bigg)^{\!N-1}
    \left(\sup_y\sum_{w,x}\tau D(w-y)Q'_{0}(w,x)\right),
\end{equation}
it is sufficient to show that
{\begin{enumerate}
  \item[(i)] $\sum_{v,x}P'^{\sss (0)}_v(0,x)\le O(1)$,
  \item[(ii)] $\sup_y\sum_{w,x}\tau D(w-y)Q'_{0}(w,x)\le O(\beta)$,
  \item[(iii)] $\sup_y\sum_{w,v,x}\tau D(w-y)Q''_{0,v}(w,x)\le O(\beta)$.
\end{enumerate}}
We will now prove these bounds one at a time.

(i) We first show that
    $\sum_{v,x}P'^{\sss (0)}_v(0,x)$
is uniformly bounded.
Indeed, by (\ref{eqAkira19}) and (\ref{eqAkira11}),
\begin{equation}\label{eqAkira17}
    \sum_{v,x}P'^{\sss (0)}_v(0,x)
    = \sum_{v,x} G(x)^2 G(v) G(v-x)
    \le \left(\sup_y \sum_v G(v) G(v-y)\right) \sum_{x} G(x)^2
    \le B^2.
\end{equation}

(ii) We bound
\begin{equation}\label{eqAkira3}
    \sum_{w,x}\tau D(w-y)Q'_{0}(w,x)
    =\sum_{u,x}\left(\sum_w\tau D(w-y)\big(\delta_{w,u}+\tilde G(u-w)\big)\right)P'_{0}(u,x),
\end{equation}
cf.\ (\ref{e:Q'-def}).
The factor $\beta$ comes from the nonzero line segment
$\sum_w\tau D(w-y)\big(\delta_{w,u}+\tilde G(u-w)\big)$,
as we have seen in the discussion around (\ref{eqAkira5}).

It remains to show that $\sum_{u,x}P'_{0}(u,x)=\sum_{u,x}\sum_{j=0}^\infty P'^{\sss (j)}_{0}(u,x)$ is uniformly bounded.
\begin{claim}[Bound on $P'$]\label{lemmaBoundPprime}
\begin{equation}\label{eqAkira29}
    \sum_{u,x}P'_{0}(u,x) \le O(1).
\end{equation}
\end{claim}
\proof
To this end, it suffices to show
\begin{equation}\label{eqAkira16}
    \sum_{u,x}P'^{\sss (j)}_{0}(u,x)\le (2j-1)\,O(\beta)^j,\qquad (j\ge1),
\end{equation}
since the case $j=0$ has been treated in (\ref{eqAkira17}).
The bound (\ref{eqAkira16}) will be achieved by decomposing the diagrams describing $P'^{\sss (j)}$ into bubble diagrams, and we demonstrate this for the case $j=4$ explicitly.

Recall from (\ref{eqAkira15}) that
\begin{equation}\label{eqAkira18}
P^{{\sss(4)}}(u,x)=\raisebox{-14pt}{\includegraphics[scale=.3]{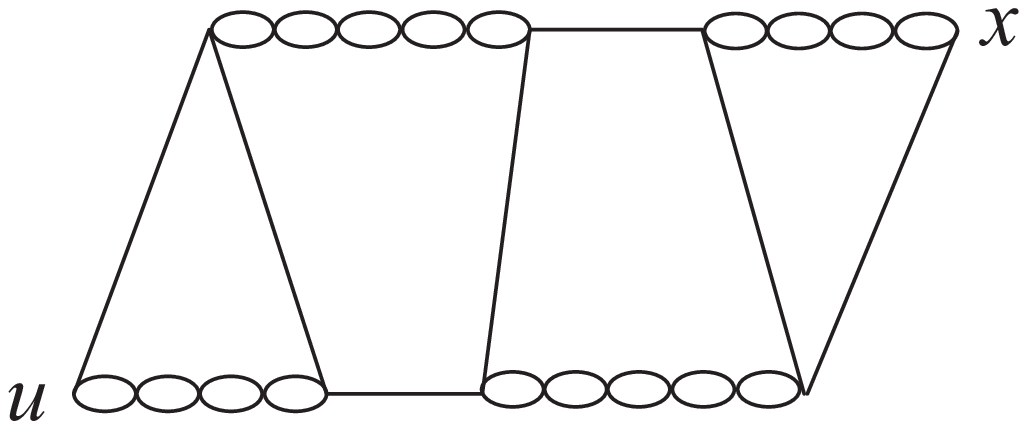}\;\;,}
\end{equation}
and we obtain $P^{\prime{\sss(4)}}(u,x)$ from $P^{{\sss(4)}}(u,x)$ by replacing one of the $7(=2j-1)$ factors of the form $G(u,v)$ by $\sum_w G(u,w)\,G(w,v)$.
In terms of diagrams, there is an extra vertex added to either of the $7$ straight lines in (\ref{eqAkira18}).
This explains the factor $(2j-1)$ in (\ref{eqAkira16}).

In case this extra vertex falls to one of the horizontal lines, say the lower one, we bound as follows.
We first extend our diagrammatical notation in the following way: we mark vertices that are summed over by a full dot, and fixed vertices (possibly with a supremum) are marked with an open dot, i.e.,
\begin{equation*}
    \sum_{u,x}\raisebox{-14pt}{\includegraphics[scale=.3]{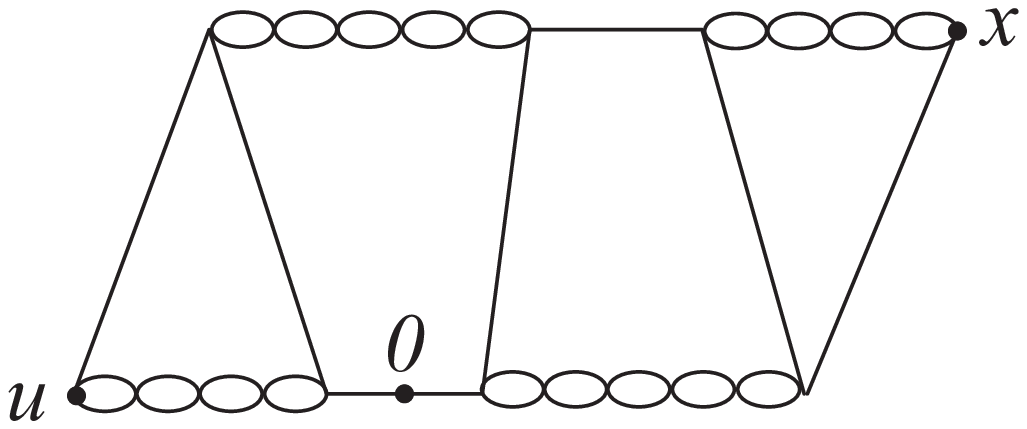}}
    =\raisebox{-14pt}{\includegraphics[scale=.3]{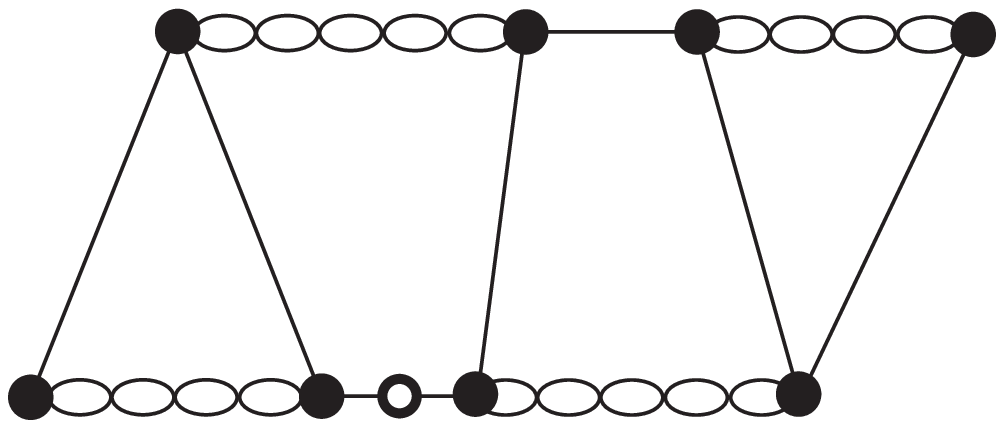}\;\;.}
\end{equation*}
By multiple use of translation invariance of the model, we obtain
\begin{eqnarray}
  \raisebox{-14pt}{\includegraphics[scale=.3]{P4p1-1}}
  &=& \sum_{\substack{x_1,x_2,x_3,x_4,\\x_5,x_6,x_7,x_8}}
    \qquad\raisebox{-20pt}{\includegraphics[scale=.3]{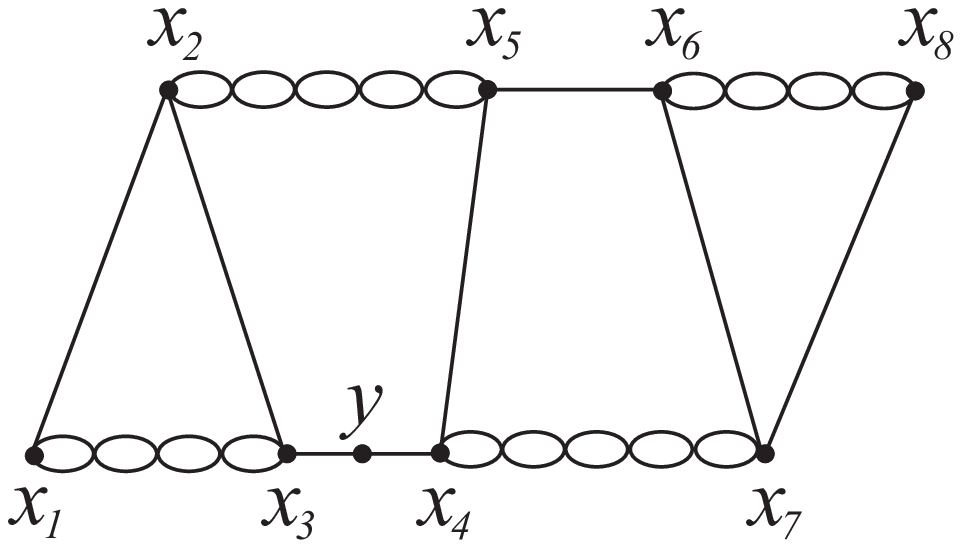}}\nnb
  &=& \sum_{\substack{x_1,x_2,x_3,x_4,\\x_5,x_6,x_7,x_8}}
    G(x_1,x_2)\,G(x_2,x_3)\,G(x_3,y)\,G(y,x_4)\,G(x_4,x_5)\,G(x_5,x_6)\nnb
  && \hspace{4em}{}\times G(x_6,x_7)\,G(x_7,x_8)\,\tilde\psi(x_1,x_3)\,\tilde\psi(x_2,x_5)\,\tilde\psi(x_4,x_7)\,\tilde\psi(x_6,x_8)\nnb
  &=& \sum_{\substack{x_1,x_2,y,x_4,\\x_5,x_6,x_7,x_8}}\cdots\text{ (expression as above with $x_3$ fixed) }\nnb
  &\le& \left(\sum_{x_1}\tilde\psi(x_1,x_3)\right)
    \left(\sup_{\bar x_1}\sum_{x_2}G(\bar x_1,x_2)\,G(x_2,x_3)\right)\nnb
  && {}\times
    \left(\sup_{\bar x_2}\sum_{x_5}\tilde\psi(\bar x_2,x_5)\right)
    \left(\sup_{\bar x_4}\sum_{y}G(x_3,y)\,G(y,\bar x_4)\right)\nnb
  && {}\times
    \left(\sup_{x_5}\sum_{x_4,x_6,x_7,x_8} \!\!\!\!G(x_4,x_5)\,G(x_5,x_6)\,G(x_6,x_7)\,G(x_7,x_8)\,\tilde\psi(x_4,x_7)\,\tilde\psi(x_6,x_8)\right)
    \nnb
  &=& \raisebox{-14pt}{\includegraphics[scale=.3]{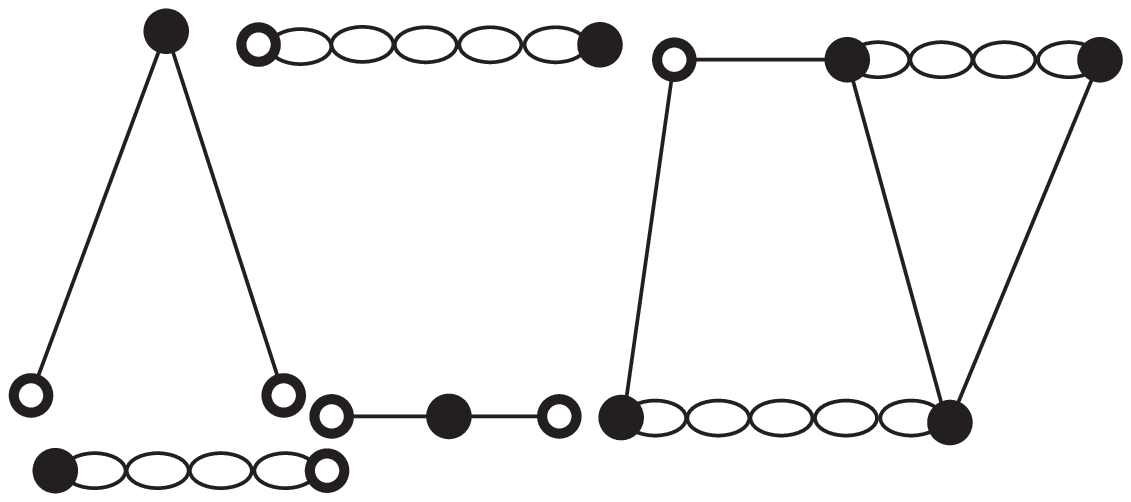}}
\end{eqnarray}
For the remaining component on the right hand side, we again use translation invariance and bound further as
\begin{equation}\label{eqAkira31}
    \raisebox{-14pt}{\includegraphics[scale=.3]{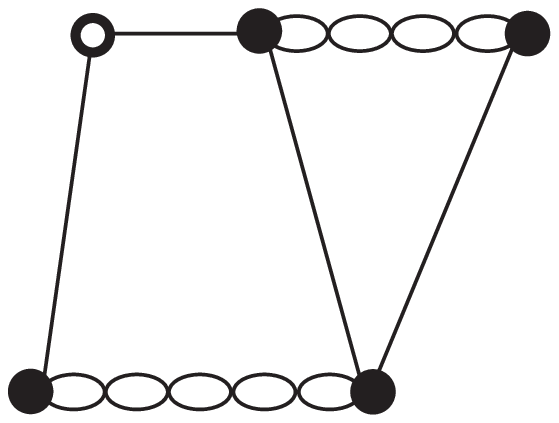}}
    =\raisebox{-14pt}{\includegraphics[scale=.3]{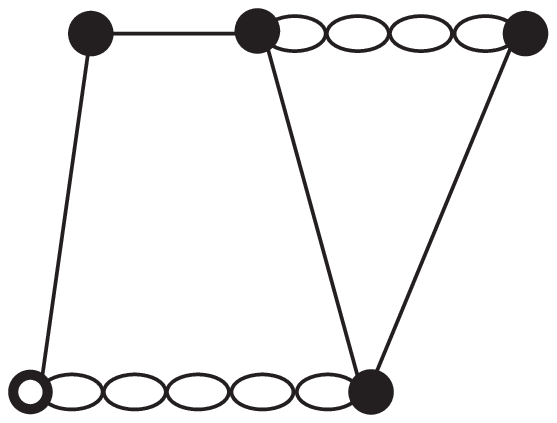}}
    \le\raisebox{-14pt}{\includegraphics[scale=.3]{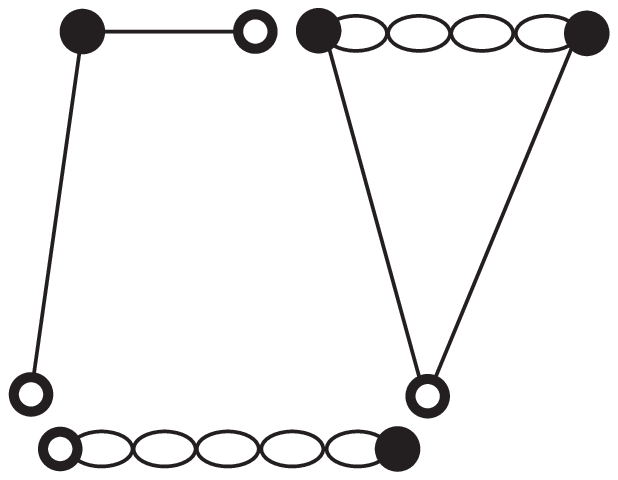}\;\;,}
    \qquad
    \raisebox{-14pt}{\includegraphics[scale=.3]{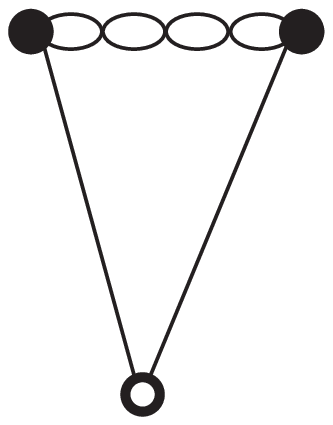}}
    =\raisebox{-14pt}{\includegraphics[scale=.3]{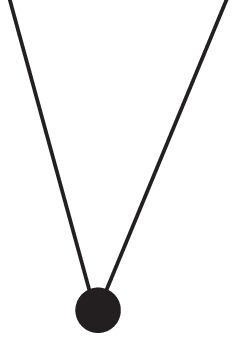}}
    \le\raisebox{-14pt}{\includegraphics[scale=.3]{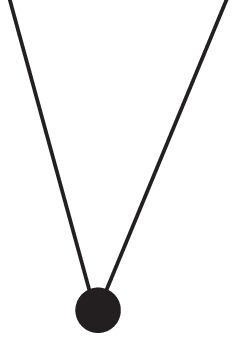}}
    \le \tilde\psi\,B.
\end{equation}
Hence,
\begin{equation}
    \raisebox{-14pt}{\includegraphics[scale=.3]{P4p1-1}}
    \le B^4\tilde\psi^4,
\end{equation}
and this can be made smaller than $O(\beta)^4$, cf.\ (\ref{eqAkira10}) and (\ref{eqAkira38}).

However, if the extra vertex falls to one of the vertical lines, then the details are slightly different:
\begin{equation}\label{eqAkira24}
    \raisebox{-14pt}{\includegraphics[scale=.3]{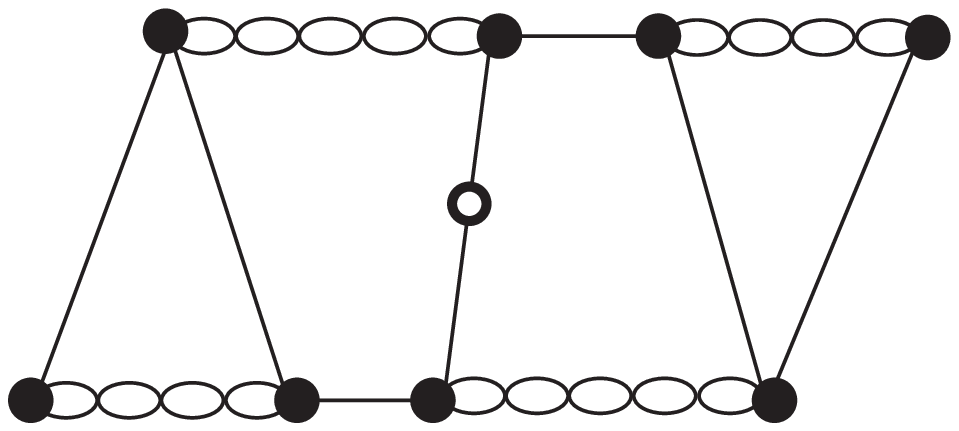}}
    =\raisebox{-14pt}{\includegraphics[scale=.3]{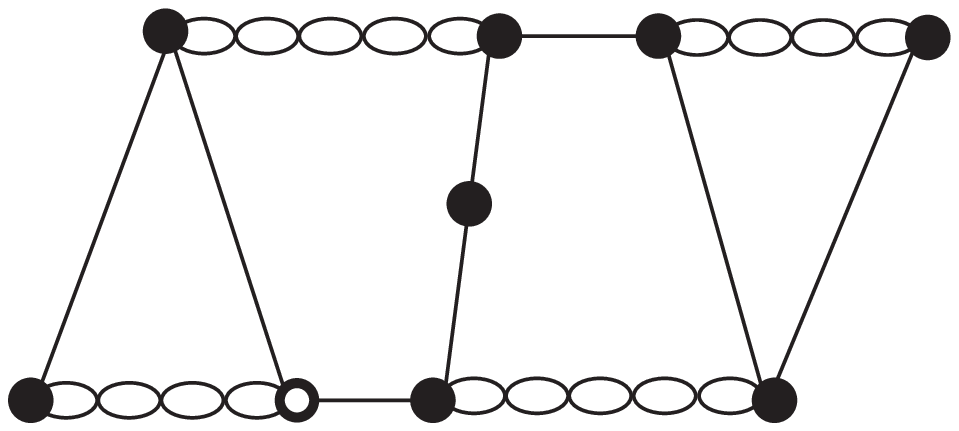}}
    \le\raisebox{-14pt}{\includegraphics[scale=.3]{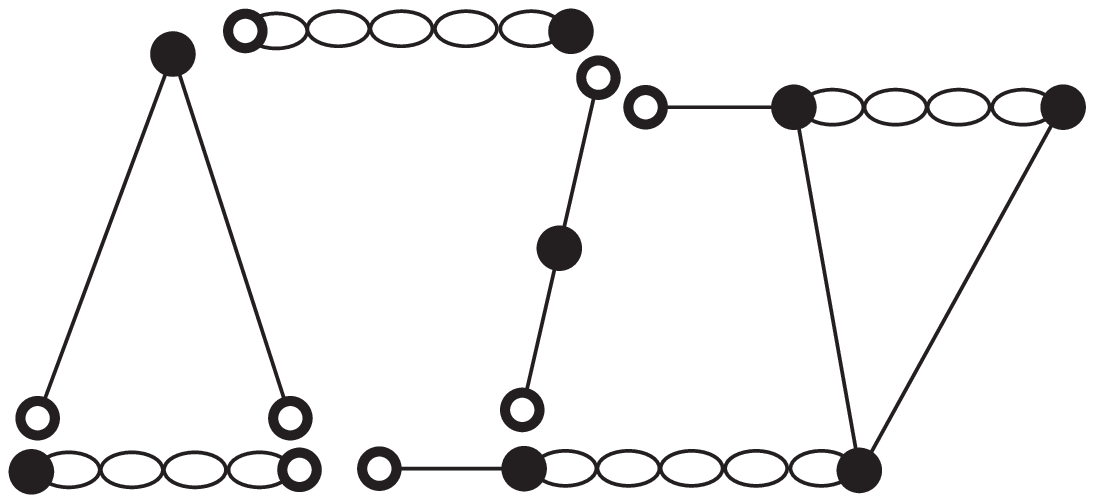}}
    \le B^2\,\tilde\psi^2\raisebox{-14pt}{\includegraphics[scale=.3]{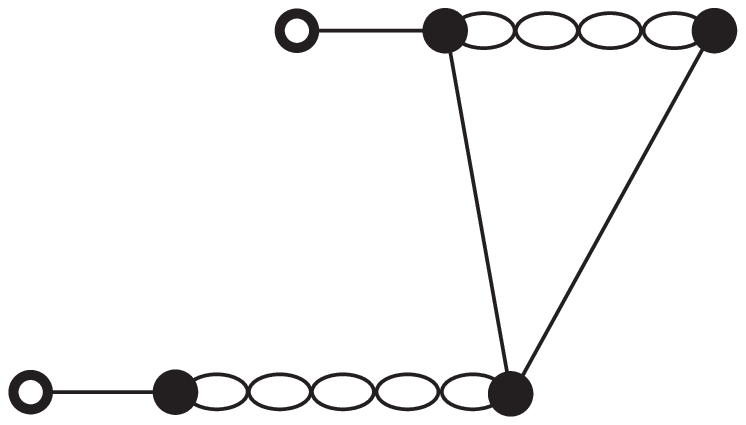}\;\;.}
\end{equation}
The remaining diagram in (\ref{eqAkira24}) is bounded by multiple use of translation invariance, as we will show now:
\begin{eqnarray}
  \raisebox{-14pt}{\includegraphics[scale=.3]{P4p2-5}}
  &=&\sup_w\sum_{v,x,y,z}\raisebox{-30pt}{\includegraphics[scale=.4]{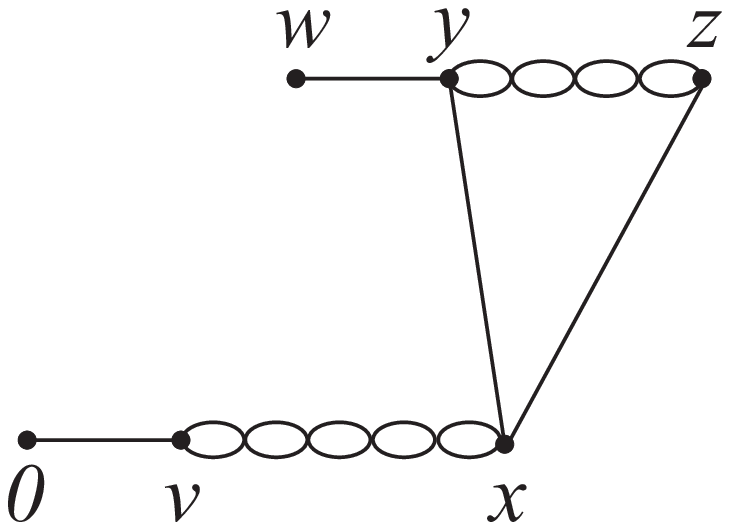}}
  = \sup_w\sum_{x,y,z,v}\raisebox{-30pt}{\includegraphics[scale=.4]{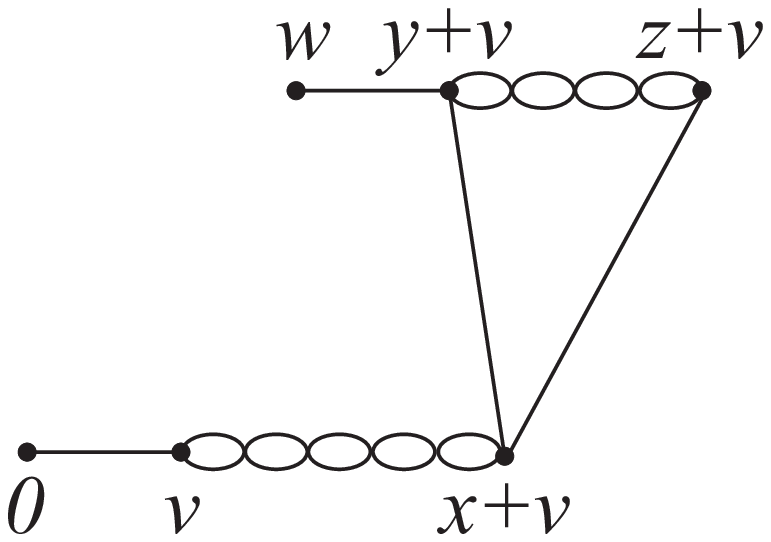}}\nnb
  \label{eqAkira30}
  &\le&
  \left(\sup_{w,y}\sum_{v}\raisebox{-12pt}{\includegraphics[scale=.4]{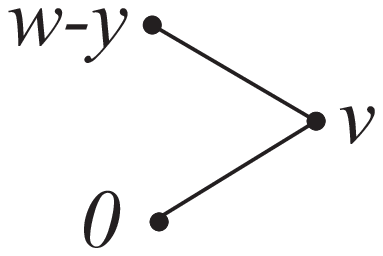}}\right)
  \left(\sum_{x,y,z}\raisebox{-30pt}{\includegraphics[scale=.4]{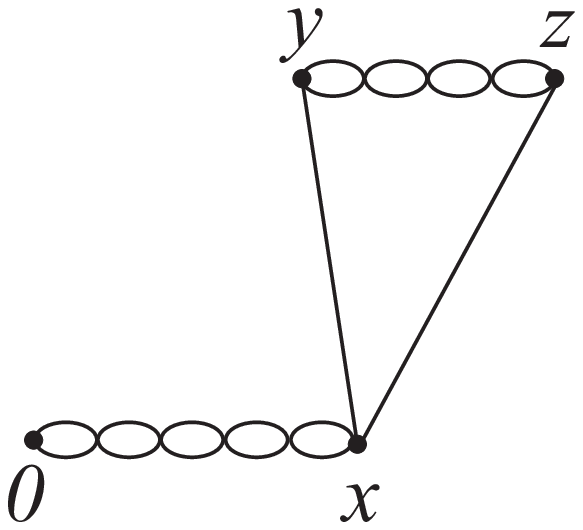}}\right)
  \le B^2\cdot\tilde\psi^2.
\end{eqnarray}

This proves (\ref{eqAkira16}) for $j=4$.
The cases $j\nin\{0,4\}$ are omitted, since the same methods will lead to the desired bounds.
\qed

\vspace{.5em}
(iii) We now turn to the bounds involving $Q''$, i.e., we prove
\begin{equation}\label{eqAkira20}
    \sup_y\sum_{w,v,x}\tau D(w-y)Q''_{0,v}(w,x)\le O(\beta).
\end{equation}
Recalling the definition of $Q''$ in (\ref{e:Q''-def}), (\ref{eqAkira20}) is established once we have shown \begin{equation}\label{eqAkira21}
    \sup_y\sum_{w,v,v',z,x}\tau D(w-y)
    \big(\delta_{w,v'}+\tilde G(w,v')\big)\,\tilde G(v',z)\,P'_{0}(z,x)\,\psi(v',v)
    \le O(\beta)
\end{equation}
and
\begin{equation}\label{eqAkira22}
    \sup_y\sum_{w,v,z,x}\tau D(w-y)
    \big(\delta_{w,z}+\tilde G(w,z)\big)P''_{0,v}(z,x)
    \le O(\beta).
\end{equation}
A decomposition of the left hand side of (\ref{eqAkira21}) yields as an upper bound
\begin{equation}\label{eqAkira23}
    \left(\sup_z\sum_{w,v'}\tau D(w-y)\big(\delta_{w,v'}+\tilde G(w,v')\big)\,\tilde G(v',z)\right)
    \left(\sup_{v'}\sum_v\psi(v',v)\right)
    \left(\sum_{z,x}P'_{0}(z,x)\right),
\end{equation}
where the first term is bounded by $O(\beta)$, the second term is bounded by $1+\tilde\psi= O(1)$ and the final term is bounded by $O(1)$, by Claim \ref{lemmaBoundPprime}.

It thus remains to show the following claim:
\begin{claim}[Bound on $P''$]\label{lemmaBoundPprimeprime}
The estimate (\ref{eqAkira22}) is true.
\end{claim}

\proof
In our pictorial representation, (\ref{eqAkira22}) can be expressed like
\begin{equation}
    \raisebox{-12pt}{\includegraphics[scale=.3]{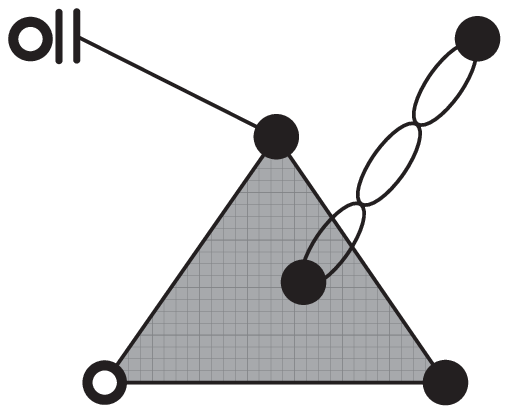}}
    \le O(\beta).
\end{equation}
Similarly to the proof of (\ref{eqAkira29}), it is sufficient to show
\begin{equation}\label{eqAkira25}
    \sup_y
    \sum_{w,v,z,x}\tau D(w-y)
    \big(\delta_{w,z}+\tilde G(w,z)\big)P''^{\sss (j)}_{0,v}(z,x)
    \le O(\beta)^{j\vee1}
\end{equation}
for $j=0,1,2,\dots$.
We explicitly perform this bound for $j=0,1$, and omit the details for $j\ge2$.

For $j=0$, we bound
\begin{equation}
    \raisebox{-12pt}{\includegraphics[scale=.3]{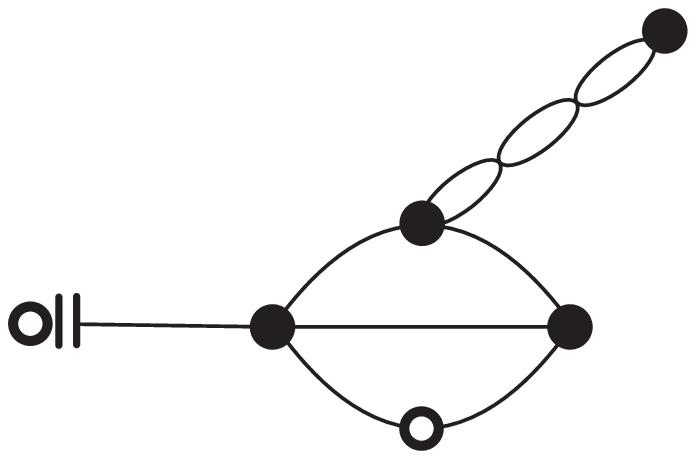}}\quad
    \le \quad\raisebox{-12pt}{\includegraphics[scale=.3]{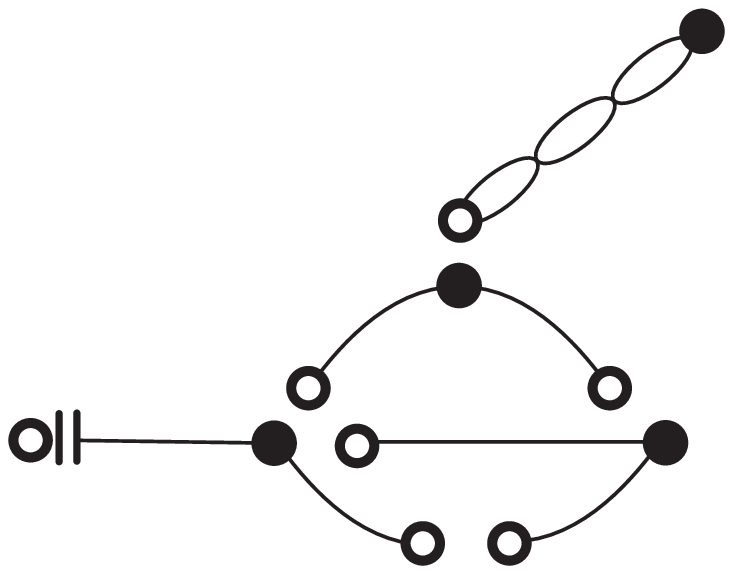}\;\;,}
\end{equation}
i.e.,
\begin{equation}
    \sup_y
    \sum_{w,v,z,x}\tau D(w-y)
    \big(\delta_{w,z}+\tilde G(w,z)\big)P''^{\sss (0)}_{0,v}(z,x)
    \le O(\beta)\,B^2\,(1+\tilde\psi),
\end{equation}
where the $O(\beta)$-factor arises from the open bubble involving the extra vertex, and the chain of bubbles hanging off from the top produces a factor $1+\tilde\psi$.

For $j=1$ we proceed similarly by recalling the definition of $P''^{\sss (1)}$ in (\ref{e:P''1-def}) and bound
\begin{eqnarray*}
    \sup_y
    \sum_{w,v,z,x}\tau D(w-y)
    \big(\delta_{w,z}+\tilde G(w,z)\big)P''^{\sss (1)}_{0,v}(z,x)
    &=&\raisebox{-12pt}{\includegraphics[scale=.3]{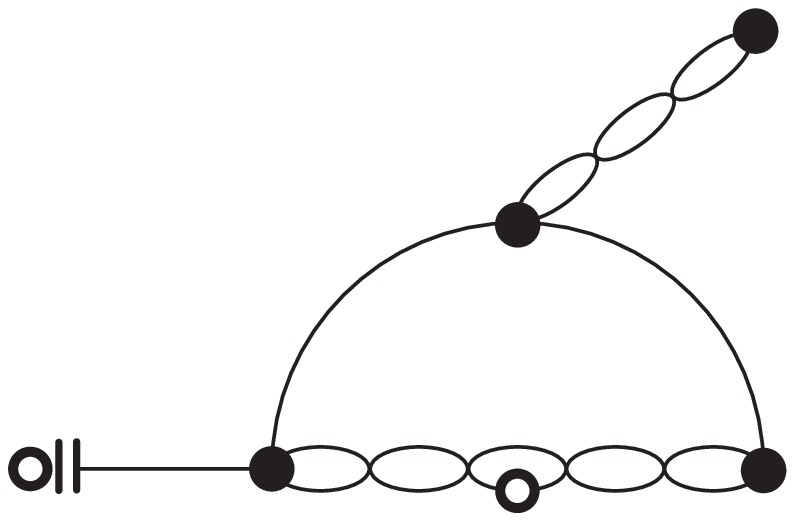}}\quad
      +\quad\raisebox{-12pt}{\includegraphics[scale=.3]{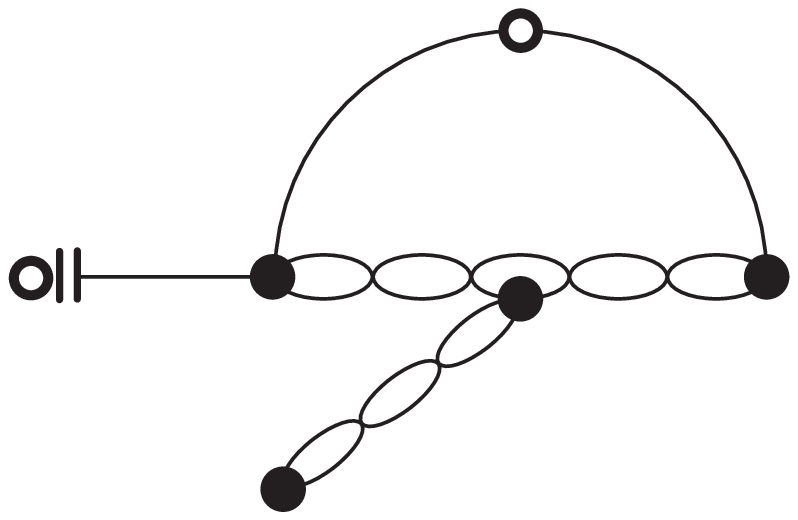}}\\
    &\le& \raisebox{-12pt}{\includegraphics[scale=.3]{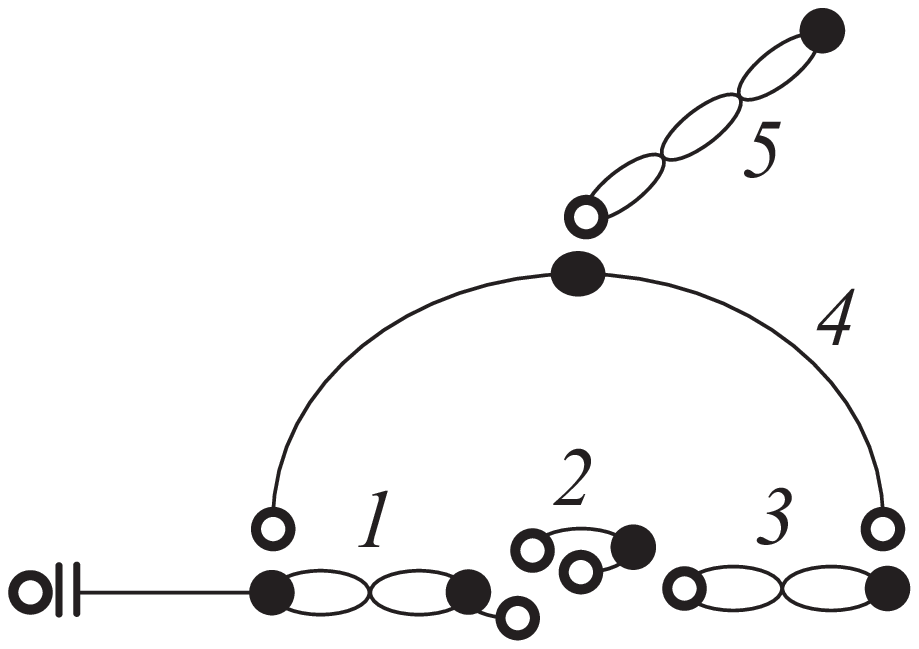}}\quad
      +\quad\raisebox{-12pt}{\includegraphics[scale=.3]{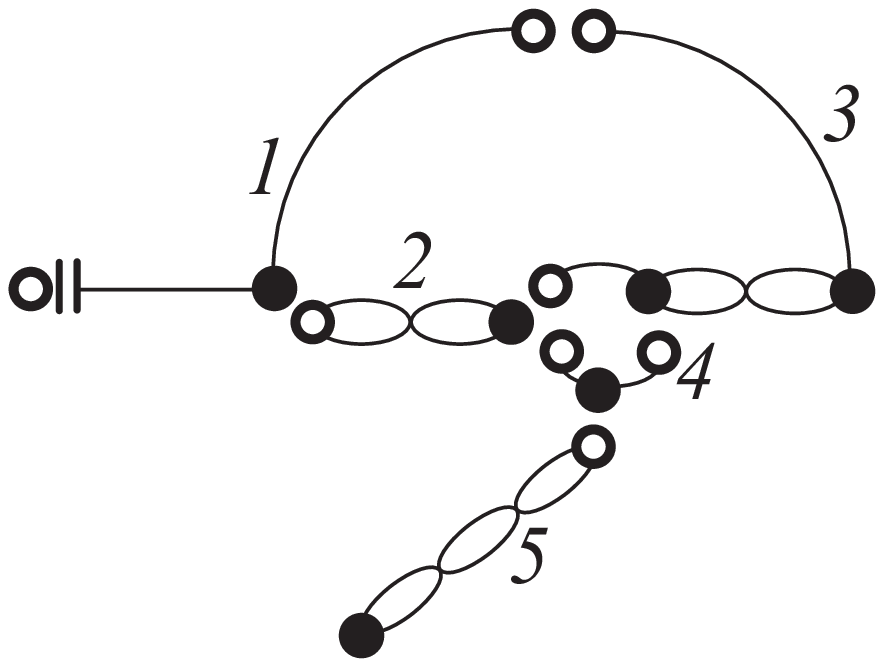}\;\;,}
\end{eqnarray*}
where the numbers indicate the order in the decomposition.
A calculation similar to (\ref{eqAkira30}) shows that
$\raisebox{1.3pt}{\includegraphics[scale=.3]{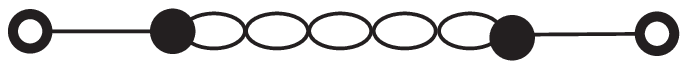}}
\le(\,\raisebox{1.3pt}{\includegraphics[scale=.3]{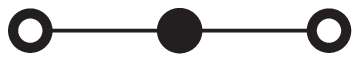}}\,)\,(\,\raisebox{1.3pt}{\includegraphics[scale=.3]{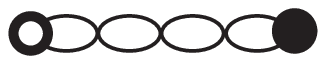}}\,)
=B(1+\tilde\psi)$ (if the initial two-point function is dashed, then we obtain $\tilde B(1+\tilde\psi)$ as an upper bound).
Hence (\ref{eqAkira25}) for $j=1$ follows.
The terms for $j\ge2$ are bounded in the same fashion.
\qed

This completes the proof of (\ref{eqIsingDiagramResult1}).

\vspace{.5em}
\noindent
\emph{Proof of (\ref{eqIsingDiagramResult2}).}
We now turn towards the proof of the bound (\ref{eqIsingDiagramResult2}) in Proposition \ref{propIsingDiagram}, which we restate here for convenience:
    \begin{equation*}
        \sum_x[1-\cos(k\cdot x)]\pi^\sN_{\Lambda}(x)\le
        \Ci(\bar c_K\beta)^{N\vee1}.
    \end{equation*}
We start by considering the case $N=0$.
By (\ref{eqDiagrammaticBound}) and (\ref{eqAkira11}),
\begin{equation}\label{eqAkira32}
    \sum_x\left[1-\cos(k\cdot x)\right]\pi^{\sss (0)}_{\Lambda}(x)
    \le \sum_{x\neq0}\left[1-\cos(k\cdot x)\right]G^3(x).
\end{equation}
This is bounded above by
\begin{equation}\label{eqAkira33}
    \left(\sup_x\left[1-\cos(k\cdot x)\right]G(x)\right)\left(\sum_{x\neq0}G^2(x)\right)
    \le \left(\sup_x\left[1-\cos(k\cdot x)\right]G(x)\right)\tilde B.
\end{equation}
Then the desired bound follows from (\ref{eqAkira10}) and the following lemma:
\begin{lemma}\label{lemmaCosKxGBound}
If for some model we have that $f(z)\le K$ for some $z\in(0,z_c)$, $K>1$, then
\begin{equation}
    \sup_x \left[1-\cos(k\cdot x)\right] G(x)
    \le 300 \,K\, \Ci (C_{\lambda_z}\ast C_{\lambda_z})(0).
\end{equation}
\end{lemma}
\noindent
Casually speaking, the multiplication by $\left[1-\cos(k\cdot x)\right]$ yields a factor $\Ci$ at the expense of adding an extra vertex in the bounding ($C$-)diagram.
In fact, we need only that $\Ci O(1)$ is an upper bound.
Although the lemma is applied to the Ising model here, it is valid for any model as long as $f_3(z)\le K$.
\proof[Proof of Lemma \ref{lemmaCosKxGBound}]
Since
\begin{eqnarray}
    \nonumber
    \sup_x \left[1-\cos(k\cdot x)\right] G(x)
    &=& \sup_x\int_{\Td} \e^{-il\cdot x}\left(\G_z(l)-\frac12\left(\G_z(l-k)+\G_z(l+k)\right)\right)
    \frac{\operatorname{d}\!l}{(2\pi)^{d}}\\
    \nonumber
    &\le&\int_{\Td} \left|\G_z(l)-\frac12\left(\G_z(l-k)+\G_z(l+k)\right)\right|\;
    \frac{\operatorname{d}\!l}{(2\pi)^{d}}\\
    &=&\int_{\Td} \left|\frac12\Delta_k\, \G(l)\right|
    \frac{\operatorname{d}\!l}{(2\pi)^{d}},
\end{eqnarray}
our bound $f_3\le K$ implies that
\begin{equation}
    \begin{split}
    &\sup_x \left[1-\cos(k\cdot x)\right] G(x)\\
    &\quad\le100K\, \Ci \int_{\Td} \left(\Cl(l-k)\,\Cl(l+k)+\Cl(l-k)\,\Cl(l)+\Cl(l)\,\Cl(l+k)\right) \frac{\operatorname{d}\!l}{(2\pi)^{d}}.
    \end{split}
\end{equation}
Denoting $C_{\lambda_z,k}(x):=\cos(k\cdot x)\,C_{\lambda_z}(x)$, we observe that
$|C_{\lambda_z,k}(x)|\le C_{\lambda_z}(x)$ and
\begin{equation}
    \hat C_{\lambda_z,k}(l) = \frac12\left(\Cl(l-k)+\Cl(l+k)\right).
\end{equation}
Hence,
\begin{eqnarray}
    \int_{\Td} \left(\Cl(l-k)\,\Cl(l)+\Cl(l)\,\Cl(l+k)\right)
    \frac{\operatorname{d}\!l}{(2\pi)^{d}}
    &=&
    2\int_{\Td} \Cl(l)\, \hat C_{\lambda_z,k}(l)
    \frac{\operatorname{d}\!l}{(2\pi)^{d}}\\
    \nonumber
    &=& 2(C_{\lambda_z}\ast C_{\lambda_z,k})(0)
    \le 2(C_{\lambda_z}\ast C_{\lambda_z})(0).
\end{eqnarray}
Furthermore,
\begin{eqnarray}
    \nonumber
    \Cl(l-k)\,\Cl(l+k)
    &=&\frac14\left[\Cl(l-k)+\Cl(l+k)\right]^2-\frac14\left[\Cl(l-k)-\Cl(l+k)\right]^2\\
    &\le& \frac14\left[\Cl(l-k)+\Cl(l+k)\right]^2
    =\hat C_{\lambda_z,k}(l)^2,
\end{eqnarray}
so that
\begin{equation}
    \int_{\Td} \Cl(l-k)\,\Cl(l+k)
    \frac{\operatorname{d}\!l}{(2\pi)^{d}}
    \le \int_{\Td} \hat C_{\lambda_z,k}(l)^2
    \frac{\operatorname{d}\!l}{(2\pi)^{d}}
    =  (C_{\lambda_z,k}\ast C_{\lambda_z,k})(0)
    \le (C_{\lambda_z}\ast C_{\lambda_z})(0).
\end{equation}
The combination of the above inequalities implies the claim.
\qed
\vspace{.5em}

For $N>0$, our strategy is to break the term $1-\cos(k\cdot x)$ into parts using
\begin{equation}\label{eqSumCosBound}
    1-\cos t\le (2N+3)\sum_{n=0}^N[1-\cos t_n]\qquad\text{for $t=\sum_{n=0}^Nt_n$}
\end{equation}
from \cite[(4.51)]{BorgsChayeHofstSladeSpenc05b},
which is reminiscent of the decomposition of squares in \citeAkira{(5.39)}.

In the case $N=1$ this allows for the following calculation.
Recall from Prop.\ \ref{propDiagrammaticBoundAkira} the upper bound on $\pi_\Lambda^{\sss (1)}(x)$.
An application of (\ref{eqSumCosBound}) for $N=1$ yields
\begin{eqnarray}\nonumber
    \sum_x [1-\cos(k\cdot x)]\pi^{\sss (1)}_{\Lambda}(x)
    &\le& \sum_x [1-\cos(k\cdot x)]\;\;\raisebox{-14pt}{\includegraphics[scale=.3]{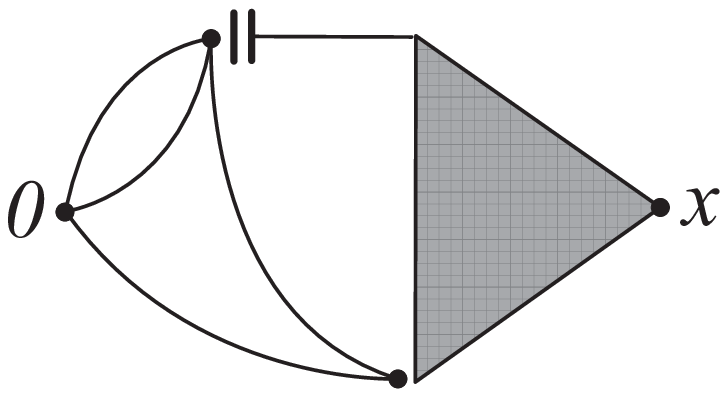}}\\
    \label{eqCosPiBd1}
    &\le& 5\left( \raisebox{-14pt}{\includegraphics[scale=.3]{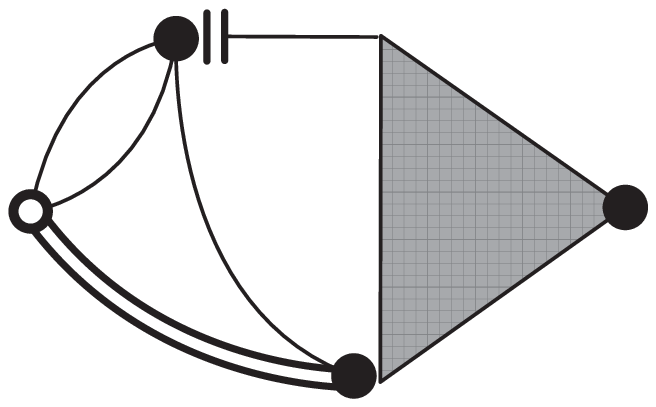}}
           + \raisebox{-14pt}{\includegraphics[scale=.3]{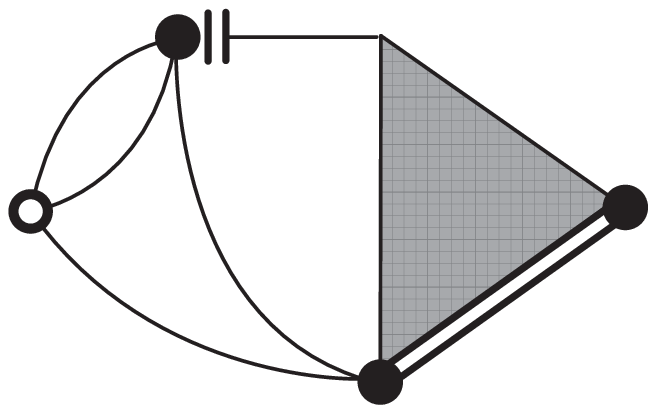}} \right).
\end{eqnarray}
In (\ref{eqCosPiBd1}) we extend our pictorial representation to incorporate factors of the form $[1-\cos(k\cdot x)]$.
Here a double line between two points, say $y_1$ and $y_2$, represents a factor $[1-\cos(k\cdot (y_1-y_2))]\,G(y_1-y_2)$,
while, as before, a normal line represents a factor $G(y_1-y_2)$.
For the second summand in (\ref{eqCosPiBd1}) , there is not a single two-point function between the two endpoints of the double line. Here our understanding is that
\begin{equation}\label{eqAkira35}
    \raisebox{-14pt}{\includegraphics[scale=.3]{pi1cos2-1}}
    =\sum_{x,y}[1-\cos(k\cdot (x-y))]\;\;\raisebox{-18pt}{\includegraphics[scale=.3]{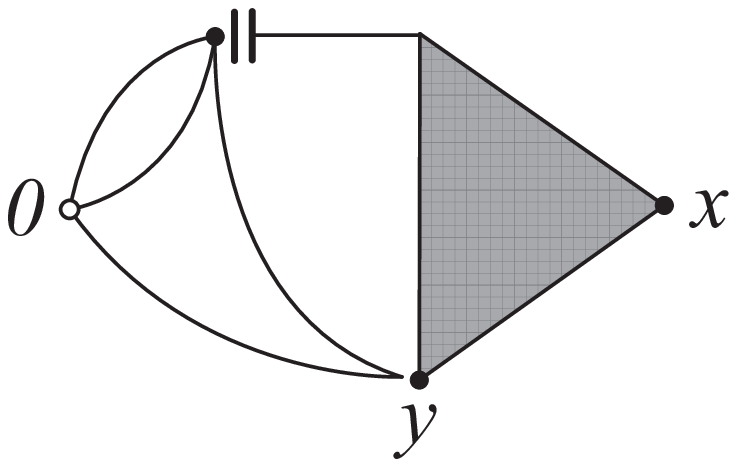}\;\;.}
\end{equation}
In other words, the double line between the two points $y$ and $x$ gives rise to the factor $[1-\cos(k\cdot (x-y))]$.

The first term in (\ref{eqCosPiBd1}) is estimated like
\begin{equation}\label{eqCosPiBd7}
\raisebox{-14pt}{\includegraphics[scale=.3]{pi1cos1-1}}
\quad\le \quad
\raisebox{-14pt}{\includegraphics[scale=.3]{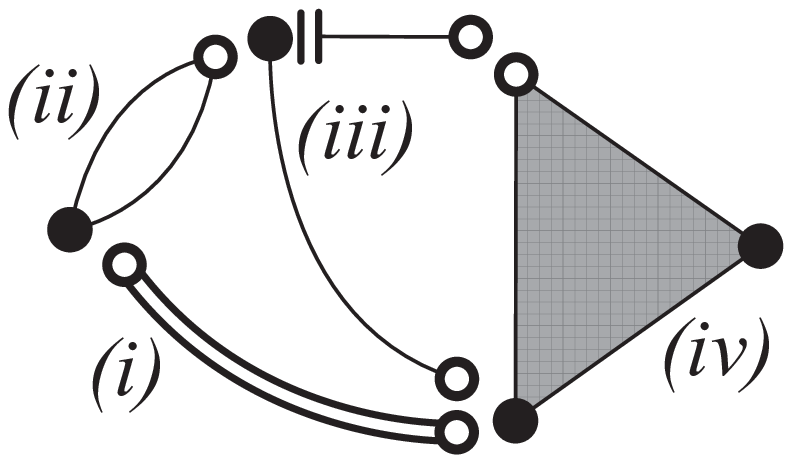}\;\;,}
\end{equation}
which yields factors $B\Ci$ arising from \emph{(i)} by Lemma \ref{lemmaCosKxGBound}, $B$ from \emph{(ii)}, $\tilde B$ from \emph{(iii)}, and $O(1)$ from \emph{(iv)} by Claim \ref{lemmaBoundPprime}.
Thus,
\begin{equation}\label{eqAkira34}
    \raisebox{-14pt}{\includegraphics[scale=.3]{pi1cos1-1}}
    \le \Ci O(\beta).
\end{equation}

For the second term in (\ref{eqCosPiBd1}) we bound
\begin{equation}\label{eqAkira36}
    \raisebox{-14pt}{\includegraphics[scale=.3]{pi1cos2-1}}
    \quad\le \quad
    \raisebox{-14pt}{\includegraphics[scale=.3]{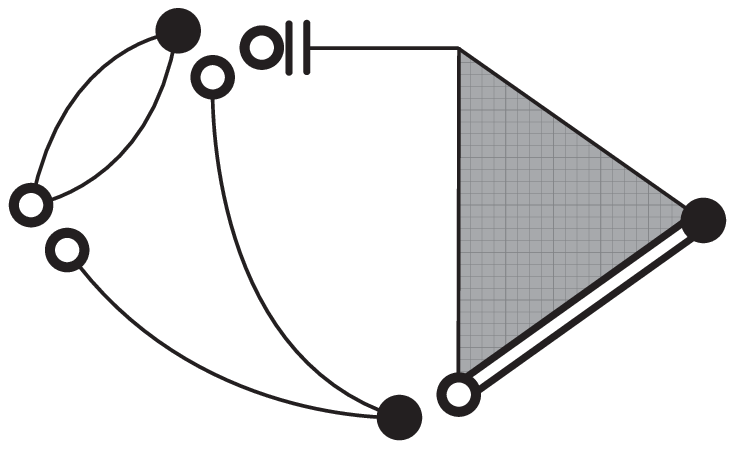}}
    \quad\le \quad
    B^2\;\cdot\;\;\raisebox{-14pt}{\includegraphics[scale=.3]{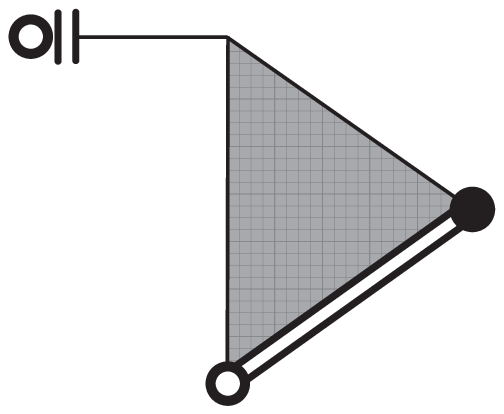}\;\;,}
\end{equation}
The remaining factor $\sup_y\sum_{w,x}[1-\cos(k\cdot x)]\,\tau D(w-y)Q'_{0}(w,x)$ is bounded by the following claim:
\begin{claim}\label{claimA}
Under the assumptions of Proposition \ref{propIsingDiagram},
\begin{equation}\label{eqClaimA1}
    \raisebox{-14pt}{\includegraphics[scale=.3]{pi1cos2-3}}
    =\sup_y\sum_{w,x}[1-\cos(k\cdot x)]\,\tau D(w-y)Q'_{0}(w,x)
    \le \Ci O(\beta).
\end{equation}
\end{claim}

\proof
By (\ref{e:Q'-def}),
\begin{equation}\label{eqCosPiBd4}
    \sup_y\sum_{w,x}[1-\cos(k\cdot x)]\,\tau D(w-y)Q'_{0}(w,x)
    =\sup_y\sum_{w,x,z}[1-\cos(k\cdot x)]\,\tau D(w-y)\sum_{j=0}^\infty
        \left(\delta_{w,z}+\tilde G_z(w,z)\right)P'^{\sss (j)}_0(z,x).
\end{equation}
In diagrams, that is
\begin{equation}\label{eqCosPiBd2}
\raisebox{-14pt}{\includegraphics[scale=.3]{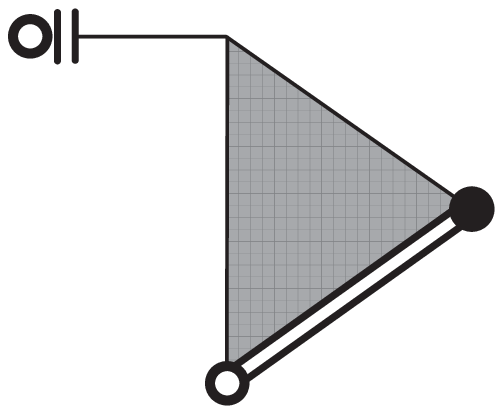}}
\quad\le \quad
\raisebox{-14pt}{\includegraphics[scale=.3]{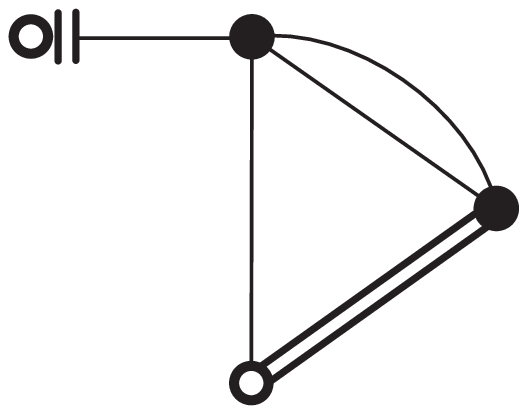}}
+\raisebox{-14pt}{\includegraphics[scale=.3]{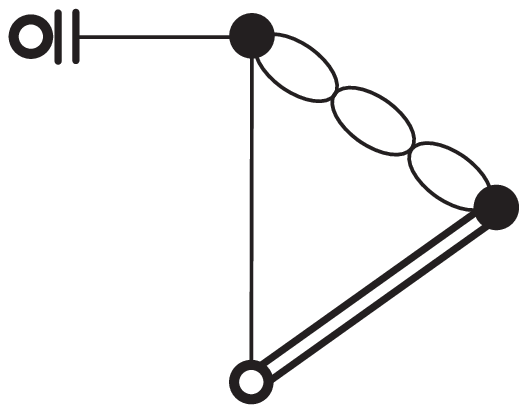}}
+\left( \hspace{1mm}\raisebox{-14pt}{\includegraphics[scale=.3]{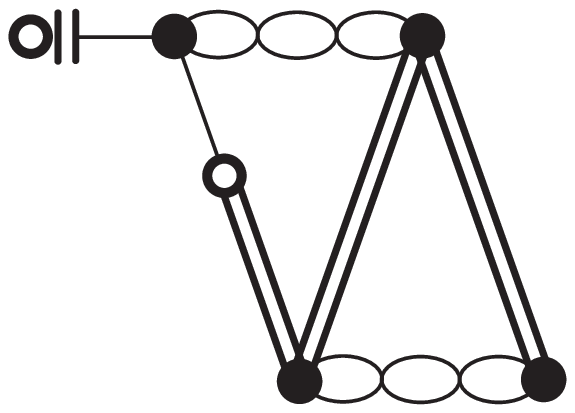}}
       +\raisebox{-14pt}{\includegraphics[scale=.3]{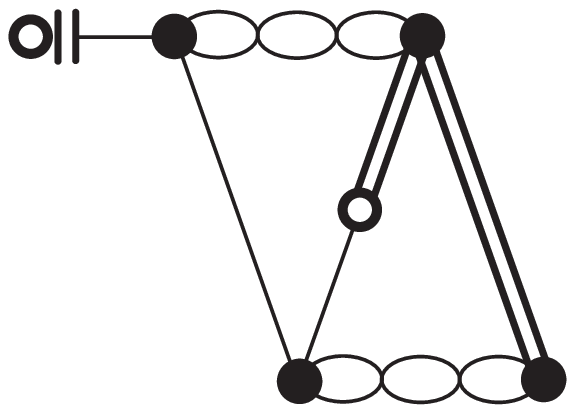}}
       +\raisebox{-14pt}{\includegraphics[scale=.3]{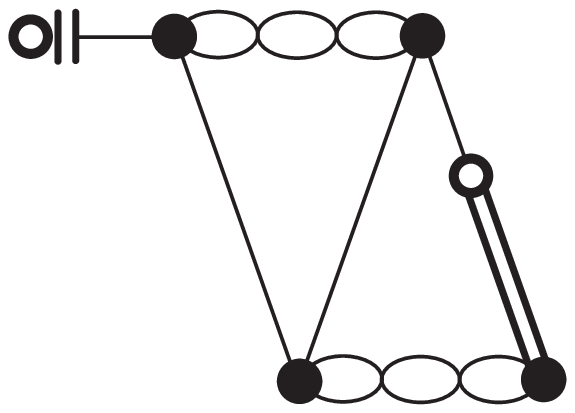}}
\right)
+\cdots,
\end{equation}
where contributions according to $j=0,1,2$ are shown explicitly and higher order contributions are indicated by dots.
When we have a series of connected double lines (like in the first term in parenthesis), this indicates a factor $[1-\cos(k\cdot (y_1-y_2))]$, where $y_1$ is the starting point of the lines, and $y_2$ is the endpoint.
We then use (\ref{eqSumCosBound}) to decompose the series of double lines.
For example, for the first term in parenthesis we obtain
$$  \raisebox{-14pt}{\includegraphics[scale=.3]{Qp-cos-2a}}
    \le 7\left(\hspace{1mm}\raisebox{-14pt}{\includegraphics[scale=.3]{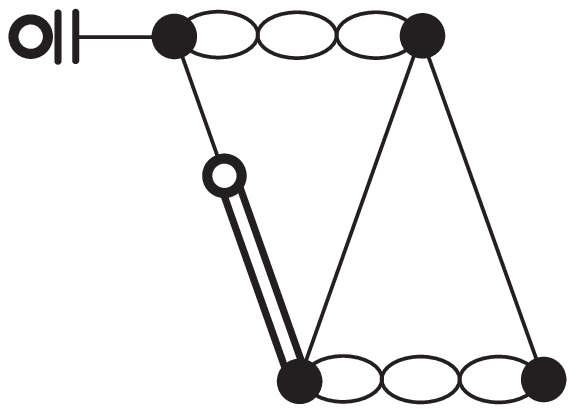}}
       +\raisebox{-14pt}{\includegraphics[scale=.3]{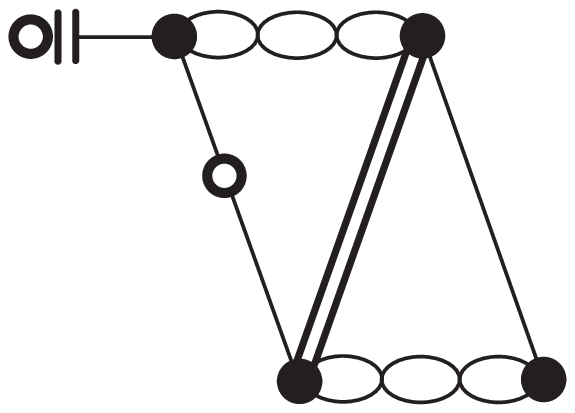}}
       +\raisebox{-14pt}{\includegraphics[scale=.3]{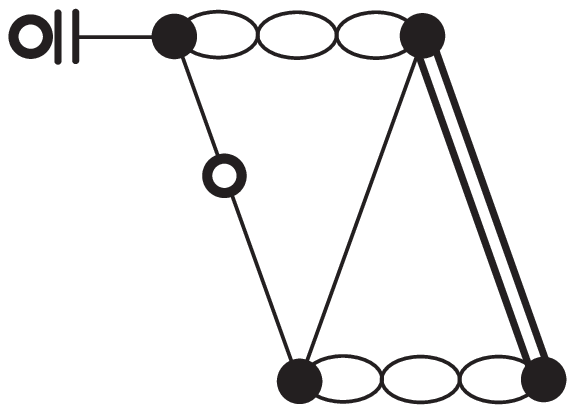}}\right),$$
and a similar bound holds for the second term.
With Lemma \ref{lemmaCosKxGBound} it follows that the contribution from $j=2$ in (\ref{eqCosPiBd2}) (the term in parenthesis) is bounded by $O(\beta)^3\Ci$.
The method can be generalized to $j\ge3$ showing
\begin{equation}\label{eqCosPiBd3}
    \sup_y\sum_{w,x,z}[1-\cos(k\cdot x)]\,\tau D(w-y)
        \left(\delta_{w,z}+\tilde G_z(w,z)\right)P'^{\sss (j)}_0(z,x)
    \le O(j^2)\,O(\beta)^{j+1}\Ci.
\end{equation}
By (\ref{eqCosPiBd4}), this is sufficient for (\ref{eqClaimA1}).
\qed
\vspace{0.5cm}

For $N>1$, we proceed by distributing the spatial displacement $1-\cos(k\cdot x)$ along the ``bottom line'' of the diagram. E.g., for $N=3$, this yields
\begin{eqnarray}\nonumber
    \sum_x [1-\cos(k\cdot x)]\pi^{\sss (3)}_{\Lambda}(x)
    &\le& \sum_x [1-\cos(k\cdot x)]\;\;\raisebox{-14pt}{\includegraphics[scale=.3]{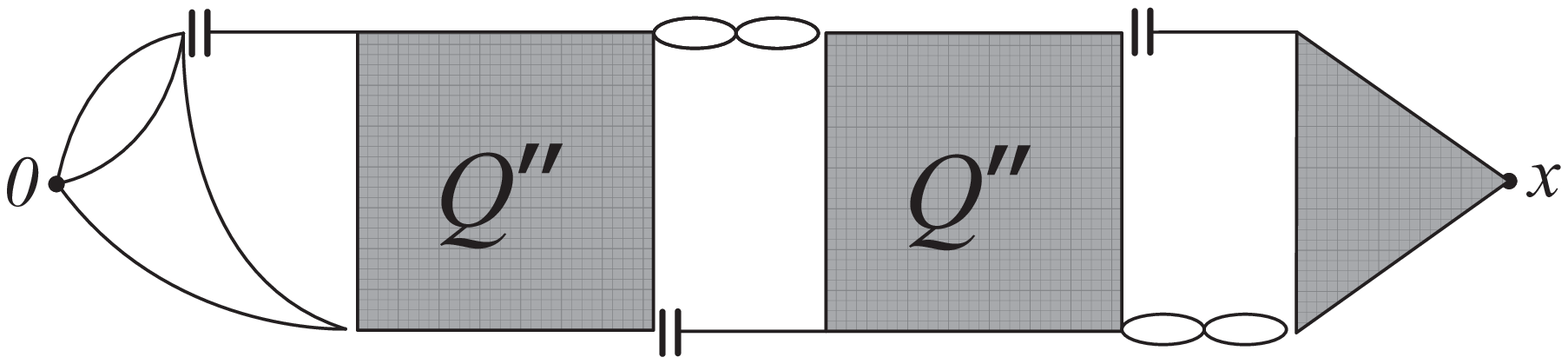}}\\
    \label{eqCosPiBd5}
    &=& \;\;\raisebox{-14pt}{\includegraphics[scale=.3]{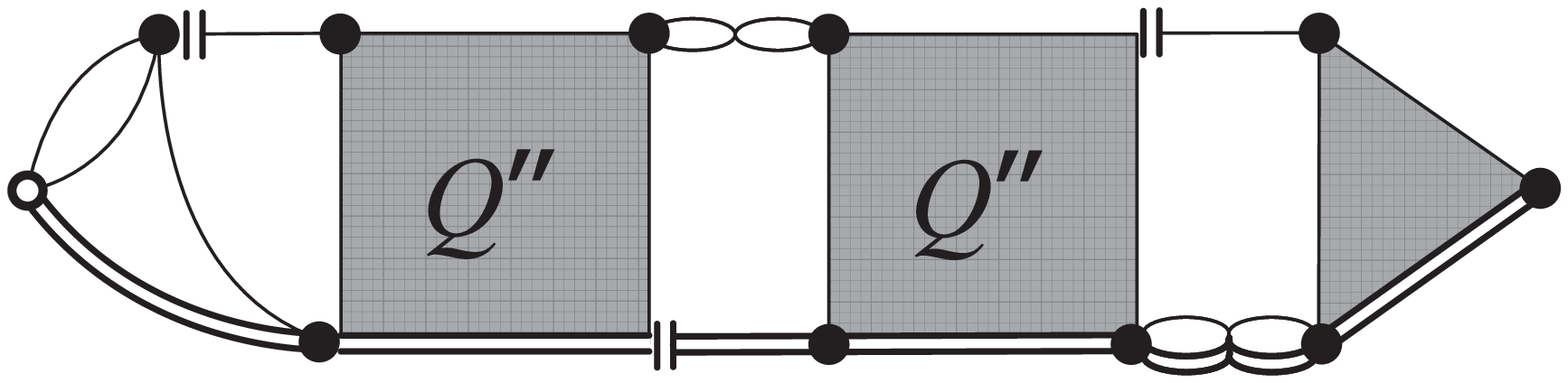}\;\;.}
\end{eqnarray}
By (\ref{eqSumCosBound}), the right hand side of (\ref{eqCosPiBd5}) is bounded above by $9$ times
\begin{align}\label{eqCosPiBd6}
    & \raisebox{-14pt}{\includegraphics[scale=.3]{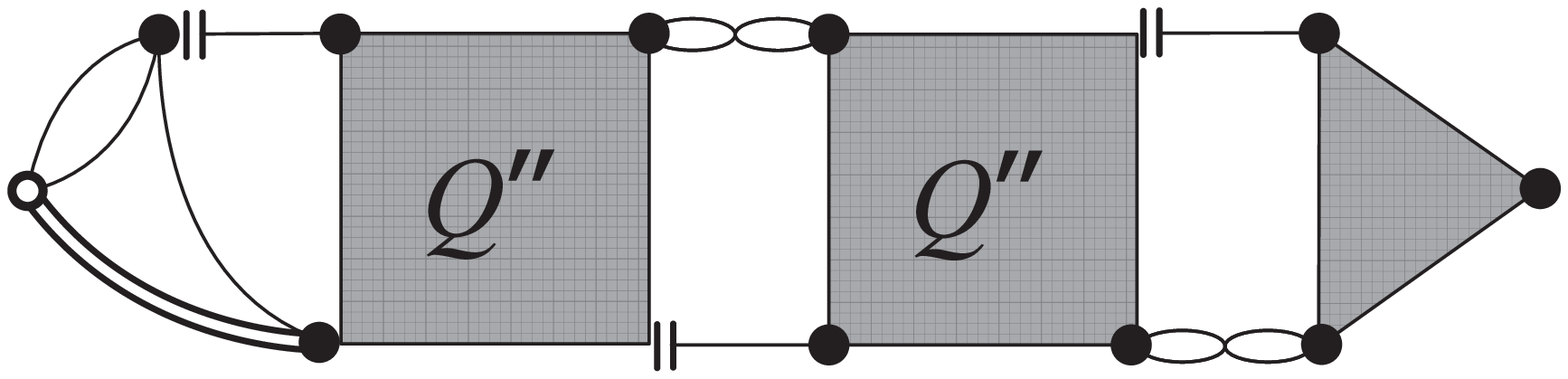}}
     \;+\;\raisebox{-14pt}{\includegraphics[scale=.3]{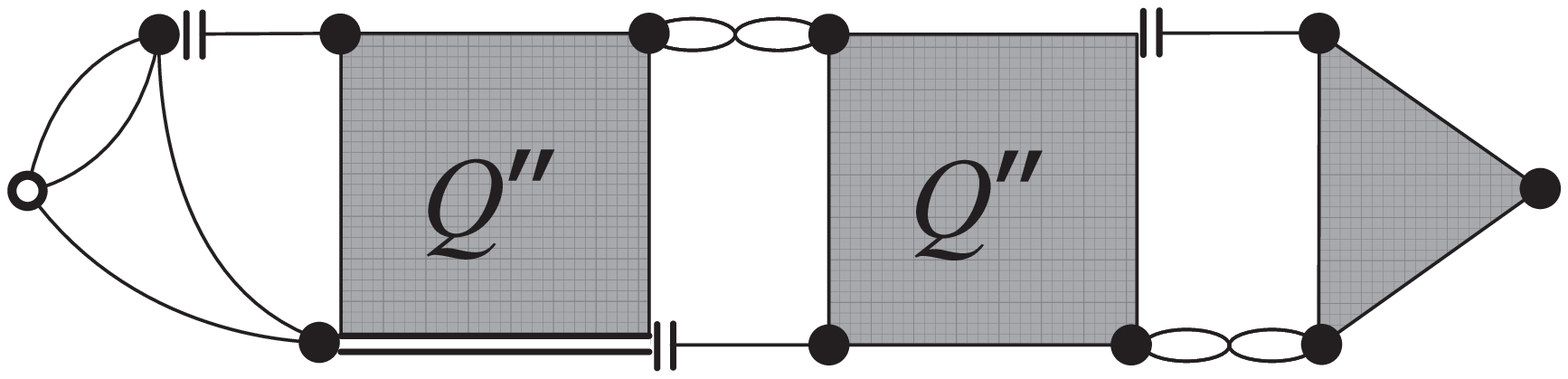}}\nnb
    {}+\,\;&\raisebox{-14pt}{\includegraphics[scale=.3]{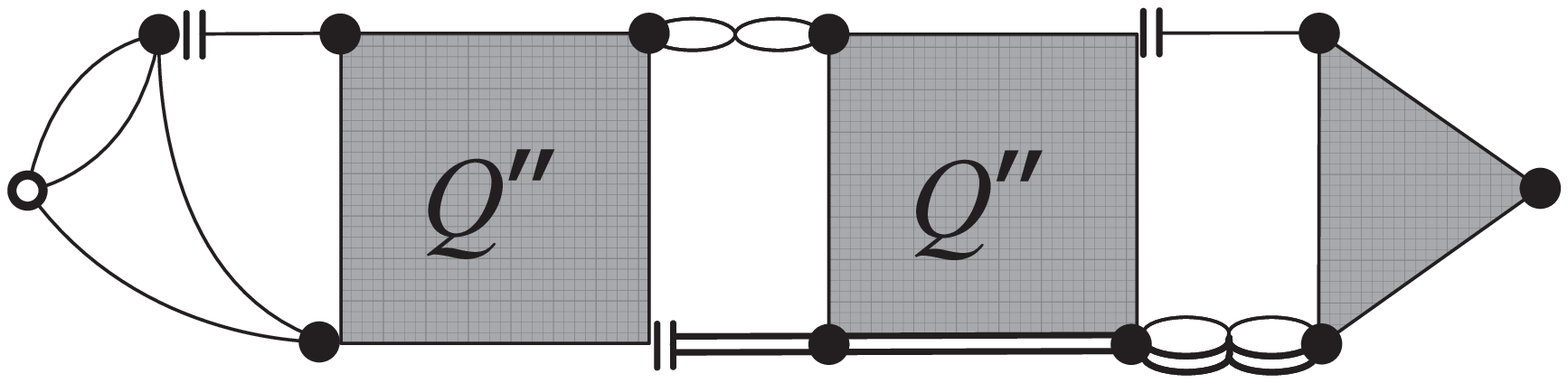}}
     \;+\;\raisebox{-14pt}{\includegraphics[scale=.3]{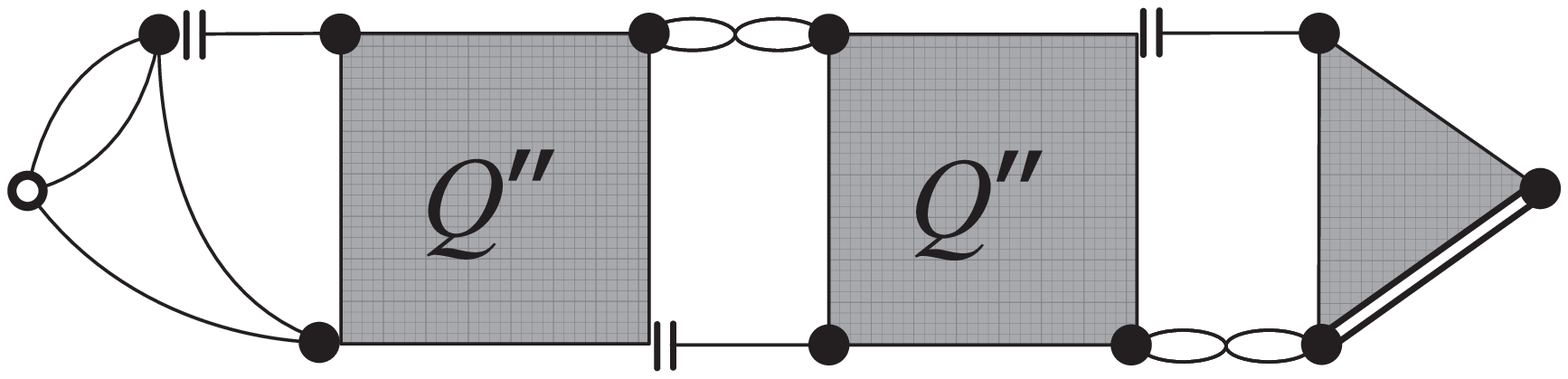}\;\;.}
\end{align}
In the following we refer by (I), (II), (III) and (IV) to the four terms in (\ref{eqCosPiBd6}), respectively. In fact, all 4 terms are bounded by $\Ci\,O(\beta)^N$, as we will show now.

The bound on (IV) is an immediate consequence of (i) and (iii) below (\ref{eqAkira1}), and Claim \ref{claimA}. For the bound on (I), we use translation invariance to obtain the factorization
\begin{equation}
 \underbrace{\raisebox{-14pt}{\includegraphics[scale=.3]{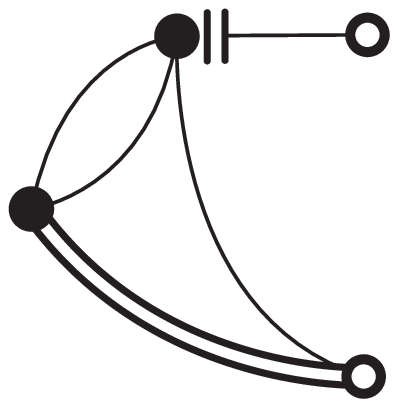}}}_{\text{(I-1)}}\;\;
 \underbrace{\raisebox{-14pt}{\includegraphics[scale=.3]{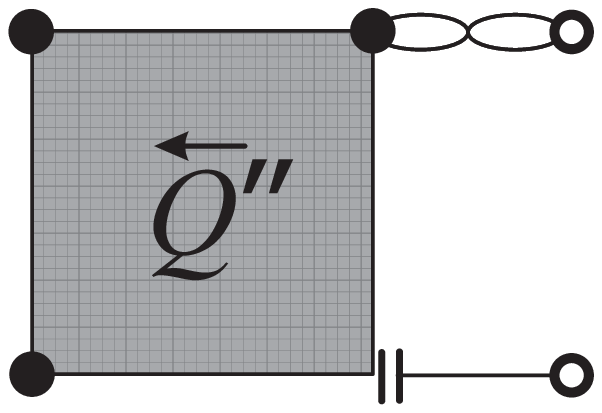}}}_{\text{(I-2)}}\;\;
 \underbrace{\raisebox{-14pt}{\includegraphics[scale=.3]{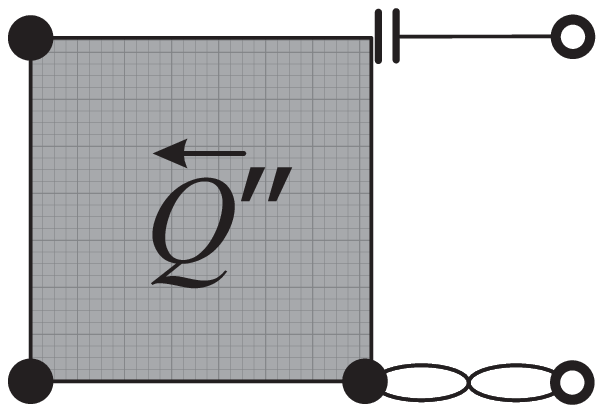}}}_{\text{(I-3)}}\;\;
 \underbrace{\raisebox{-14pt}{\includegraphics[scale=.3]{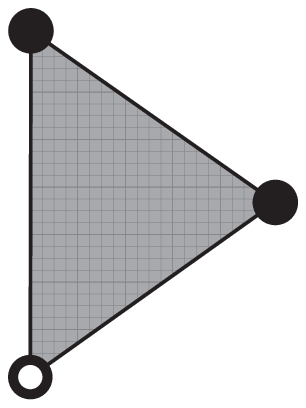}}}_{\text{(I-4)}}
 \raisebox{-14pt}{\;\;.}
\end{equation}
The terms indicated by $\overleftarrow{Q}''$ in the diagram are obtaines from $Q''$ by shifting the two-point functions hanging off the left side of the $Q''$-box to the next factor on the left hand side, i.e.\ (compare with (\ref{eqQppPict}))
\begin{align}
    \overleftarrow{Q}''_{0,v}(y,x)&=
    \sum_{z,z'} P''_{0,v}(y,z)\,\tau D(z'-z)\big(\delta_{z',x}+\tilde G(z',x)\big)\nonumber\\&%
    \quad+\sum_{z,z',w}\tilde G(y,w)\,P'_{0}(w,z)\,\psi(y,v)\,\tau D(z'-z)\big(\delta_{z',x}+\tilde G(z',x)\big)\\
    \nonumber
    &=\raisebox{-14pt}{\includegraphics[scale=.3]{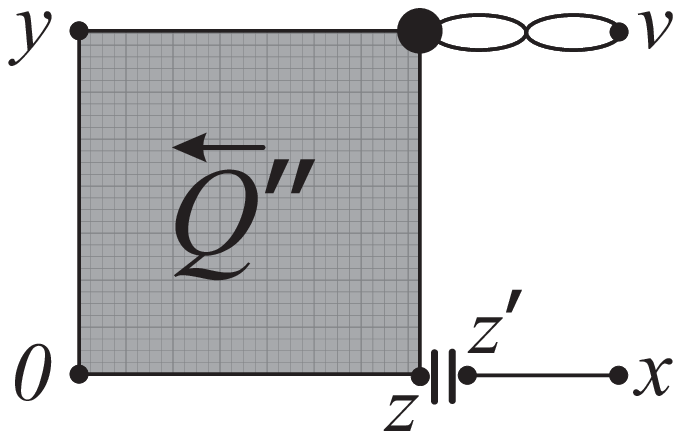}\;\;.}
\end{align}
The first factor (I-1) is bounded by $\Ci O(\beta)$ as in (\ref{eqCosPiBd7}).
The middle terms (I-2) and (I-3) are equal to $\sup_x\sum_{v,y}\overleftarrow{Q}''_{0,v}(y,x)$.
Performing calculations as in (\ref{eqAkira20})--(\ref{eqAkira22}), it can be shown that actually
\begin{equation}\label{eqAkira27}
    \sup_x\sum_{v,y}\overleftarrow{Q}''_{0,v}(y,x)
    \le O(\beta),
\end{equation}
and this term occurs $N-1$ times in (I).
The last term (I-4) is bounded by $O(1)$, cf.\ Claim \ref{lemmaBoundPprime}.
The bounds on (I-1)--(I-4) show that $\text{(I)}\le \Ci O(\beta)^N$.

The terms (II) and (III) are bounded in a similar fashion by product structures:
\begin{equation}
    \text{(II)} \le \raisebox{-14pt}{\includegraphics[scale=.3]{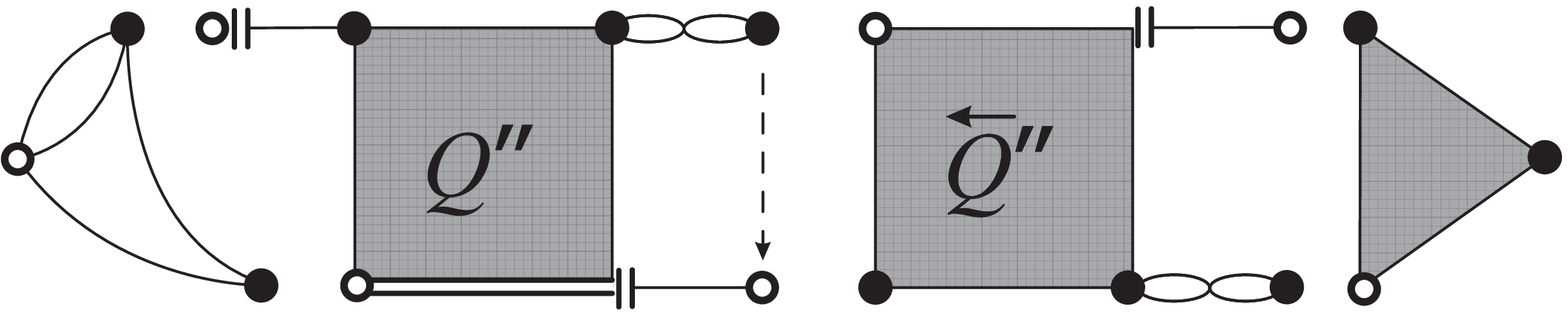}\;\;,}
\end{equation}
\begin{equation}
    \text{(III)} \le \raisebox{-14pt}{\includegraphics[scale=.3]{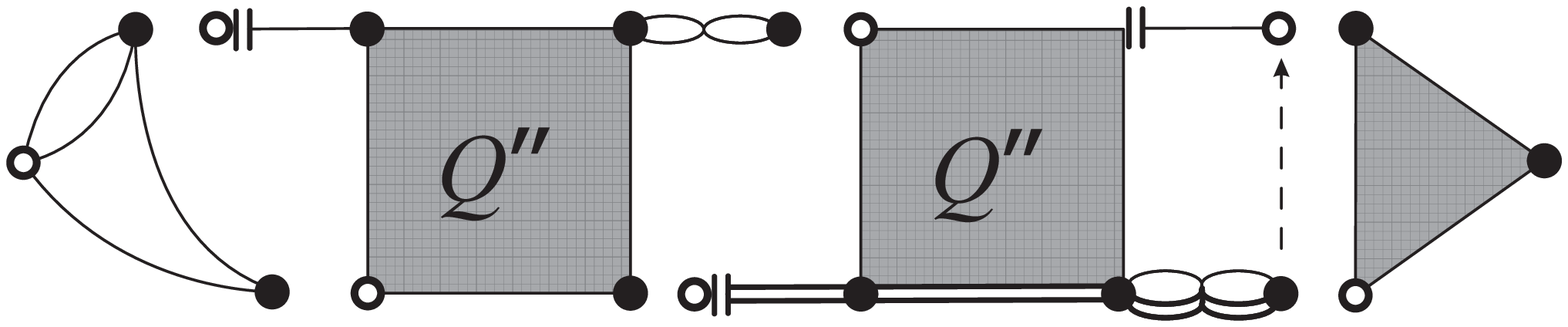}\;\;.}
\end{equation}

The term $\sum_{v,x}P'^{\sss (0)}_v(0,x)$ on the left hand side is bounded by $O(1)$ by (\ref{eqAkira17});
the term $\sum_{u,x}P'_0(u,x)$ (the gray triangle on the right) is bounded by $O(1)$ by Claim \ref{lemmaBoundPprime}.
The terms involving $Q''$ and $\overleftarrow{Q}''$ are bounded by $O(\beta)$ by (\ref{eqAkira20}) and (\ref{eqAkira27}), and together there are $N-2$ of these terms.

It remains to show that
\begin{equation}\label{eqAkira28}
    \raisebox{-14pt}{\includegraphics[scale=.3]{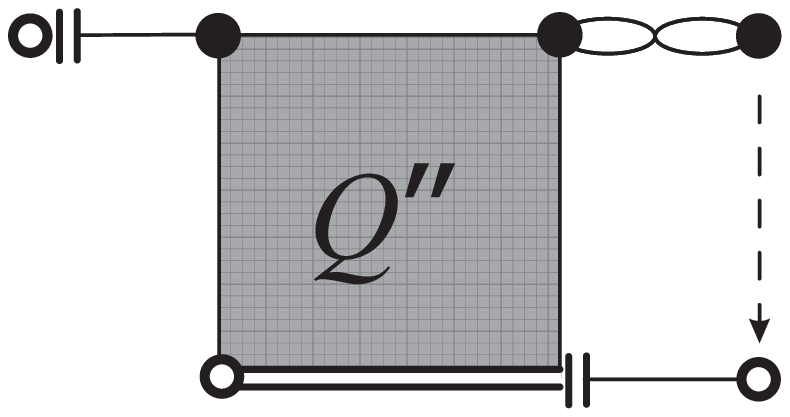}}\le \Ci O(\beta)^2,
    \qquad
    \raisebox{-14pt}{\includegraphics[scale=.3]{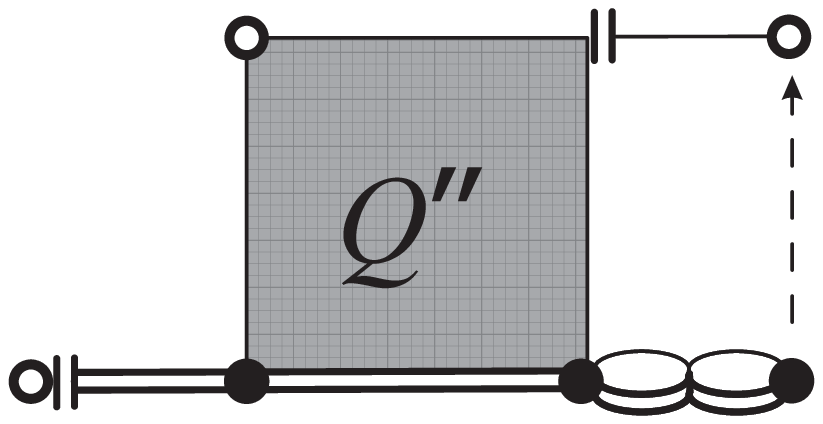}}\le \Ci O(\beta)^2.
\end{equation}
Here the dashed arrow indicates that the supremum is taken over the \emph{difference} between the two vertices at top and bottom of the arrow; see also \citeAkira{(5.46)}.
In order to achieve the bounds in (\ref{eqAkira28}) we proceed as follows.
First we use (\ref{eqSumCosBound}) to distribute the spatial displacement of $1-\cos(k\cdot x)$ to single two-point functions $G$ or $\tilde G$.
Secondly, from each of the emerging summands, we eliminate the term of the form
$\sup_{x,y}[1-\cos(k\cdot (y-x))]G(x,y)$ (where $x$ and $y$ are chosen appropriately),
and bound it by $\Ci O(1)$, cf.\ Lemma \ref{lemmaCosKxGBound}.
Finally, we bound the remaining quantity in the same fashion as in (\ref{eqAkira20})--(\ref{eqAkira22}).
Note that the removed bond is compensated by an extra bond hanging off the lower / upper right corner. The factor $\beta^2$ arises from the bubbles involving the two non-zero two-point functions hanging off the box.
This finally leads to the required bound
\begin{equation}\label{eqAkira37}
\text{(II)}+\text{(III)} \le \Ci O(\beta)^N,
\end{equation}
and thus proves (\ref{eqIsingDiagramResult2}).
This completes the proof of Proposition \ref{propIsingDiagram}.

\vspace{0.5cm}
\noindent {\bf Acknowledgement. }
We thank Roberto Fern\'andez and Aernout van Enter for inspiring discussions.
This work was supported by the Netherlands Organization for Scientific Research (NWO).
MH visited the University of Bath on a grant from the RDSES programme of the European Science Foundation.

\def\cprime{$'$}

\end{document}